\documentclass[a4paper,11pt,english]{amsart}

\usepackage{multirow}
\usepackage[english]{babel}
\usepackage[latin1]{inputenc}
\usepackage[T1]{fontenc}

\usepackage{graphicx}
\usepackage{hyperref}
\usepackage{amssymb}
\usepackage{amsmath}
\usepackage{verbatim,enumerate}
\usepackage{tikz}

\newtheorem{thm}{Theorem}[section]
\newtheorem{lemma}[thm]{Lemma}
\newtheorem{prop}[thm]{Proposition}
\newtheorem{cor}[thm]{Corollary}
\newtheorem{defi}[thm]{Definition}
{\theoremstyle{definition}
\newtheorem{exa}[thm]{Example}
\newtheorem{rem}[thm]{Remark}}

\newcommand{\C}{\mathbb{C}}
\newcommand{\CP}{\mathbb{C}P}

\renewcommand{\P}{\mathbb{P}}

\renewcommand{\epsilon}{\varepsilon}
\newcommand{\R}{\mathbb{R}}
\newcommand{\Z}{\mathbb{Z}}

\newcommand{\D}{\mathcal{D}}
\renewcommand{\SS}{\mathcal{S}}

\newcommand{\x}{\underline{x}}

\newcommand{\X}{\widetilde X}
\newcommand{\N}{\mathcal N}
\newcommand{\CC}{\mathcal C}
\newcommand{\FF}{\mathcal F}
\newcommand{\YY}{\widetilde{\mathcal Y}}
\newcommand{\ZZ}{\widetilde{\mathcal Z}}
\newcommand{\Aut}{\text{Aut}}
\usepackage{wasysym}

\renewcommand{\qedsymbol}{\smiley}

\begin{document}

\title[Floor diagrams  relative to a conic]{Floor diagrams  relative to a conic, and GW-W
  invariants of Del Pezzo surfaces}

\author{Erwan Brugall\'e}
\address{\'Ecole polytechnique,
Centre Math\'ematiques Laurent Schwartz, 91 128 Palaiseau Cedex, France}

\address{Université Pierre et Marie Curie,
4 Place Jussieu, 75 005 Paris, France}
\email{erwan.brugalle@math.cnrs.fr}


\begin{abstract}
We  enumerate, via floor diagrams, complex and real curves in 
 $\C P^2$ blown up in $n$ points on a
  conic. 
As an application, we deduce Gromov-Witten and Welschinger 
invariants of Del Pezzo surfaces. These results are mainly obtained
using Li's degeneration formula and its real counterpart.
\end{abstract}
\maketitle
\tableofcontents

\section{Introduction}

The main question addressed in this paper is 
``Given a Del Pezzo surface $X$, how many
algebraic curves  of a given genus and homology class pass through a
given configuration of points $\x$?''. The
cardinality of $\x$ is always chosen such that the number of curves is finite.
Recall that a Del Pezzo
surface is either isomorphic to $\C P^1\times \C P^1$ or to $\C P^2$
blown up in a generic configuration of  $n\le 8$ points. We
denote by $X_n$ a surface of this latter type.

A possible approach to solve such an enumerative problem 
is to
construct configurations $\x$ for which one can
exhibit
 \emph{all} curves of a given genus and homology class 
passing through $\x$. 
 Such configurations are called \emph{effective}.
The main advantage of effective configurations is to
 provide simultaneous
enumeration of both complex and real curves, furthermore \emph{without}
assuming any invariance with respect to $\x$. This is particularly
useful in real enumerative geometry where 
invariants are lacking.

\medskip
The goal of this paper is to construct effective configurations of
points in Del Pezzo
surfaces, and to compute the corresponding Gromov-Witten and Welschinger
invariants. This is done in two steps.
I first 
enumerate in Theorems \ref{NFD} and \ref{WFD}, via floor diagrams,
 curves passing through an effective configuration of points
in $\C P^2$ blown up at $n$ points located on a
 conic, 
 the resulting complex surface is  denoted by $\X_n$.
Next, by  a suitable degeneration of $X_6$ and $X_7$, the computations
of their enumerative invariants are reduced to enumeration
of curves in  the surfaces $\X_n$ with $n\le 8$. See Theorems
\ref{thm:NFD2}
and \ref{thm:W X6} in the case of $X_6$, and Theorems \ref{thm:GWX7}
and \ref{thm:WX7} in the case of $X_7$. 
The first degeneration is
classical: one degenerates $X_6$ into the union of $\X_6$ and 
$\C P^1\times \C P^1$, which
 basically corresponds to degenerating $X_6$ to a nodal Del
Pezzo surface. Enumerative invariants of $X_6$ are then computed by
enumerating curves on $\X_6$ thanks to the
Abramovich-Bertram-Vakil formula \cite{Vak2} 
and its real versions \cite{Br20,BP14}.
By blowing up an additional section of the previous degeneration of $X_6$,
one produces a degeneration of $X_7$ into the union of $\X_6$ and
$\X_2$. A very important feature of these two degenerations is that
\emph{no} multiple covers appear.
As a consequence, this method extends 
to 
 the case 
of $X_8$, see Theorems \ref{thm:GWX8}
and \ref{thm:WX8}. 
 By blowing up an additional section, one 
degenerates $X_8$ to the
union of $\X_{6,1}$ and $\X_2$, where $\X_{6,1}$ denotes the blow up
of $\C P^2$ at seven points, six of them lying on a conic.
Some non-trivial ramified coverings might appear during this
degeneration, however all of them are regular and can be treated 
using results from \cite{Shu13}. 
To the best of my knowledge, this is the first
explicit computation of Gromov-Witten invariants in any genus 
of $X_8$ (see \cite{CapHar1,Vak2,Shu13} for similar computations in
other Del Pezzo surfaces). 

Another nice property
 of  effective configurations is that,
in addition to allowing computations of enumerative
invariants, they often bring out some of their  qualitative properties. 
Results about the sign  of Welschinger invariants, their sharpness,
their arithmetical properties, their vanishing, and 
comparison of real and complex invariants were for example
previously obtained in this way in
\cite{IKS1,IKS2,IKS3,IKS10,IKS11,IKS13,Wel4,Br20,BP14}. Several
extensions of those results are deduced from 
the methods
presented here, see Corollaries \ref{cor:X6 1}, \ref{cor:X6 2},
\ref{cor:X8 dec}, \ref{cor:X7 van},
\ref{cor:X8 positive}, \ref{cor:X8 congruence}, \ref{cor:X8 van}, and Proposition
\ref{prop:X7 sign}.

\medskip
Among the available techniques to construct   effective
configurations, one can cite methods based on 
\emph{Tropical geometry}  \cite{Mik1}, 
and  on the degeneration of the target space $X$, such as
 \emph{Li's degeneration
 formula} \cite{Li02,Li04} in the algebraic setting, \emph{symplectic sum
 formulas} in the symplectic setting \cite{IP00,LiRu01,TehZin14}, or
 more generally  \emph{symplectic field theory} \cite{EGH}.
The results
 of this paper are obtained by degenerating the target space.
As a rough outline,  these methods consist of  degenerating
the ambient space $X$ into a union $\bigcup_i Y_i$ 
of ``simpler'' spaces $Y_i$, and 
to recover enumerative
invariants of $X$ out of those of the $Y_i$'s. 
Note that to achieve this second step, one has to consider
Gromov-Witten invariants of the surfaces $Y_i$ \emph{relative} to the
divisors $E_{i,j}=Y_i\cap Y_j$. 
In other words, one has to enumerate curves
satisfying
 some
  incidence conditions \emph{and} intersecting the divisors $E_{i,j}$ in
some prescribed way.
A very practical feature of 
these degeneration methods is that, in
nice cases, including the absence of ramified coverings, 
deformations of a curve in $\bigcup_i Y_i$ to a curve in $X$
 only depend on 
the intersections of the curve with 
the divisors $E_{i,j}$. In particular,
if one knows how to construct effective configurations in
the surfaces $Y_i$, one can construct effective configurations in $X$.
Since I am interested  here in the computation of 
 enumerative invariants of Del Pezzo surfaces, I made the choice
to work in the algebraic category, and to use Li's degeneration
formula. Nevertheless the whole paper should  be easily translated  in
the symplectic setting using symplectic sum formulas.

Using this general strategy,  it  usually remains  a non-trivial
task to find a 
suitable degeneration $\bigcup_i Y_i$ of a particular variety $X$, 
from which one can deduce 
effective configurations in $X$. 
The \emph{floor decomposition technique}, elaborated in collaboration
with Mikhalkin  \cite{Br7,Br6b}, provides in some cases such a
useful degeneration. The starting observation is that configurations
containing at most two points in a Hirzebruch surface
(i.e. holomorphic $\C
P^1$-bundle over $\C P^1$) are effective.
Then the strategy is to choose a suitable 
rational curve $E$ in $X$,  to degenerate $X$ into the union of
$X$ and a chain of copies of $\P(\mathcal N_{E/X}\oplus\C)$, and to
choose a configuration of at most two points in each of these copies. 
In lucky situations, the union of all those points can be deformed into 
 an effective configuration $\x$ in $X$. When this is the case,  
all complex and real  
curves passing through $\x$ can be encoded into purely combinatorial objects
called
 \emph{floor diagrams}. A more detailed outline of this
 technique together with its relation to
Caporaso and Harris approach
is given in Section \ref{sec:CH FD}.

\medskip
\emph{Personne n'est jamais assez fort pour ce calcul}.
Guided by this french adage, I illustrated the general theorems
\ref{thm:NFD2}, \ref{thm:W X6}, \ref{thm:GWX7}, \ref{thm:WX7},
\ref{thm:GWX8}, and  \ref{thm:WX8} 
by explicit examples including detailed computations.
I usually find it very useful,  as a reader as well as  an author, 
that a paper provides details in passing from 
 the general theory to particular examples.
This is specifically the case in
   enumerative geometry, where 
 formulas, sometimes abstruse at first sight, 
   may hinder the reader's understanding of  the
  geometrical phenomenons they describe.
Moreover, working out concrete examples in full details is an efficient way to
check the consistency of general theorems, and that no subtility
escaped the notice.
I hope that the detailed computations given here will help the
 reader to acquire a concrete feeling of the general
and sometimes subtle phenomenons coming into play.

In the same range of ideas, I chose to dedicate a separate section to
each of the surfaces $X_6$, $X_7$, and $X_8$,
 despite the fact that 
 Section \ref{sec:X6} is formally
contained in Section \ref{sec:X7}, which in
turn is partially
contained in Section \ref{sec:X8}. Indeed, the
combinatorics becomes more involved as the number of blown-up
points increases, 
 and reducing results about $X_n$
to $X_{n-1}$ still requires
some work. By giving a specialized formula in each 
case, I hope to make concrete computations accessible to the reader.

\medskip
The plan of the paper is the following. In the remaining part of the
introduction, I explain the basic ideas underlying the floor decomposition
technique, relate the results presented here with other works, and
settle the notations and convention used throughout this paper.
Complex and real 
enumerative problems  considered in this text are defined in
Section \ref{sec:defi enum}, which also contain a few  elementary
computations.   
Floor diagrams and their relation with effective
configurations of points in $\X_n$ is given in Section \ref{sec:FD}. 
This immediately applies to compute absolute
invariants of $X_6$, which is done in Section \ref{sec:X6}.
 Section
\ref{sec:proof} is devoted to the proof of Theorems \ref{NFD} and \ref{WFD}. 
In  Section
\ref{sec:X7}, I
reduce the enumeration of curves in $X_7$ to
enumeration of curves in $\X_{8}$. The reduction of enumerative
problems of 
$X_8$ to enumerative problems in $\X_{8,1}$ is proved in Section \ref{sec:X8}.
Finally, this paper ends  in Section \ref{sec:conclusion}
with some comments and possible generalizations of
the material presented here.

\subsection{Floor  diagrams and their relation to 
Caporaso-Harris type formulas}\label{sec:CH FD}
For the sake of simplicity, I restrict  to the problem of 
counting curves of a given genus, realizing a given homology class
in
$H_2(X;\Z)$, and passing through a  generic configuration $\x$ of
points on a (maybe singular)
complex algebraic surface $X$.
Recall that the
cardinality 
of $\x$ is such that the number  of curves is
finite.

The paradigm underlying a Caporaso-Harris type formula is the
following. Choose a  suitable  irreducible curve $E$  in $X$, and  specialize 
points in $\x$ one  after the other to $E$. 
After the specialization of
sufficiently many points,  one  expects  that  curves under consideration
degenerate into
reducible curves having $E$ as a component.  By forgetting this
component,
  one is reduced to an enumerative problem in $X$ concerning
curves realizing a ``smaller'' homology class. 
With a certain amount of optimism, one can 
then hope to solve the initial problem by induction.

This method has been first proposed and successfully applied 
 by Caporaso and Harris \cite{CapHar1}
in the case of $\C P^2$ together with a line, and has been since then
applied in several other situations. Directly related to this paper, 
one can cite the work of Vakil  in the case 
of $\X_6$ together with the strict transform of the conic
\cite{Vak2}, and its
generalization  by Shoval and Shustin \cite{Shu13}
to the case of
$\X_{n,1}$.
As a very nice fact, it turned out that this approach also provided
a way to compute  certain Welschinger invariants 
for configuration $\x$  only composed of real points 
\cite{IKS3,IKS10,IKS11,IKS13}.

\medskip
When $X$ and $E$ are smooth, Ionel and Parker observed in {\cite[Section 5]{IP98}} that
the method proposed by Caporaso and Harris 
could be interpreted in terms of
degeneration of the target space $X$.
I present below the 
algebro-geometric 
version of this interpretation
{\cite[Section 11]{Li04}}. 
The  ideas underlying  symplectic interpretation are similar, however the
 two  formalisms are quite different.
 I particularly refer to {\cite{Li04}}
for an introduction 
to this degeneration technique in enumerative geometry.
Given $X$ and $E$ as above,  denote by $\N_E=\P(\N_{E/X}\oplus
\C)$, and
 do the following:

\begin{enumerate}
\item   degenerate $X$ into a reducible surface $Y=X \cup \N_E$, and 
specialize exactly one point in $\N_E$  during this degeneration;

\item determine all possible degenerations in $Y$ of the enumerated curve; 

\item  for each such limit curve in $Y$,  compute the  number of 
curves  of which it is the limit.  
\end{enumerate}

This method produces recursive formulas à la Caporaso-Harris if all
limit curves in $Y$ can be recovered by solving separate enumerative
problems in its components $X$ and $\N_E$.

\medskip
The idea behind floor diagrams is to get rid of any recursion, 
which
implicitly 
refers to some invariance property of the enumerative problem under
consideration. 
 To do so, one 
considers a single degeneration
of $X$ into the union $Y_{max}$ of $X$ and a
 chain of copies of $\N_E$, and  specializes exactly one element of $\x$
 to each copy of $\N_E$. 
Floor diagrams then  correspond to  dual graphs of the
limit curves in $Y_{max}$, and the way they meet the points in $\x$
is encoded in a \emph{marking}.
In good situations, all limit curves in $Y_{max}$ can be
completely 
recovered only from the combinatorics of the marked floor
diagrams. In 
particular, effective configurations in $X$ can be deduced from effective
configurations in $\N_E$.

This  method  have been first successfully applied in collaboration
with  Mikhalkin  in \cite{Br7,Br6b},
in the case  when $X$ is a toric surface and $E$
is a toric divisor satisfying some \emph{$h$-transversality}
condition.
We used methods from tropical geometry, which in particular 
allowed us to get
rid of the  smoothness assumption on $X$ and $E$ required in
 Li's degeneration formula.
Note that when both floor diagrams and  Caporaso-Harris type formulas
are available, it follows  from the above description that these two
methods
 provide two different, although equivalent, ways of clustering
 curves under enumeration. 
Passing from one presentation to the other 
does not present any difficulty \cite{Br8}.

When both $X$ and $E$ are chosen to be real, 
floor diagrams   can also be adapted  to enumerate
real curves passing through a
real configurations of  $r$ real
points and $s$ pairs of complex conjugated points:  
\begin{enumerate}
\item[(1')]  degenerate  $X$ to the union $Y_{max}'$ of
$X$ and a chain of $r+s$ copies of
$\N_E$, specializing
  exactly one real point or one pair of complex conjugated points
 of $\x$ to each copy of $\N_E$;

\item[(2')] determine real curves in step $(2)$ above;

\item[(3')] adapt computations of step $(3)$ above to 
 determine real curves converging to a given
 real  limit curve.
\end{enumerate}
As in the complex situation, one can  associate floor diagrams
to  real limit curves in $Y_{max}'$, each of them being now naturally  
  equipped with an involution induced by
 the real structure of $X$.
Again in many situations, all
necessary information about enumeration of real curves in $Y_{max}'$ are
encoded by the
combinatorics of these \emph{real marked floor diagrams}. 
In the case of toric
surfaces equipped with their tautological real structure, 
this has been done in \cite{Br6b} under the $h$-transversality assumption. 
The present paper shows that this is also
the case when $X=\X_n$.
This reduction of an algebraic problem to a purely
 combinatorial question 
might
not seem so surprising when all points in $\x$ are real, since then the
situation is  similar to the complex one. However in the presence of
complex conjugated points,  I am still 
puzzled by the many
 cancellations  that allow this reduction.

\medskip
The floor diagram technique clearly takes advantage over the Caporaso-Harris
method  when one wants to count real curves passing
through general real configuration of points. 
In the enumeration of complex curves, or of real curves interpolating
configurations of real points, the use of any of these two  methods
 is certainly a matter of taste.
From my own experience, I could notice that floor diagrams provide a
more geometrical picture of   curve degenerations 
 which helps sometimes to minimize mistakes in practical computations. 
Finally, it is worth stressing that  floor diagrams also led to 
the discovery of new phenomenons also in complex enumerative geometry,
for example concerning the (piecewise-)polynomial behavior of Gromov-Witten 
invariants, e.g.
 \cite{FM,Blo11,ArdBlo,LiuOsse14,BA13}.

\subsection{Related works}
Higher dimensional versions exist of 
both Caporaso-Harris \cite{Vak1,Vak06} and floor diagram techniques
\cite{Br7,Br6}. 

Tropical geometry \cite{Mik1} provides also a powerful tool to construct
effective configurations. Historically it provided  the first
computations of Welschinger invariants of toric Del Pezzo surfaces.

Methods and results from \cite{Wel4} constituted the main
source of  motivation for me to study floor diagrams relative to a
conic. In this paper
 Welschinger uses symplectic field theory 
to decompose a real symplectic manifold into the disjoint
union of the complement
of 
a connected component of its real part on one hand, 
with the cotangent bundle of
this component on the other hand.
To the best of my knowledge, \cite{Wel4} is the first 
explicit use of 
degeneration of the target space in the framework of 
real enumerative geometry. 
Similar results using symplectic sums were also obtained in \cite{Teh10}.

In collaboration with Puignau, we 
also used symplectic field theory   in
\cite{Br20,BP14}  to 
provide relations among
Welschinger invariants of a given 4-symplectic manifold
 with
possibly different real structures,  and to obtain
 vanishing results.

As mentioned above, Itenberg, Kharlamov, and Shustin used the
Caporaso-Harris approach to study Welschinger invariants in the case
of configurations of real points. In a series of paper
\cite{IKS10,IKS11,IKS13},  they  thoroughly studied the case of 
all real structures on  Del Pezzo surfaces of degree at least two. 
Due to methods presenting some similarities, the present paper
and \cite{IKS11,IKS13} contain some results in common, nevertheless obtained
independently and more or less simultaneously.

Another treatment of  
effective configurations
has been proposed in 
\cite{CooPan12}.

A real version of the WDVV equations for rational $4$-symplectic
manifolds have been proposed by Solomon \cite{Sol1}. 
Those equations  
provide many relations among Welschinger invariants of a given
real $4$-symplectic
manifold, that hopefully reduce the computation of all 
invariants to the computation of finitely many simple cases. 
This program has
been completed  in \cite{HorSol12} in the case of  rational surfaces
equipped with a standard real structure, i.e. induced by the standard
real structure on $\C P^2$ via the blowing up
map. 
In a work in progress
in collaboration with
Solomon
\cite{BruSol12}, we combine symplectic field theory and 
real  WDVV equations to cover the case of all remaining
real rational algebraic
surfaces.  At the
time I am writing these lines, this project has been completed 
for all minimal real rational algebraic surfaces, except for the
minimal Del
Pezzo surface of degree 1. 
As a side remark, I  would like to stress that 
if real WDVV
equations are definitely better from a computational point of view
than floor diagrams, it seems nevertheless
very difficult to extract from them qualitative information about Welschinger
invariants.

A real WDVV equation in the case of odd dimensional
projective spaces has also been proposed by Georgieva and Zinger 
\cite{GeoZin13}.

\subsection{Conventions and notations}
\subsubsection{}
A real algebraic variety $(X,c)$ is a complex algebraic variety $X$
equipped with a antiholomorphic involution $c:X\to X$.
The real part of $(X,c)$, denoted by $\R X$, is by definition the set
of points of $X$ fixed by $c$. When the real structure $c$ is clear
from the context, I sometimes 
use the notation $\overline p$ instead of $c(p)$.

The complex projective space $\C P^N$ is always considered equipped
with its standard real structure given by the complex conjugation.

I assume that the reader has some acquaintance with the classification
of real rational algebraic surfaces. For some refreshment on the
subject, I recommend \cite{Kol1,DK}.

\subsubsection{}
The connected sum of $1+k$ copies of $\R P^2$ is denoted by $\R P^2_k$.

\subsubsection{}
The blow up of $\C P^2$ at $n$
points in general position is denoted by $X_n$.
The blow up of $\C P^2$ at $n$
points lying on a smooth conic $E$ is denoted by $\X_n$.
The blow up of $\C P^2$ at $n+1$ points, exactly $n$ of them 
 lying on a smooth conic $E$, is denoted by $\X_{n,1}$. Surfaces
 $\X_{n,1}$ will only appear in Section \ref{sec:X8}.
The strict transform of $E$ in  $\X_n$ or $\X_{n,1}$ is still denoted
by $E$.
In particular if $n\le 5$, the surface $\X_n$ denotes
 the surface $X_n$ together with the distinguished  curve  $E$.

The normal bundle of $E$ in $\X_n$ is denoted by $\N_{E/\X_n}$, and I
will use the notation $\N=\mathbb P(\N_{E/\X_n}\oplus \C)$ throughout
the text. 
The surface $\N$ contains two distinguished non-intersecting rational curves
$E_\infty=\P (\N_{E/\X_n}\oplus \{0\})$ and $E_0=\P (E\oplus \{1\})$.
Moreover the line bundle $\N_{E/\X_n}$ induces
a canonical $\C P^1$-bundle $\pi_E:\N\to
E_\infty$.

Suppose in addition that $E$ is
a smooth real conic in $\C P^2$ and that $\X_n$ is obtained by blowing
up  $n-2\kappa$ points on $\R E$ and $\kappa$ pairs of complex conjugated
points on $E$. 
The real structure on $\X_n$ induced by the real structure on $\C P^2$
via the blow up map is denoted by $\X_n(\kappa)$. In particular  
$\R \X_n(\kappa)=\R P^2_{n-2\kappa}$.
If $n=2\kappa$, the connected component of $\R \X_n(\kappa)\setminus
\R E$ with Euler characteristic $\epsilon\in\{0,1\}$ is denote by
$\widetilde L_\epsilon$.

\subsubsection{}
All invariants considered in this text do not  depend on the
deformation class of the complex or real algebraic surface under
consideration, see \cite{IKS14}. Consequently, the surfaces $X_n$ and the pairs
$(\X_n,E)$ are
 always implicitly considered up to deformation.

\subsubsection{}
The class realized in  $H_2(X;\Z)$ by 
 an algebraic curve $C$ in a complex algebraic surface $X$ is denoted by
$[C]$.

\subsubsection{}
The image $f(C)$ of an  algebraic map $f:C\to X$ denotes its scheme
theoretic image, i.e. irreducible components of $f(C)$ are
considered with multiplicities. If $Y\subset X$ is a divisor
intersecting $f(C)$ in finitely many points,   the pull
back of $Y$ to $C$ is denoted by $f^*(Y)$.

An isomorphism between two algebraic maps $f_1:C_1\to X$  and
$f_2:C_2\to X$ is an isomorphism $\phi:C_1\to C_2$ such that 
$f_1=f_2\circ\phi$.
Maps are always considered up to isomorphisms. 
The group of automorphisms of a map $f$ is denoted
by $Aut(f)$.

\subsubsection{}

If $X$ is a  complex algebraic surface, 
the intersection product of two elements
$d_1,d_2\in H_2(X;\Z)$ is
denoted by $d_1\cdot d_2\in \Z$.

\subsubsection{}

Given a vector $\alpha=(\alpha_i)_{1\le i\le\infty}  \in\Z_{\ge
  0}^\infty$, I use the notation
$$|\alpha|=\sum_{i=1}^\infty \alpha_i,\quad I\alpha=
\sum_{i=1}^\infty i\alpha_i,\quad \mbox{and}\quad I^\alpha=
\prod_{i=1}^\infty i^{\alpha_i}.$$
The vector in
$\Z_{\ge 0}^\infty$ whose all coordinates are equal to $0$, except the
$i$th one which is equal to 1, is denoted by $u_i$.

\subsubsection{}
The sets of vertices and edges of a finite graph $\Gamma$ are
respectively denoted by $Vert(\Gamma)$ and $Edge(\Gamma)$. 
If  $\Gamma$ is oriented, it is said to be 
{\em acyclic} if it does not contain any non-trivial oriented
cycle. Its set of
sources (i.e. vertices such that all their adjacent edges are
outgoing) is denoted by
$Vert^\infty(\Gamma)$, and $Edge^\infty(\Gamma)$ denotes 
the set of edges adjacent
to a source.

\subsection*{Acknowledgment}
I started to think about floor diagrams relative to a conic during the
fall 2009 program \emph{Tropical Geometry} held at MSRI in Berkeley. 
Significant progress have been made in fall 2011
during my stay  at
IMPA in Rio de 
Janeiro, and a consequent part of this manuscript has been written
during the program \emph{Tropical geometry in its complex and
  symplectic aspects} held in spring 2014
at CIB in Lausanne. I would like to thank
these three organizations for excellent working conditions, as well as the two institutions I have been affiliated to
during this period, Université Pierre et Marie Curie and École Polytechnique.

I am also grateful to Benoît Bertrand, Nicolas Puignau, and Kristin
Shaw for their comments on a preliminary version of the text.

\section{Enumeration of curves}\label{sec:defi enum}

\subsection{Absolute invariants of Del Pezzo surfaces}\label{sec:GWWXn}

Recall that $X_n$ 
denotes $\C P^2$ blown up in a generic
configuration of $n$ points. The group $H_2(X_n;\Z)$ is the free
abelian group generated by $[D],[E_1],\ldots ,[E_n]$ where
$E_1,\ldots ,E_n$ are the 
 exceptional curves of the $n$ blow-ups, and $D$ is 
the strict transform of a  line not passing through any of 
those $n$ points. The first Chern class of $X_n$ is given by
$$c_1(X_n)=3[D] -\sum_{i=1}^n[E_i]. $$
Given $n\le 8$ and $d\in H_2(X_n;\Z)$, the number of complex algebraic curves
of genus $g$, realizing the class $d$,
 and passing through a generic configuration $\x$ of
$c_1(X_n)\cdot d -1+g$ points in $X_n$ is finite and does not depend on
 $\x$ {\cite[Section 4.3]{Vak2}}. 
We denote this number, known as a \emph{Gromov-Witten
   invariant} of $X_n$, by $GW_{X_n}(d,g)$. 

Suppose now that $X_n$ is endowed with a real structure $c$. Then one may
consider \emph{real} configurations of points $\x$, i.e. satisfying
$c(\x)=\x$, and count real algebraic curves.
In this case, the number of such curves 
usually  heavily depends on the choice of $\x$. Nevertheless, in the
case when $g=0$, Welschinger \cite{Wel1b,Wel1} proposed a way to associate a
sign to each real curve 
so that counting them with this sign produces an invariant.

More precisely, let $L$ be a connected component of $\R X_n$, and
choose a decomposition $c_1(X_n)\cdot d -1=r+2s$ with $r,s\in\Z_{\ge 0}$.
Let $\x$ be a generic
real configuration of $c_1(X_n)\cdot d -1$ points in
$X_n$ such that $|\R\x|=r$ and  $\R\x\subset L$. 
We denote by $\R\CC_L(d,\x)$ the set of real rational algebraic curves $C$
in
$X_n$ realizing the class $d$,  passing through all
points in $\x$, and\footnote{This latter condition is empty as soon as
  $r\ge 1$.} such that $|\R C\cap L|=+\infty$.
Recall that a \emph{solitary node} $p\in\R X$ of a real algebraic curve $C$ in a real
algebraic surface $(X,c)$ is the transverse intersection of two smooth
$c$-conjugated branches of $C$.
To each curve $C\in\R\CC_L(d,\x)$, we associate two
\emph{masses}\footnote{The mass $m_{\R X}(C)$ is the quantity originally
  considered by Welschinger in \cite{Wel1b}. 
Itenberg, Kharlamov, and Shustin observed in \cite{IKS11} that
Welschinger's proof actually leaves room for different choices in the mass
one could associate to a real algebraic curve in order to get an invariant.
I decided to restrict here to these two particular masses since they
seem to be the most meaningful ones according to \cite{Wel4,BP14}.}:
\begin{enumerate}
\item $m_{\R Xn}(C)$ is the total number of solitary nodes of $C$; 
\item $m_L(C)$ is the number of solitary nodes of $C$ contained in $L$.
\end{enumerate}
Given $n\le 8$ and $L'=\R X_n$ or $L'=L$, the  number
$$W_{(X_n,c),L,L'}(d,s)=\sum_{C\in \R\CC_L(d,\x)}(-1)^{m_{L'}(C)} $$
does not depend on $\x$ as long as $|\R\x|=r$, neither on the
deformation class of $(X_n,c)$
\cite{Wel1b,Wel1,IKS11}. 
Those numbers are known
as \emph{Welschinger invariants} of $(X_n,c)$.
When $\R X_n$ is connected, we use the shorter notation $W_{(X_n,c)}(d,s)$
instead of $W_{(X_n,c),\R X_n,\R X_n}(d,s)$

\subsection{Relative invariants of $\X_n$}\label{sec:relative Xn}
Recall that $\X_n$ denotes $\C P^2$ blown up in $n$ distinct points on
a conic $E$.
 Again, we denote by $E_1,\ldots,E_n$ the exceptional
divisors of
$n$  blow ups, and by $D$ the strict transform  of a line not passing
through any of these $n$ points. 
 The group $H_2(\X_n;\Z)$ is the free abelian group  generated by
$[D],[E_1],\ldots,[E_n]$, and we have
$$c_1(\X_n)=3[D] -\sum_{i=1}^n[E_i] \quad \mbox{and}\quad [E]^2=4-n. $$
To define  Gromov-Witten invariants of $\X_n$ \emph{relative} to
the curve $E$, it is more convenient to consider maps $f:C\to\X_n$
rather that algebraic curves $C\subset\X_n$, because of the appearance
of non-trivial ramified coverings.
Note that in the definition of Gromov-Witten invariants of
$X_n$ with $n\le 8$ given in Section \ref{sec:GWWXn}, all 
curves under consideration in $X_n$ are reduced,  hence it makes no difference to
consider immersed or parametrized curves.

\medskip
Let $d\in H_2(\X_n;\Z)$ and 
$\alpha,\beta\in \Z_{\ge 0}^\infty$
such that
$$I\alpha + I\beta=d\cdot [E].$$
Choose a configuration $\x=\x^\circ\sqcup \x_E$ of points in $\X_n$,
with $\x^\circ$ a configuration of $d\cdot [D]-1+g  +  |\beta|$ points
in $\X_n\setminus E$, and $\x_E=\{p_{i,j}\}_{0\le j\le \alpha_i, i\ge 1}$ a configuration of
$|\alpha|$ points in $E$.
Let $\CC^{\alpha,\beta}(d,g,\x)$ be the set
 of holomorphic maps $f:C\to \X_n$ such that 
\begin{itemize}
\item $C$ is a connected  algebraic curve of arithmetic genus $g$;
\item $f(C)$ realizes the homology class $d$ in $\X_n$;
\item $\x\subset f(C)$;
\item $E$ is not a component of $f(C)$;
\item $f^*(E)=\sum_{i\ge 1} \sum_{j=1}^{\alpha_i}iq_{i,j} + 
\sum_{i\ge 1} \sum_{j=1}^{\beta_i}i\tilde q_{i,j}$, with $f(q_{i,j})=p_{i,j}$.
\end{itemize}
The Gromov-Witten invariant $GW^{\alpha,\beta}_{\X_n}(d,g)$ relative
to $E$ is defined as
$$GW^{\alpha,\beta}_{\X_n}(d,g)=
\sum_{f\in
  \CC^{\alpha,\beta}(d,g,\x)} \mu(f,\x^\circ)$$
for a generic choice of $\x$, where
 $$\mu(f,\x^\circ)=
\frac{1}{|\Aut(f)|}\prod_{p\in\x^\circ}|f^{-1}(p)|.$$
When $\alpha=0$ and $\beta=(d\cdot [E])u_1$, I use the shorter notation 
$GW_{\X_n}(d,g)$. 
Define  also
$$\CC_*^{\alpha,\beta}(d,g,\x)=\left\{f(C) \ | \ (f:C\to\X_n)\in
\CC^{\alpha,\beta}(d,g,\x)\right\}.$$  

Next proposition is a particular case of 
 {\cite[Proposition 2.1]{Shu13}}.

\begin{prop}[{\cite[Proposition 2.1]{Shu13}}]\label{prop:shoshu}
For a generic configuration  $\x$, 
the set
$\CC_*^{\alpha,\beta}(d,g,\x)$ is finite, and
its 
cardinal does not depend on $\x$. 
Moreover, if $d\ne l[E_i]$ with $l\ge 2$,
then   the map $\CC^{\alpha,\beta}(d,g,\x)\to
\CC_*^{\alpha,\beta}(d,g,\x)$ is 
one-to-one, and any element
 $f:C\to\X_n$ of
  $\CC^{\alpha,\beta}(d,g,\x)$  satisfies the following
properties: 
\begin{itemize}
\item the curve $C$ is smooth and irreducible;
\item $f$ is an immersion, birational onto its image (in particular it
  has no non-trivial automorphism);
\item $f(C)$ intersects the curve $E$ at non-singular points.
\end{itemize}
\end{prop}

\begin{rem}
Note that if $l\ge 2$, the set $\CC^{\alpha,\beta}(l[E_i],g,\emptyset)$
might  not be finite, however 
$\CC_*^{\alpha,\beta}(l[E_i],g,\x)$ is either empty or consists
of  the curve $E_i$ with multiplicity $l$. 
\end{rem}

\begin{prop}[{\cite[Proposition 2.5]{Shu13}}]\label{prop:initial values}
Suppose that
$d\cdot [D] -1+g
+|\beta|=0$. Then the number $GW^{\alpha,\beta}_{\X_n}(d,g)$ is
non-zero only in the following cases:
$$GW^{0,u_1}_{\X_n}([E_i],0)=GW^{2u_1,0}_{\X_n}([D],0)=
GW^{u_2,0}_{\X_n}([D],0)=1,$$
$$GW^{u_1,0}_{\X_n}([D]-[E_i],0)=GW^{0,0}_{\X_n}([D]-[E_i]
-[E_j],0)=1,$$
and
$$GW^{0,u_l}_{\X_n}(l[E_i],0)=+\infty \mbox{ if } l\ge 2,$$
where $i,j=1,\ldots,n$ and $i\ne j$.
\end{prop}

\begin{rem}
The value $GW^{0,u_l}_{\X_n}(l[E_i],0)=+\infty$ comes from the fact
that $\CC^{0,u_l}(l[E_i],0,\emptyset)$ has  dimension strictly bigger
than the expected one.
To define a better ``enumerative'' invariant,
one should 
consider the
virtual fundamental class of this space. 
This is doable, however useless  for
the purposes of this paper. Note however that it follows from Li's
degeneration formula \cite{Li02} 
combined with the proof of Corollary \ref{cor:no sEi} that this
finer invariant should be equal to $0$.
\end{rem}

\subsection{Enumeration of real curve in $\X_n$}\label{sec:real Xn}
Recall that $\X_n(\kappa)$ is the real surface obtained by blowing up $\C
P^2$ at
$\kappa$
pairs of conjugated points and $n-2\kappa$ real points on $E$.

\begin{defi}
A real  configuration $\x^\circ$ in $\X_n(\kappa)$ is said to be
$(E,s)$-compatible if $\x^\circ\cap  E=\emptyset$ and 
$\x^\circ$ contains $s$ pairs of complex conjugated points. 
If $L$ is a connected component of $\R \X_n(\kappa)\setminus \R E$, 
we say that $\x^\circ$ is
$(E,s,L)$-compatible
if $\R\x^\circ$ is in addition contained in $L$.

Given $\alpha^\Re,\alpha^\Im \in \Z_{\ge 0}^\infty$, a real configuration
$\x_E=\{p_{i,j}\}_{0\le j\le \alpha^\Re_i,\ i\ge 1}\sqcup 
\{q_{i,j},\overline{q_{i,j}}\}_{0\le  j\le \alpha^\Im_i,\ i\ge 1}$
 in $E\setminus \bigcup_{i=1}^n E_i$ is said to be 
of type
$(\alpha^\Re,\alpha^\Im)$ if $\{p_{i,j}\}\subset \R E$ 
and $\{q_{i,j}\}\subset E\setminus \R E$.
\end{defi}

Choose $d\in H_2(\X_n;\Z)$ so that $d\ne l[E_i]$ with $l\ge 2$, 
choose $r,s\in\Z_{\ge 0}$, and 
$\alpha^\Re,\beta^\Re,\alpha^\Im,\beta^\Im\in \Z_{\ge 0}^\infty$
such that
$$d\cdot [D]-1+g  +  |\beta^\Re|+2|\beta^\Im|=r+2s\quad\mbox{and}\quad
   I\alpha^\Re + I\beta^\Re +2I\alpha^\Im +2I\beta^\Im=d\cdot [E].$$
Choose a generic real configuration $\x=\x^\circ\sqcup \x_E$ of points in
$\X_n$,
with $\x^\circ$ a  $(E,s)$-compatible 
configuration of $d\cdot [D ]-1+g  +  |\beta^\Re|+2|\beta^\Im|$ points, 
and $\x_E$ a configuration of type $(\alpha^\Re,\alpha^\Im)$.
Denote by $\R\CC^{\alpha^\Re,\beta^\Re,\alpha^\Im,\beta^\Im}(d,s,\x)$ the set
of real maps $f:\C P^1\to \X_n(\kappa)$ in 
$\CC^{\alpha^\Re+2\alpha^\Im,\beta^\Re+2\beta^\Re}(d,0,\x)$ such that
 for any $i\ge 1$, the curve $f(C)$ has exactly $\beta^\Re_i$ real intersection
points (resp. $\beta^\Im_i$ pairs of conjugated intersection points) with
$E$ of multiplicity $i$ and disjoint from $\x_E$.
Then define the following number
$$W_{\X_n(\kappa)}^{\alpha^\Re,\beta^\Re,\alpha^\Im,\beta^\Im}(d,s,\x)=
\sum_{f\in\R\CC^{\alpha^\Re,\beta^\Re,\alpha^\Im,\beta^\Im}(d,s,\x)}(-1)^{m_{\R\X_n(\kappa)}(f(C))}. $$ 

Suppose now that $n=2\kappa$, in particular $\R E$ disconnects $\R
\X_n(\kappa)$.
 Given
 $L$ a connected component of $\R \X_n(\kappa)\setminus \R E$,
and a $(E,s,L)$-compatible configuration $\x^\circ$,
 denote by $\R\CC^{\alpha^\Re,\beta^\Re,\alpha^\Im,\beta^\Im}_L(d,s,\x)$ the set
of elements of $\R\CC^{\alpha^\Re,\beta^\Re,\alpha^\Im,\beta^\Im}(d,s,\x)$    such that
 $f(\R P^1)\subset L\cup\R E$.
For  $L'=\R \X_n(\kappa)$ or $L'=L$, define:
$$W_{\X_n(\kappa),L,L'}^{\alpha^\Re,\beta^\Re,\alpha^\Im,\beta^\Im}(d,s,\x)=
\sum_{f\in\R\CC^{\alpha^\Re,\beta^\Re,\alpha^\Im,\beta^\Im}_L(d,s,\x)}(-1)^{m_{L'}(f(C))}. $$ 
Note that these three series of
 numbers  may vary with the choice of $\x$.

\medskip
The following lemma will be needed later on, in particular in the proof of Theorem \ref{WFD}.
Recall that  $\R P^2\setminus \R E$  has two connected components: one
is homeomorphic to a disk and 
is called the \emph{interior} of $\R E$, while the other is homeomorphic to
a M\"obius band and is called its \emph{exterior}.

\begin{lemma}\label{lem:line conic}
Let $D$ be a non-real line in $\C P^2$, intersecting $E\setminus \R E$ in the points
$p$ and $q$ ($p=q$ if $D$ is tangent to $E$). Then $D$ intersects
$\R P^2$ in the interior of $\R E$ if and only if $p$ and $q$ are in the same
connected component of $E\setminus \R E$. 
\end{lemma}
\begin{proof}
 By continuity, the
connected component of $\R P^2\setminus \R E$ containing $D\cap \R
P^2$ only depends on whether $p$ and $q$ are in the same connected
component of $E\setminus \R E$ or not. 

Clearly, the two lines tangent to $E$ and passing through a point $p$ in $\R P^2$
are  real if and only if $p$ lyes in the exterior of
$\R E$. Hence we get that $D\cap \R P^2$ is in the interior of $\R E$
if $p$ and $q$ are in the same
connected component of $E\setminus \R E$.
Since there exist non-real lines intersecting $\R P^2$ in the exterior
of $\R E$, the converse is proved.
\end{proof}

\subsection{Relative invariants of $\N$}\label{sec:relative N}
Recall that $\N=\P(\N_{E/\X_n}\oplus\C)$, and that $\N$ contains two
distinguished disjoint sections $E_\infty$ and $E_0$ of the $\C P^1$-bundle $\pi_E:\N\to E_\infty$.
One computes easily that 
$$[E]^2=[E_0]^2=-[E_\infty]^2= 4-n.$$
The group $H_2(\N;\Z)$ is the free abelian group  generated by
$[E_\infty]$ and $[F]$, where $F$ is a fiber of $\pi_E$, and the first
Chern class of $\N$ is given by
$$c_1(\N)=2[E_\infty] +(6-n)[F]. $$
Note that $[E_0]=[E_\infty] +(4-n)[F]$. 
Let $d\in H_2(\N;\Z)$  and 
$\alpha,\alpha',\beta,\beta'\in\Z_{\ge 0}^\infty$
such that
$$I\alpha + I\beta=d\cdot [E_0]\quad \mbox{ and }\quad I\alpha' +
I\beta'=d\cdot [E_\infty].$$ 
Choose a configuration $\x=\x^\circ\sqcup \x_{E_0}\sqcup
\x_{E_\infty}$ of points in $\N$, 
with $\x^\circ$ a configuration of $2d\cdot [F]-1+g  +  |\beta|+|\beta'|$ points
in $\N\setminus \left(E_0\cup E_\infty\right)$, 
and $\x_{E_0}=\{p_{i,j}\}_{0\le i\le \alpha_j, j\ge 0}$
(resp. $\x_{E_\infty}=\{p'_{i,j}\}_{0\le i\le \alpha'_j, j\ge 0}$) a configuration of
$|\alpha|$ (resp. $|\alpha'|$) points in $E_0$ (resp. $E_\infty$).
Let $\FF^{\alpha,\beta,\alpha',\beta'}(d,g,\x)$ be the set
 of holomorphic maps $f:C\to \X_n$ with
$C$  a connected  algebraic curve of arithmetic genus $g$, such that 
$f(C)$  realizes the homology class $d$ in $\N$, contains $\x$, does
 not contain neither $E_0$ nor $E_\infty$ as a component, and
$$f^*(E_0)=\sum_{i\ge 1} \sum_{j=1}^{\alpha_i}iq_{i,j} + 
\sum_{i\ge 1} \sum_{j=1}^{\beta_i}i\tilde q_{i,j}\quad \mbox{with} f(q_{i,j})=p_{i,j}$$
and
$$f^*(E_\infty)=\sum_{i\ge 1} \sum_{j=1}^{\alpha'_i}i q'_{i,j} + 
\sum_{i\ge 1} \sum_{j=1}^{\beta'_i}i\tilde q'_{i,j}\quad
\mbox{with} f( q'_{i,j})=p'_{i,j}.$$
The corresponding Gromov-Witten invariant 
relative to $E_0\cup E_\infty$ is defined by
$$GW_\N^{\alpha,\beta,\alpha',\beta'}(d,g)=\sum_{f\in 
\FF^{\alpha,\beta,\alpha',\beta'}(d,g,\x)}\mu(f,\x^\circ) $$ 
for a generic configuration $\x$, where $\mu(f,\x^\circ)$ is defined
as in Section \ref{sec:relative Xn}.

\begin{prop}[{\cite[Section 3]{Vak2}}]\label{prop:generic N}
The number $GW_\N^{\alpha,\beta,\alpha',\beta'}(d,g)$
 is finite and does not depend on $\x$.

If $d\ne l[F]$ with $l\ge 2$,
then    
any element $f:C\to\N$ of
  $\FF^{\alpha,\beta,\alpha',\beta'}(d,g,\x)$  satisfies the following
properties:
\begin{itemize}
\item the curve $C$ is smooth and irreducible;
\item $f$ is an immersion, birational onto its image (in particular it
  has no non-trivial automorphism);
\item $f(C)$ intersects the curves $E_0$ and $E_\infty$ at
  non-singular points.
\end{itemize}

If $d= l[F]$ with $l\ge 2$, then the set
$\FF^{\alpha,\beta,\alpha',\beta'}(d,g,\x)$ is either empty or has a unique element 
 $f:C\to\N$ which is a ramified covering of degree $l$ to its image,
with exactly two ramification points,  one of them is
mapped to $E_0$, and the other to $E_\infty$.
\end{prop}

\begin{prop}\label{prop:initial values N}
Suppose that
$2d\cdot [F] -1+g+|\beta|+|\beta'|=0$. 
Then the number $GW^{\alpha,\beta,\alpha',\beta'}_{\N}(d,g)$ is
non-zero only in the following cases
$$GW_\N^{u_l,0,0,u_l}(l[F],0)=GW_\N^{0,u_l,u_l,0}(l[F],0)=\frac{1}{l}.$$

Suppose that
$2d\cdot  [F] -1+g+|\beta|+|\beta'|=1$. 
Then the number $GW^{\alpha,\beta,\alpha',\beta'}_{\N}(d,g)$ is
non-zero only in the following cases:
$$GW_\N^{\alpha,0,\alpha',0}([E_\infty] +l[F],0)= GW_\N^{0,u_l,0,u_l}(l[F],0)=1, $$
where $I\alpha = l$ and $I\alpha' = n-4+l$.

Suppose that
$2d\cdot [F] -1+g+|\beta|+|\beta'|=2$. 
Then the number $GW^{\alpha,\beta,\alpha',\beta'}_{\N}(d,g)$ is
non-zero only in the following cases:
$$GW_\N^{\alpha,u_j,\alpha',0}([E_\infty]+l[F],0)=
GW_\N^{\alpha,0,\alpha',u_j}([E_\infty] +lF,0)=j . $$ 
where  $I\alpha = l$  and
$I\alpha' = n-4+l-j$ in the former case, and  $I\alpha = l-j$ and
$I\alpha' = n-4+l$ in the latter case. 
\end{prop}
\begin{proof}
The cases $2d\cdot [F] -1+g+|\beta|+|\beta'|\le 1$ are considered in
\cite[Section 8]{Vak2}, so  suppose that $2d\cdot [F]
-1+g+|\beta|+|\beta'|=2$.
Writing $d=l_\infty[E_\infty] + l[F]$ we get that
$$2l_\infty  + g + |\beta|= 3.$$
In particular we have $l_\infty\le 1$. One sees  easily that in both
case $l_\infty=0$ and $l_\infty=g=1$ we have 
$GW^{\alpha,\beta,\alpha',\beta'}_{\N}(d,g)=0$.
The value 
$GW_\N^{\alpha,u_j,\alpha',0}(E_\infty+lF,0)=
GW_\N^{\alpha,0,\alpha',u_j}(E_\infty +lF,0)=j $  can  be computed
using the recursion formula {\cite[Theorem 6.12]{Vak2}}, 
nevertheless I  give here a proof by hand that
 will  be useful  in Section \ref{sec:local real}.

The surface $\N$ is a toric compactification of $(\C^*)^2$, 
and for a suitable choice of coordinates, a curve in $\FF^{\alpha,u_j,\alpha',0}(E_\infty +lF,0,\x)$ is the
compactification of a curve in $(\C^*)^2$ with equation
$$ay(x-b)^j=Q(x)$$
with $Q(x)$ 
a monic rational function of degree 
$4-n+j$.
Note that 
$Q(x)$  is entirely determined by 
 $\x_{E_0}\sqcup\x_{E_\infty}$.
If 
$\x^\circ=\{(x_0,y_0),(x_1,y_1)\}$, then $a$ and $b$ satisfy the
two equations $ay_i(x_i-b)^j=Q(x_i)$.
Hence $a$ is determined by
 $b$, and this latter is a solution of the equation
$$\left(\frac{x_0-b}{x_1-b} \right)^j =
 \frac{y_1Q(x_0)}{y_0Q(x_1)}$$
which clearly has exactly $j$ solutions.
\end{proof}

\subsection{Enumeration of real curves in $\N$}\label{sec:local real}
Now suppose that  $E$ is  real  with 
 a non empty real part in $\X_n(\kappa)$. In this case
 the surface $\N$ has a natural real
structure induced by the real structure on $\X_n(\kappa)$,
 and both curves
$E_0$ and $E_\infty$ are real with a non-empty real part.
Note that the map $\pi_E$ is a real map.
The real part $\R\N$ is a Klein bottle if $n$ is odd, and a torus if
$n$ is even.
In this latter case $\R \N\setminus \left(\R E_0\cup \R E_\infty\right)$
has two connected components that we denote arbitrarily by $N^\pm$.

\begin{lemma}\label{lem:m real}
Suppose that $\x=\x^\circ \sqcup \x_{E_0}\sqcup \x_{E_\infty}$ is a
generic real configuration of points in $\N$ with
$\x^\circ=\{(x_0,y_0),(\overline{x_0},\overline{y_0})\}$. Then 
the $j$ elements $f:\C P^1\to\N$ of $\FF^{\alpha,u_j,\alpha',0}([E_\infty] +l[F],0,\x)$
are all real. If moreover $n$ is even, and  each point in 
$\R x_{E_0}\sqcup \R \x_{E_\infty}$ is a point of  even order of
contact of $f(\C P^1)$ with $E_0\sqcup E_\infty$,
then $\frac{j}{2}$ of those elements satisfy $f(\R P^1)\subset N^+\cup E_0\cup
E_\infty$, and $\frac{j}{2}$ elements  satisfy $f(\R P^1)\subset N^-\cup E_0\cup
E_\infty$.

An analogous statement holds for elements of
$\FF^{\alpha,0,\alpha',u_j}([E_\infty] +l[F],0,\x)$. 
\end{lemma}
\begin{proof}
Let us use notations introduced in the proof of Proposition
\ref{prop:initial values N}. With the additional assumption of the
lemma, we have that $x_1=\overline{x_0}$ and
 $y_1=\overline{y_0}$, and that $Q(x)$ is real.
The $j$ elements of $\FF^{\alpha,u_j,\alpha',0}([E_\infty] +l[F],0,\x)$
correspond to the $j$ solutions of the equation
\begin{equation}\label{equ:b}
\left(\frac{x_0-b}{\overline{x_0}-b} \right)^j =
 \frac{\overline{y_0}Q(x_0)}{y_0Q(\overline{x_0})}.
\end{equation}
The projective transformation
 $b\mapsto \frac{x_0-b}{\overline{x_0}-b} $ maps the
real line to the set of
complex numbers of absolute value 1, hence all
elements of $\FF^{\alpha,u_j,\alpha',0}([E_\infty] +l[F],0,\x)$ are
real.

Suppose now that $n$ is even, and  each point in 
$\R x_{E_0}\sqcup \R \x_{E_\infty}$ is a point of even order of
contact of $f(\C P^1)$ with $E_0\sqcup E_\infty$.
In particular $j$ is even, and the sign of   $Q(x)$ is
constant.
The real part of an element of $\FF^{\alpha,u_j,\alpha',0}([E_\infty]
+l[F],0,\x)$ is mapped to 
 $N^\pm\cup \R E_0\cup\R E_\infty$  depending on the sign of
$\frac{Q(b)}{a}$, that is to say
 depending on the sign of 
$$a=\frac{Q(x_0)}{(x_0-b)^j y_0}.$$ 
Denoting by $e^{i\theta}$ any $j$-th root of 
$\frac{\overline{y_0}Q(x_0)}{y_0Q(\overline{x_0})}
$, the $j$ solutions of Equation (\ref{equ:b}) are given by
$$b_k=\frac{x_0 -\overline{x_0}e^{i(\theta +
    \frac{2k\pi}{j})}}{1-e^{i(\theta + \frac{2k\pi}{j})}},
 \quad k=0,\ldots, j-1,$$
and so
$$(x_0-b_k)^j=\left( \frac{-x_0e^{i(\theta + \frac{2k\pi}{j})} +\overline{x_0}e^{i(\theta +
    \frac{2k\pi}{j})}}{1-e^{i(\theta + \frac{2k\pi}{j})}}\right)^j=
e^{i  k\pi}\left(\frac{e^{\frac{i\theta}{2}} Im\left(x_0\right)}{sin\left( \frac{\theta}{2}
  +    \frac{k\pi}{j}\right)}\right)^j.$$
Hence the sign of $\frac{a_0}{a_k}$ coincides with
$(-1)^k$, and the lemma is proved in the case of
$\FF^{\alpha,u_j,\alpha',0}([E_\infty] +l[F],0,\x)$.
The proof in the case of $\FF^{\alpha,0,\alpha',u_j}([E_\infty]
+l[F],0,\x)$ is analogous. 
\end{proof}

Both sets $E_0\setminus \R E_0$ and  $E_\infty\setminus \R E_\infty$  have two
connected components, denoted respectively by $E_0^\pm$ and
$E_\infty^\pm$ in such a way that a non-real fiber of $\N$ intersects
both $E_0^+$ and $E_\infty^+$, or both $E_0^-$ and $E_\infty^-$.
Given a complex algebraic curve $C$ in $\N$,  denote by $n_C^\pm$ the
sum of the multiplicities of intersection points of $C$ with $E_0\cup
E_\infty$ contained in $E_0^\pm\cup E_\infty^\pm$.

\begin{lemma}\label{lem:even}
Suppose that $n$ is even, and
let $C$ be a complex algebraic curve in $\N$ realizing the class
$[E_\infty] + l[F]$, intersecting $\R\N$ transversely and 
in finitely many points, and
with $C\cap\left(\R E_0\cup \R E_\infty \right)=\emptyset$. 
Then $n_C^++n_C^-$ is even and both numbers $|C\cap N^+|$
and $|C\cap N^-|$ have the same parity as $ \frac{n_C^+-n_C^-}{2}$.

\end{lemma}
\begin{proof}
By homological reasons $n_C^++n_C^-$ has the same parity as $n$,
and is indeed even.
Next by continuity, it is enough to consider the case when $C$ has equation 
$y=i(x-i)^{a}(x+i)^{b}$ with $a,b\in\Z$ and $a+b$ even. 
Since
$(x-i)^{a}(x+i)^{b}=(x-i)^{a-b}(x^2+1)^{b}$, 
we can further restrict
 to the case when $C$ has equation 
$y=i(x-i)^{a}$ with $a\in 2\Z$.
By setting $X=x-i$, intersection points of $C$ with $\R N^+$ and $\R N^-$
respectively correspond to the
solutions of the systems of equations
$$\left\{ \begin{array}{l}iX^{a}\in \R_{>0}\\  X+i\in\R\end{array}\right.
\quad \mbox{ and}\quad \left\{ \begin{array}{l}iX^{a}\in \R_{<0}\\  X+i\in\R\end{array}\right..  $$
The number of solutions of these two systems is easily determined graphically,
see Figure \ref{fig:even}. 
\end{proof}
\begin{figure}[h!]
\centering
\begin{tabular}{c} 
\includegraphics[height=3cm, angle=0]{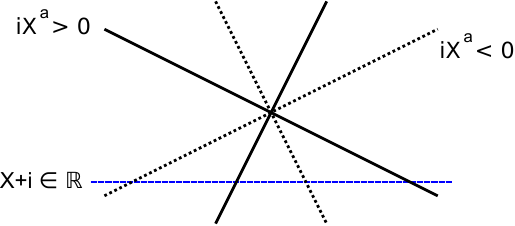}
\end{tabular}
\caption{Solutions of $iX^{a}\in \R$ and $  X+i\in\R$, with $a$ even}
\label{fig:even}
\end{figure}

\section{Floor diagrams relative to a conic}\label{sec:FD} 

\subsection{Floor diagrams}
A {\em weighted graph} is a graph $\Gamma$  equipped with 
a function $w:Edge(\Gamma)\to\Z_{>0}$.
The weight allows one to define the {\em divergence} at the vertices.
Namely, for a vertex $v\in Vert(\Gamma)$
we define the divergence $div(v)$ to be the sum of the weights of all
incoming edges minus the sum of the weights of all outgoing edges.

\begin{defi}
A connected weighted oriented graph $\D$ is called a
{\em floor diagram of genus $g$ and degree $d_\D$}
if the following conditions hold
\begin{itemize}
\item the oriented graph $\D$ is acyclic;
\item any element in $Vert^\infty(\D)$ is adjacent to exactly one edge of $\D$;
\item  $div(v)=2$ or $4$ for any 
$v\in Vert(\D)\setminus Vert^{\infty}(\D)$, and 
$div(v)\le  -1$
for every $v\in Vert^{\infty}(\D)$;
\item if  $div(v)=2$, then $v$ is a sink
  (i.e. all its adjacent edges are oriented toward $v$);
\item the first Betti number $b_1(\D)$ equals $g$;
\item one has
$$\sum_{v\in Vert^\infty (\D)}div(v)=-2d_\D. $$
\end{itemize}
A vertex $v\in Vert(\D)\setminus Vert^\infty(\D)$ is called {\em a floor of degree $\frac{div
  (v)}{2}$}.
\end{defi}
Formally, those objects should be called 
\emph{floor diagrams in $\C P^2$ relative to a conic}. However since
these are the only floor diagrams considered in this text, I opted for
an abusive but  shorter name.
 Note that there are slight differences with the original 
definition of
floor diagrams in \cite{Br7}, \cite{Br6b}, and \cite{Br6}.

Here are the convention I use to depict floor diagrams :
floors of degree 2 are represented by white ellipses; floors of degree 1 are
represented by grey ellipses; vertices in $Vert^\infty(\D)$ are not
represented; edges of $\D$
are represented by vertical lines, and
the orientation is
implicitly from down to up. We specify the weight of an edge only if
this latter is at least 2.

\begin{exa}
Figure  \ref{fig:ex FD} depicts all floor diagrams of  degree 1, 2
and 3 with  each edge in $Edge^\infty(\D)$  of weight 1.
\end{exa}

\begin{figure}[h]
\centering
\begin{tabular}{ccccccccc}
\includegraphics[width=1.5cm, angle=0]{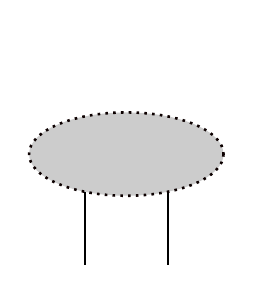}
&
\includegraphics[width=1.5cm, angle=0]{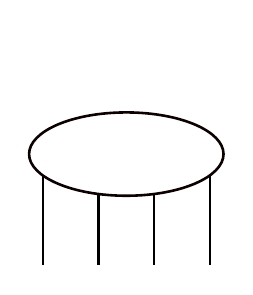}
&
\includegraphics[width=1.5cm, angle=0]{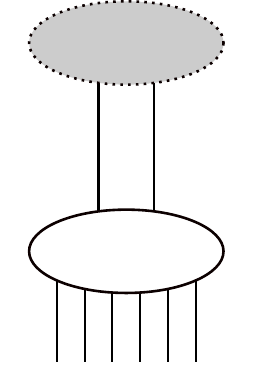}
&
\includegraphics[width=1.5cm, angle=0]{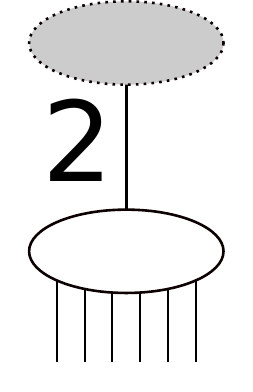}
&
\includegraphics[width=1.5cm, angle=0]{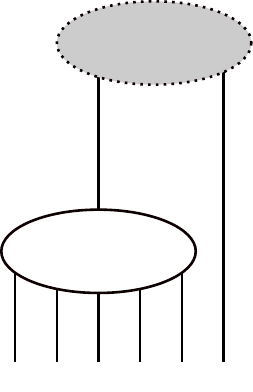}
\\
\\ a) $d=1$, $g=0$ & b)  $d=2$, $g=0$ & c)  $d=3$, $g=1$
 & d)  $d=3$, $g=0$ & e)  $d=3$, $g=0$
\end{tabular}
\caption{Examples of floor diagrams}
\label{fig:ex FD}
\end{figure}

A map $m$ between two partially ordered sets is said
to be \textit{increasing} if
$$m(i)>m(j)\Longrightarrow i>j$$
Note that a floor diagram 
inherits a partial ordering from the orientation of its
underlying graph.

\begin{defi}\label{def marking}
Choose two non-negative integers $n$ and $g$, a homology class $d\in
H_2(\X_n;\Z)$,  and two vectors $\alpha,\beta\in \Z_{\ge 0}^\infty$
such that
$$I\alpha+I\beta = d\cdot [E]. $$
Let $A_0,A_1,\ldots, A_n$ be some disjoint sets such that
$|A_i|=d\cdot [E_i]$ for
$i=1,\ldots, n$,  and
$$A_0=\left\{1,\ldots, d\cdot [D] -1 +g +|\alpha|+|\beta|\right\} $$
A $d$-marking of type $(\alpha,\beta)$ of a floor diagram $\D$ of genus
$g$ and degree $d\cdot [D]$   is a map $m  :
\bigcup_{i=0}^n A_i\rightarrow
\mathcal{D}$ such that 
\begin{enumerate}
\item the map $m$ is injective  and  increasing, with no floor of
  degree 1 of $\D$ contained in the image of $m$; 
\item for each vertex $v\in Vert^\infty(\D)$ adjacent to the edge $e\in Edge^\infty(\D)$,
  exactly one of the two elements $v$ and $e$ is in the image of $m$;
\item $m\left(\bigcup_{i=1}^n A_i \right)\subset Vert^\infty(\D)$;
\item   for
  each $i=1,\ldots, n$, 
a floor of $\D$ is adjacent to at most one edge adjacent to a vertex
in $m(A_i)$;
\item $m\left(\left\{ 1,\ldots, |\alpha| \right\}\right)= 
m(A_0)\cap Vert^\infty(\D)$;
\item for $1 \le k\le \alpha_j$, 
the edge adjacent to $m(\sum_{i=1}^{j-1}\alpha_i+k)$ is of weight $j$;
\item exactly $\beta_j$ edges in $Edge^\infty(\D)$ of weight $j$ are
  in the image of $m_{|A_0}$. 
\end{enumerate}
\end{defi}

Those conditions imply that all edges in
$m\left(\bigcup_{i=1}^n A_i \right)$ are of weight $1$.
A floor diagram 
enhanced with a  $d$-marking $m$ is called a \textit{$d$-marked floor
  diagram} and is said to be
marked by $m$.

\begin{defi}
Let  $\D$ be a floor diagram  equipped with two
$d$-markings
$$m  : A_0\cup \bigcup_{i=1}^n A_i\rightarrow
\mathcal{D}\quad \mbox{and} \quad m'  :
A_0\cup\bigcup_{i=1}^n A_i'\rightarrow
\mathcal{D}.$$
The markings $m$ and $m'$ 
are called
\textit{equivalent} if there exists an isomorphism of weighted oriented graphs
$\phi : \D\to\D$ and a  bijection 
$\psi: A_0\cup \bigcup_{i=1}^n A_i\rightarrow A_0\cup\bigcup_{i=1}^n
A_i'$, such that
\begin{itemize}
\item $\psi_{|A_0}=Id$;
\item  $\psi_{|A_i}:A_i\to A'_i$ is a bijection
for $i=1,\ldots,n$;
\item  $m'\circ \psi=\phi \circ m$.
\end{itemize}
\end{defi}

In particular, for $i=1,\ldots, n$, the equivalence class of $(\D,m)$ 
 depends on $m(A_i)$ rather than on
  $m_{|A_i}$.
From now on,  marked floor diagrams are considered up to equivalence.

\subsection{Enumeration of complex curves}
The complex multiplicity of a marked floor diagram is defined as in
\cite{Br6b,Br8}. 

\begin{defi}
The complex multiplicity of a marked floor diagram $(\mathcal D,m)$ of
type $(\alpha,\beta)$, 
denoted by $\mu^\C(\mathcal D,m)$, is defined as
$$\mu^\C(\mathcal D,m)=I^\beta
\prod_{e\in \text{Edge}(\D)\setminus Edge^\infty(\D)}w(e)^2 .$$
\end{defi}

Note that the complex multiplicity of a marked floor diagram
only depends on its type and the underlying floor diagram. 

\begin{thm}\label{NFD}
For any  $d\in H_2(\X_n;\Z)$ such that $d\cdot [D]\ge 1$, and any genus $g\ge 0$,  one has
$$GW_{\X_n}^{\alpha,\beta}(d,g)=\sum \mu^\C(\mathcal D,m)$$
where the sum is taken over all $d$-marked floor diagrams of
genus $g$ and type $(\alpha,\beta)$.
\end{thm}

As indicated in the introduction,  one easily
translates Theorem \ref{NFD} 
to a  Caporaso-Harris type formula
following the method exposed in
\cite{Br8}. One obtains in this way a formula similar to the one from
{\cite[Theorem 2.1]{Shu13}}.

\begin{exa}\label{ex1}
Theorem \ref{NFD} applied with $n\le 5$, $\alpha=0$, and
$\beta=(d\cdot [D])u_1$ gives
Gromov-Witten invariants of $X_n$.
In particular, as a simple application of Theorem \ref{NFD} one can use 
 floor diagrams depicted in Figure
\ref{fig:ex FD} to verify that
$$GW_{\C P^2}([D],0)= GW_{\C P^2}(2[D],0)=GW_{\C P^2}(3[D],1)=1\quad
\mbox{and}\quad GW_{\C P^2}(3[D],0 )=4+8=12.$$
\end{exa}

\begin{exa}\label{ex:NFD 4 and 6}
We illustrate Theorem \ref{NFD} with more details by computing
$GW_{\X_6}(4[D]-\sum_{i=1}^6[E_i],0)$ and $GW_{\X_6}(6[D]-2\sum_{i=1}^6[E_i],0)$.
These numbers have been first computed by Vakil {\cite{Vak2}}.

In Figure 
\ref{fig:FD42} are depicted all floor
diagrams of genus 0  admitting a
$(4[D]-\sum_{i=1}^6[E_i])$-marking of  type $(0,2u_1)$. 
Below 
each such floor diagram, I precised
 the sum of  complex multiplicity of all
$(4[D]-\sum_{i=1}^6[E_i])$-marked floor 
diagrams  of  type $(0,2u_1)$
with this underlying floor diagram (the signification of the array
attached to each floor diagram will be explained in Section \ref{sec:real FD}). In order to  make the
pictures clearer, 
I did not depict edges in $m\left( \bigcup_{i=1}^n A_i\right)$.
Theorem \ref{NFD} together with Figure \ref{fig:FD42}
 implies that
$$GW_{\X_6}(4[D]-\sum_{i=1}^6[E_i],0)= 616.$$

Similarly,  Figures \ref{fig:FD6} and \ref{fig:FD62} depict all floor
diagrams of genus 0 admitting a $(6[D]-2\sum_{i=1}^6[E_i])$-marking of
type $(0,0)$. Together with
Theorem \ref{NFD}, this imply that
$$GW_{\X_6}(6[D]-2\sum_{i=1}^6 [E_i],0)= 2002.$$
\end{exa}

\begin{figure}[h]
\centering
 \begin{tabular}{ccc}
\begin{tabular}{cc||c|c|c|c}
\multirow{5}{*}{\includegraphics[width=1.1cm,
    angle=0]{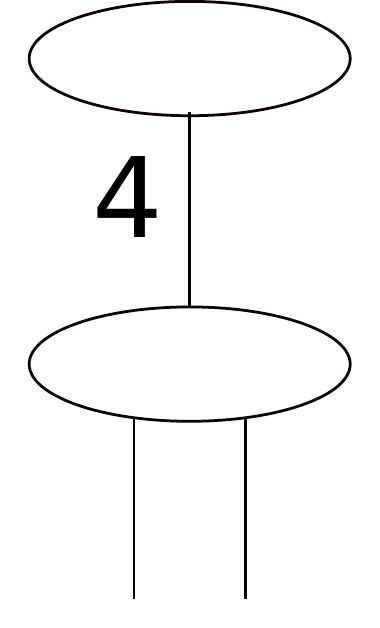}}&
$\kappa\backslash s$& 0 &  \multicolumn{2}{c|}{1} &  2 
\\ \cline{2-6}
&0 &  0& 0 &0 &0  
\\&1   &0 & 0& 0 & 0
\\  &2  &0 &0  &0 &0
\\ &3  &  0& 0& 0& 0
\end{tabular} &
\begin{tabular}{cc||c|c|c|c}
\multirow{5}{*}{\includegraphics[width=1.1cm,
    angle=0]{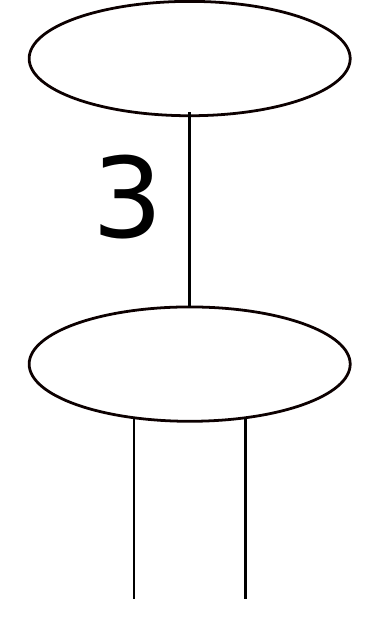}}&
$\kappa\backslash s$&0 &  \multicolumn{2}{c|}{1} & 2
\\ \cline{2-6}
&  0  & 6& 0 & 6 &   18
\\ & 1  & 4 & 0 & 4&   12
\\ & 2  & 2 & 0 & 2 &   6
\\ & 3  &  0&0 &0  &0 
\end{tabular} &
\begin{tabular}{cc||c|c|c|c}
\multirow{5}{*}{\includegraphics[width=1.1cm,
    angle=0]{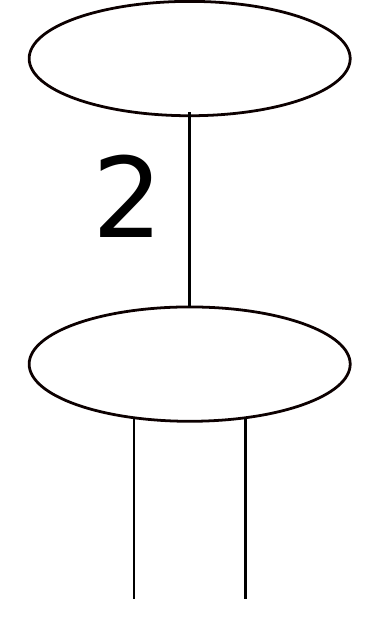}}&
$\kappa\backslash s$& 0 &  \multicolumn{2}{c|}{1} & 2 
\\ \cline{2-6}
&0 & 0&0 &0  & 0 
\\&1  &0 & 0& 0  &0 
\\  &2  & 0& 0 &0 &0
\\ &3  &0 & 0  &0 &0 
\end{tabular} 
\\  $\sum \mu^\C(\D,m) =  16$ &  $\sum \mu^\C(\D,m) = 54$&  $\sum \mu^\C(\D,m) = 60 $

\\ \\ 
\begin{tabular}{cc||c|c|c|c}
\multirow{5}{*}{\includegraphics[width=1.1cm,
    angle=0]{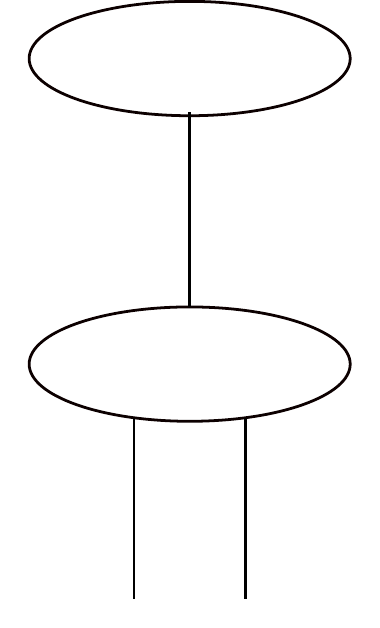}}&
$\kappa\backslash s$& 0 &  \multicolumn{2}{c|}{1} &2
\\ \cline{2-6}
&  0  & 20 & 0&  20&  20
\\ & 1  & 8 & 0& 8&  8
\\ & 2  &  4& 0& 4&  4
\\ & 3  &  0&0 & 0&0
\end{tabular} &
\begin{tabular}{cc||c|c|c|c}
\multirow{5}{*}{\includegraphics[width=1.1cm,
    angle=0]{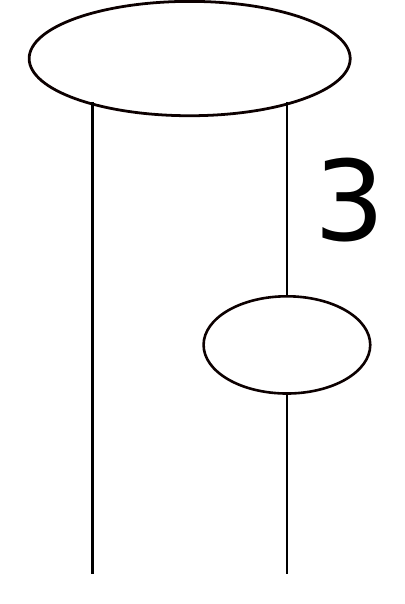}}&
$\kappa\backslash s$& 0 &  \multicolumn{2}{c|}{1} & 2 
\\ \cline{2-6}
&  0  & 4&  2& 0 &0
\\ & 1  & 4&  2& 0 &0
\\ & 2  & 4&  2& 0&0
\\ & 3  & 4&  2& 0&0
\end{tabular} &
\begin{tabular}{cc||c|c|c|c}
\multirow{5}{*}{\includegraphics[width=1.1cm,
    angle=0]{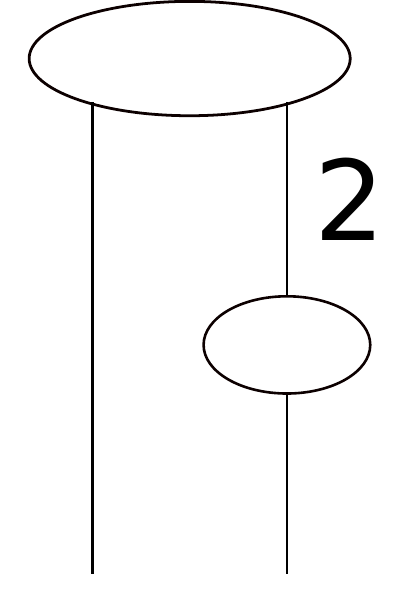}}&
$\kappa\backslash s$& 0 &  \multicolumn{2}{c|}{1} &  2 
\\ \cline{2-6}
&  0  & 0&0&0 &0
\\ & 1  &0 &0& 0&0
\\ & 2  &0 &0& 0&0
\\ & 3  &0 &0& 0&0
\end{tabular} 
\\  $\sum \mu^\C(\D,m) =  20$ &  $\sum \mu^\C(\D,m) = 36$&  $\sum \mu^\C(\D,m) = 96 $

\\\\
\begin{tabular}{cc||c|c|c|c}
\multirow{5}{*}{\includegraphics[width=1.1cm,
    angle=0]{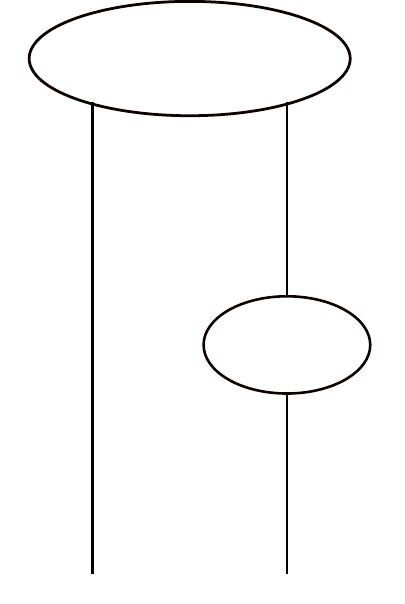}}&
$\kappa\backslash s$& 0 &  \multicolumn{2}{c|}{1} & 2
\\ \cline{2-6}
&  0  &60&  30 &0 &0 
\\ & 1  & 28 & 14 & 0  &0
\\ & 2  &  12 & 6&0&0 
\\ & 3  & 12 & 6&0&0 
\end{tabular} &
\begin{tabular}{cc||c|c|c|c}
\multirow{5}{*}{\includegraphics[width=1.1cm,
    angle=0]{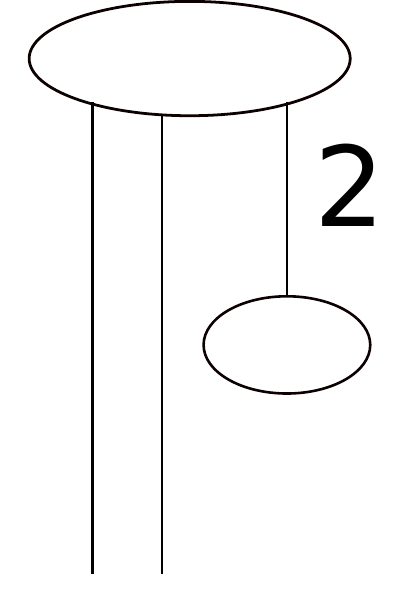}}&
$\kappa\backslash s$& 0 &  \multicolumn{2}{c|}{1} &  2
\\ \cline{2-6}
&  0  & 0&0&0 &0
\\ & 1  &0&0 &0 &0
\\ & 2  &0&0 &0 &0
\\ & 3  &0&0 &0 &0
\end{tabular} &
\begin{tabular}{cc||c|c|c|c}
\multirow{5}{*}{\includegraphics[width=1.1cm,
    angle=0]{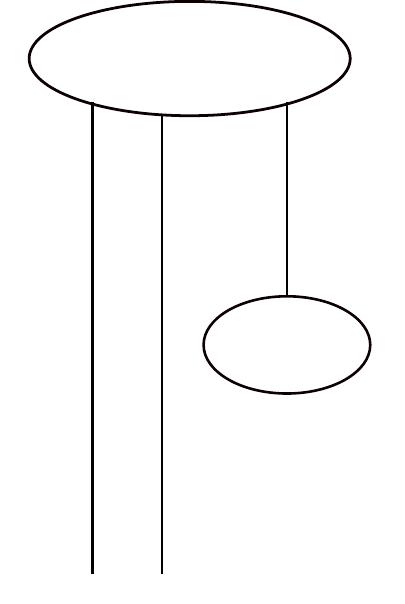}}&
$\kappa\backslash s$& 0 &  \multicolumn{2}{c|}{1} &  2
\\ \cline{2-6}
&  0  & 36 & 6 & 6 &  6
\\ & 1  & 24  & 4 & 4 & 4
\\ & 2  & 12 & 2&2 &2
\\ & 3  & 0&0&0&0 
\end{tabular} 
\\  $\sum \mu^\C(\D,m) =  60$ &  $\sum \mu^\C(\D,m) = 24$&  $\sum \mu^\C(\D,m) = 36 $

\\\\
\begin{tabular}{cc||c|c|c|c}
\multirow{5}{*}{\includegraphics[width=1.1cm,
    angle=0]{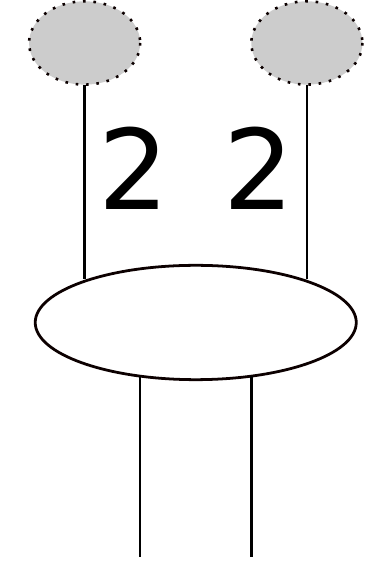}}&
$\kappa\backslash s$&0 &  \multicolumn{2}{c|}{1} &  2
\\ \cline{2-6}
&  0  & 0&0&0 &0
\\ & 1  & 0&0  &0 &0
\\ & 2  &0 &0 &0 &0
\\ & 3  &0 &0&0 &0
\end{tabular} &
\begin{tabular}{cc||c|c|c|c}
\multirow{5}{*}{\includegraphics[width=1.1cm,
    angle=0]{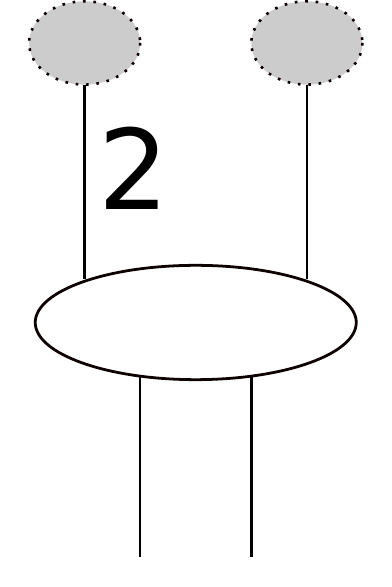}}&
$\kappa\backslash s$& 0 &  \multicolumn{2}{c|}{1} & 2
\\ \cline{2-6}
&  0  & 0&0& 0&0
\\ & 1  &0&0 &0 &0
\\ & 2  &0&0 &0 &0
\\ & 3  &0&0 &0 &0
\end{tabular} &
\begin{tabular}{cc||c|c|c|c}
\multirow{5}{*}{\includegraphics[width=1.1cm,
    angle=0]{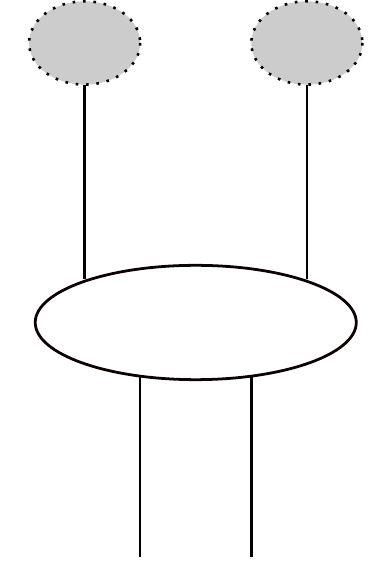}}&
$\kappa\backslash s$&0 &  \multicolumn{2}{c|}{1} & 2
\\ \cline{2-6}
&  0  & 30&  0 & 30 &  30
\\ & 1  & 12 &0 &12&  12
\\ & 2  & 2 & 0& 2&  2
\\ & 3  &0 &0&0&0 
\end{tabular} 
\\  $\sum \mu^\C(\D,m) =  16$ &  $\sum \mu^\C(\D,m) = 48$&  $\sum \mu^\C(\D,m) =  30$

\\\\
\begin{tabular}{cc||c|c|c|c}
\multirow{5}{*}{\includegraphics[width=1.1cm,
    angle=0]{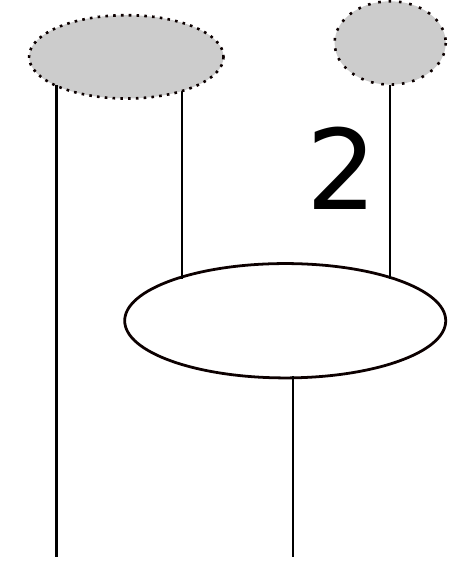}}&
$\kappa\backslash s$& 0 &  \multicolumn{2}{c|}{1} & 2 
\\ \cline{2-6}
&  0  & 0&0&0 &0
\\ & 1  & 0& 0 &0 &0
\\ & 2  &0&0  &0 &0
\\ & 3  &0&0 & 0&0
\end{tabular} &
\begin{tabular}{cc||c|c|c|c}
\multirow{5}{*}{\includegraphics[width=1.1cm,
    angle=0]{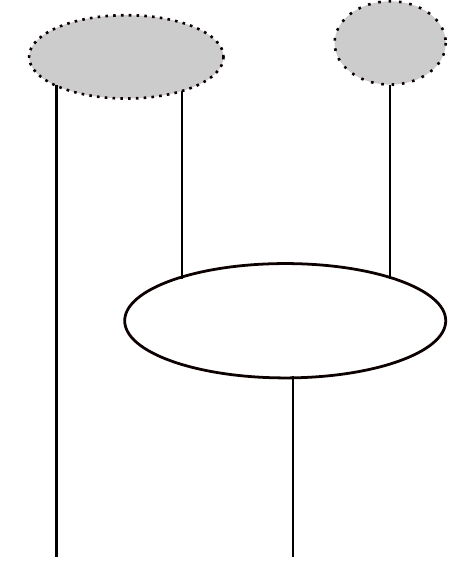}}&
$\kappa\backslash s$&0 &  \multicolumn{2}{c|}{1} & 2 
\\ \cline{2-6}
&  0  & 60&  36 & 0&0
\\ & 1  & 40 &24  & 0&0
\\ & 2  & 20 & 12& 0&0
\\ & 3  & 0&0&0 &0
\end{tabular} &
\begin{tabular}{cc||c|c|c|c}
\multirow{5}{*}{\includegraphics[width=1.1cm,
    angle=0]{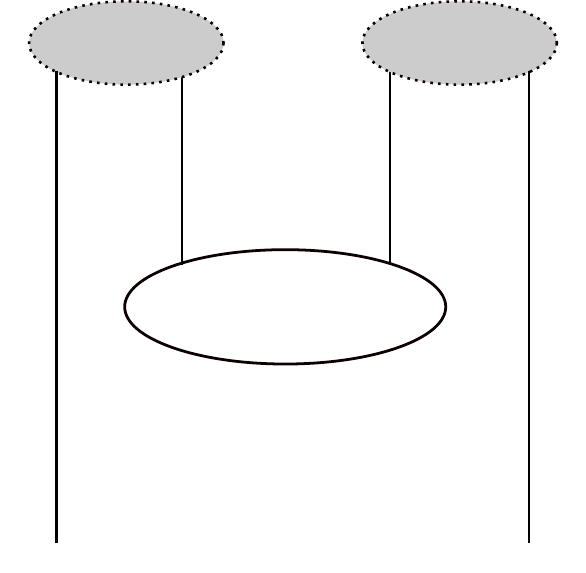}}&
$\kappa\backslash s$& 0 &  \multicolumn{2}{c|}{1} & 2 
\\ \cline{2-6}
&  0  & 20& 6& 0&0
\\ & 1  &20& 6&  0&0
\\ & 2  &20& 6&  0&0
\\ & 3  &20& 6&  0&0
\end{tabular} 
\\  $\sum \mu^\C(\D,m) = 40 $ &  $\sum \mu^\C(\D,m) = 60$&  $\sum \mu^\C(\D,m) =  20$

\end{tabular}
\caption{$(4[D]-\sum_{i=1}^6 [E_i])$-floor diagrams of genus
  0 and type $(0,2u_1)$}
\label{fig:FD42}
\end{figure}

\begin{figure}[h]
\centering
\begin{tabular}{ccc}
\begin{tabular}{cc||c|c|c}
\multirow{7}{*}{\includegraphics[width=1.1cm,
    angle=0]{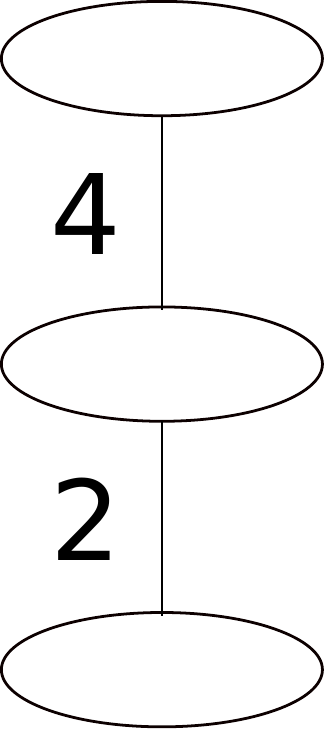}}&
$\kappa,\epsilon\backslash s$& 0 & 1 & 2
\\ \cline{2-5}
&  0  &  0 & 0 &0
\\ & 1  & 0 & 0&0
\\ & 2  & 0 & 0&0
\\ & 3  & 0 & 0&0
\\ \cline{2-5}
& 0  &  0& 0&0
\\&  1  & 0 &0 &0
\end{tabular} &
\begin{tabular}{cc||c|c|c}
\multirow{7}{*}{\includegraphics[width=1.1cm,
    angle=0]{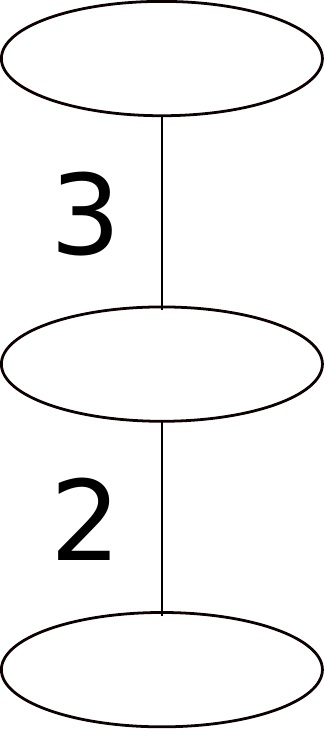}}&
$\kappa,\epsilon\backslash s$& 0 & 1 & 2
\\ \cline{2-5}
&  0  & 0&0 &0
\\ & 1  &0 & 0&0
\\ & 2  & 0& 0&0
\\ & 3  & 0& 0&0
\\ \cline{2-5}
& 0  & 0& 0&0
\\&  1  &0 & 0&0
\end{tabular} &
\begin{tabular}{cc||c|c|c}
\multirow{7}{*}{\includegraphics[width=1.1cm,
    angle=0]{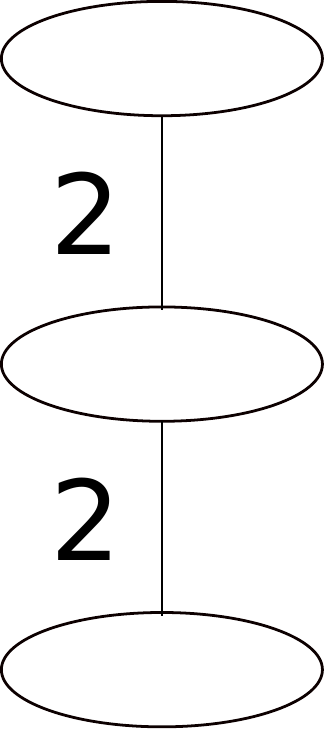}}&
$\kappa,\epsilon\backslash s$& 0 & 1 & 2
\\ \cline{2-5}
&  0  & 0&0 &0
\\ & 1&0 & 0&0
\\ & 2& 0& 0&0
\\ & 3& 0& 0&0
\\ \cline{2-5}
& 0  & 48& 24& 12
\\&  1  &48 & 24& 12
\end{tabular} 
\\  $\sum \mu^\C(\D,m) = 64 $ &  $\sum \mu^\C(\D,m) = 216$&  $\sum \mu^\C(\D,m) =  240$

\\\\\begin{tabular}{cc||c|c|c}
\multirow{7}{*}{\includegraphics[width=1.1cm,
    angle=0]{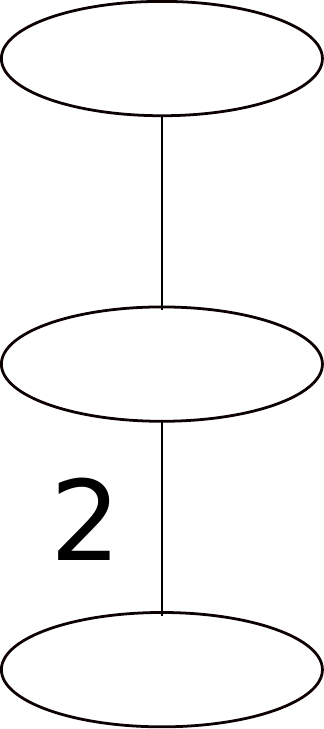}}&
$\kappa,\epsilon\backslash s$& 0 & 1 & 2
\\ \cline{2-5}
&  0  & 0 &0 &0
\\ & 1  & 0& 0&0
\\ & 2  & 0& 0&0
\\ & 3  & 0& 0&0
\\ \cline{2-5}
& 0  & 0& 0&0
\\&  1  &0 & 0&0
\end{tabular} &
\begin{tabular}{cc||c|c|c}
\multirow{7}{*}{\includegraphics[width=1.1cm,
    angle=0]{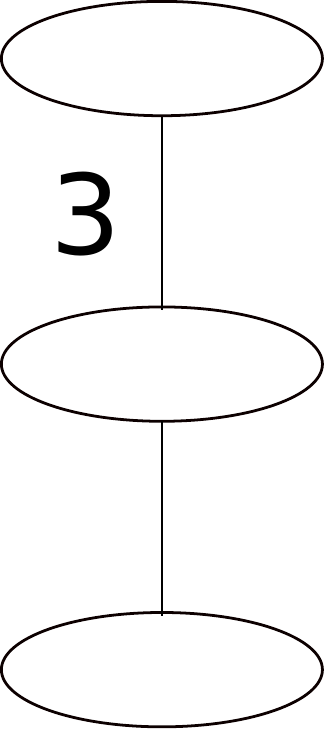}}&
$\kappa,\epsilon\backslash s$& 0 & 1 & 2
\\ \cline{2-5}
&  0  & 6& 6& 18
\\ & 1  &4 &4 & 12
\\ & 2  &2 & 2&6
\\ & 3  &0 & 0&0
\\ \cline{2-5}
& 0  & 0& 0&0
\\&  1  & 0& 0&0
\end{tabular} &
\begin{tabular}{cc||c|c|c}
\multirow{7}{*}{\includegraphics[width=1.1cm,
    angle=0]{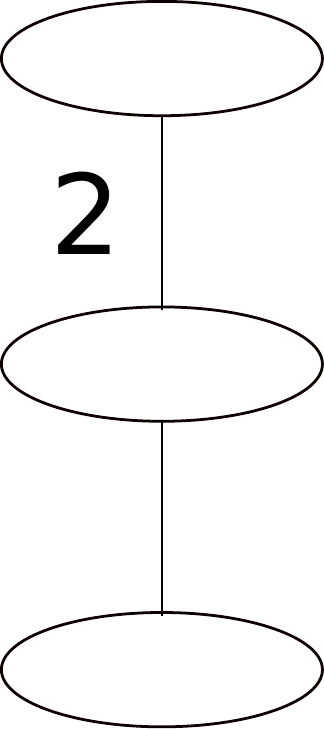}}&
$\kappa,\epsilon\backslash s$& 0 & 1 & 2
\\ \cline{2-5}
&  0  & 0& 0&0
\\ & 1  &0 & 0&0
\\ & 2  & 0& 0&0
\\ & 3  & 0& 0&0
\\ \cline{2-5}
& 0  & 0& 0&0
\\&  1  &0 & 0&0
\end{tabular} 
\\  $\sum \mu^\C(\D,m) =  80$ &  $\sum \mu^\C(\D,m) = 54$ &  $\sum \mu^\C(\D,m) =  120$

\\\\\begin{tabular}{cc||c|c|c}
\multirow{7}{*}{\includegraphics[width=1.1cm,
    angle=0]{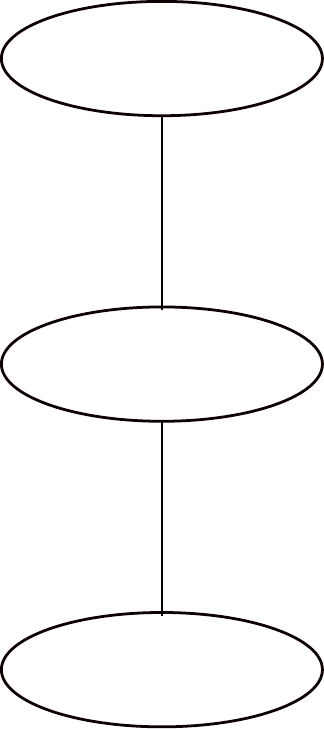}}&
$\kappa,\epsilon\backslash s$& 0 & 1 & 2
\\ \cline{2-5}
&  0  & 60& 60&60
\\ & 1  & 16& 16&16
\\ & 2  & 4& 4&4
\\ & 3  & 0& 0&0
\\ \cline{2-5}
& 0  & 0& 0&0
\\&  1  &0 & 0&0
\end{tabular} &
\begin{tabular}{cc||c|c|c}
\multirow{7}{*}{\includegraphics[width=1.1cm,
    angle=0]{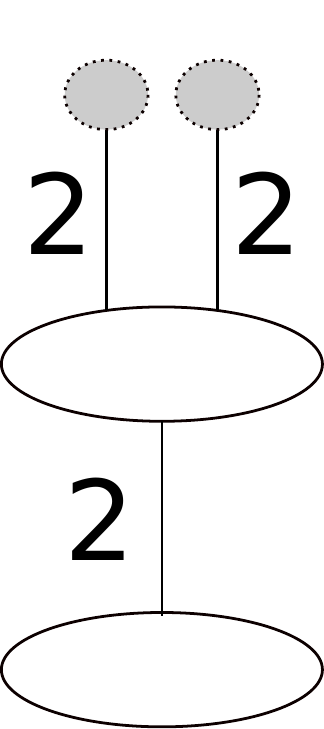}}&
$\kappa,\epsilon\backslash s$& 0 & 1 & 2
\\ \cline{2-5}
&  0  & 0& 0&0
\\ & 1  &0 & 0&0
\\ & 2  &0 & 0&0
\\ & 3  &0 & 0&0
\\ \cline{2-5}
& 0  & 64& 32&16
\\&  1  & 0& 0&0
\end{tabular} &
\begin{tabular}{cc||c|c|c}
\multirow{7}{*}{\includegraphics[width=1.1cm,
    angle=0]{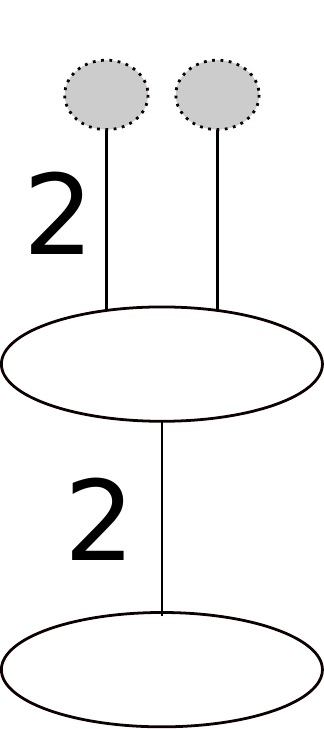}}&
$\kappa,\epsilon\backslash s$& 0 &  1& 2
\\ \cline{2-5}
&  0  & 0& 0&0
\\ & 1  &0 & 0&0
\\ & 2  & 0& 0&0
\\ & 3  & 0& 0&0
\\ \cline{2-5}
& 0  & 0& 0&0
\\&  1  &0 & 0&0
\end{tabular} 
\\  $\sum \mu^\C(\D,m) =  60$ &  $\sum \mu^\C(\D,m) =64 $&  $\sum \mu^\C(\D,m) = 192 $

\\\\\begin{tabular}{cc||c|c|c}
\multirow{7}{*}{\includegraphics[width=1.1cm,
    angle=0]{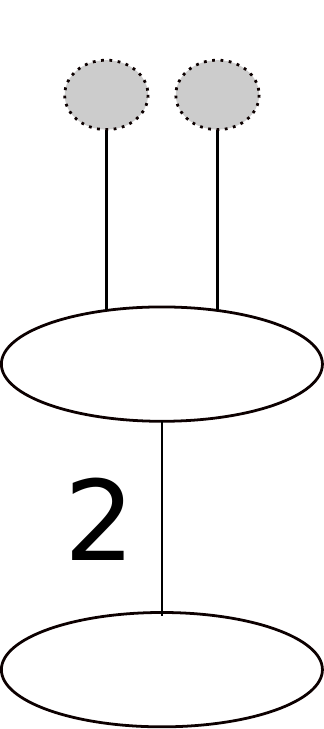}}&
$\kappa,\epsilon\backslash s$& 0 & 1& 2
\\ \cline{2-5}
&  0  & 0& 0&0
\\ & 1  &0 & 0&0
\\ & 2  &0 & 0&0
\\ & 3  &0 & 0&0
\\ \cline{2-5}
& 0  & 0& 0&0
\\&  1  &0 & 0&0
\end{tabular} &
\begin{tabular}{cc||c|c|c}
\multirow{7}{*}{\includegraphics[width=1.1cm,
    angle=0]{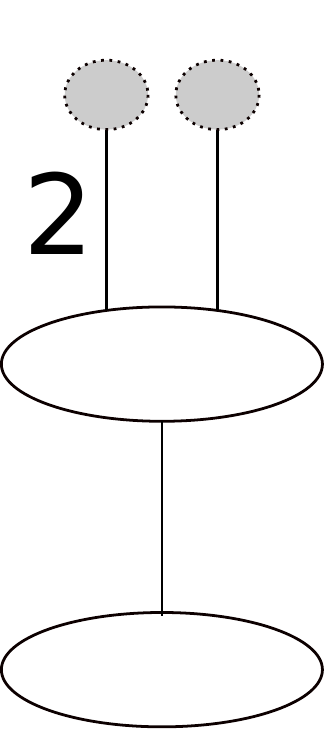}}&
$\kappa,\epsilon\backslash s$& 0 & 1 & 2
\\ \cline{2-5}
&  0  & 0& 0&0
\\ & 1  &0 & 0&0
\\ & 2  &0 & 0&0
\\ & 3  &0 & 0&0
\\ \cline{2-5}
& 0  & 0& 0&0
\\&  1  &0 & 0&0
\end{tabular} &
\begin{tabular}{cc||c|c|c}
\multirow{7}{*}{\includegraphics[width=1.1cm,
    angle=0]{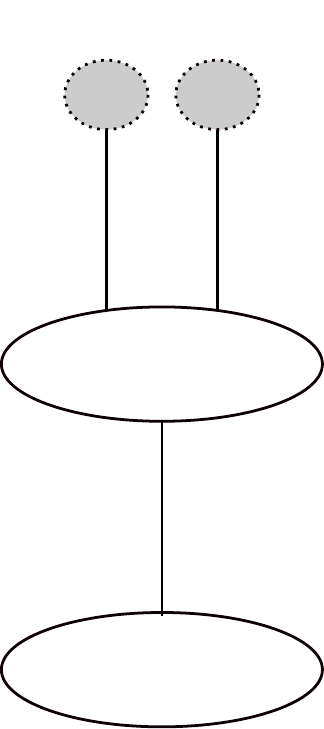}}&
$\kappa,\epsilon\backslash s$& 0 & 1 & 2
\\ \cline{2-5}
&  0  & 66& 66&66
\\ & 1  & 28& 28&28
\\ & 2  & 6& 6&6
\\ & 3  & 0& 0&0
\\ \cline{2-5}
& 0  & 0& 0&0
\\&  1  &0 & 0&0
\end{tabular} 
\\  $\sum \mu^\C(\D,m) =  120$ &  $\sum \mu^\C(\D,m) = 48$&  $\sum \mu^\C(\D,m) =  66$

\end{tabular}
\caption{$(6[D]-2\sum_{i=1}^6 [E_i])$-floor diagrams of genus 0 and type $(0,0)$}
\label{fig:FD6}
\end{figure}

\begin{figure}[h]
\centering
\begin{tabular}{ccc}
\begin{tabular}{cc||c|c|c}
\multirow{7}{*}{\includegraphics[width=1.6cm,
    angle=0]{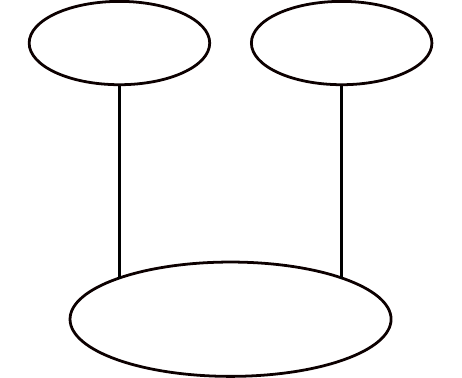}}&
$\kappa,\epsilon\backslash s$& 0 & 1 & 2
\\ \cline{2-5}
&  0  & 60& 60& 20
\\ & 1  &  24 & 24& 8
\\ & 2  &  12& 12& 4
\\ & 3  & 0& 0&0
\\ \cline{2-5}
& 0  & 0& 0&0
\\&  1  &0 & 0&0
\end{tabular} &
\begin{tabular}{cc||c|c|c}
\multirow{7}{*}{\includegraphics[width=1.6cm,
    angle=0]{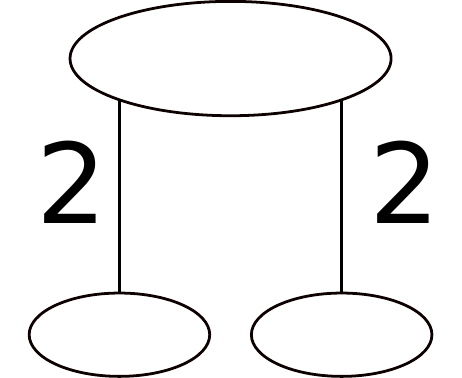}}&
$\kappa,\epsilon\backslash s$& 0 & 1 & 2
\\ \cline{2-5}
&  0  & 0& 0& 8
\\ & 1  &0 &0 & 8 
\\ & 2  &0 & 0&8
\\ & 3  &0 & 0&8
\\ \cline{2-5}
& 0  & 48& 8& -4
\\&  1  &48 &8 & -4
\end{tabular} &
\begin{tabular}{cc||c|c|c}
\multirow{7}{*}{\includegraphics[width=1.6cm,
    angle=0]{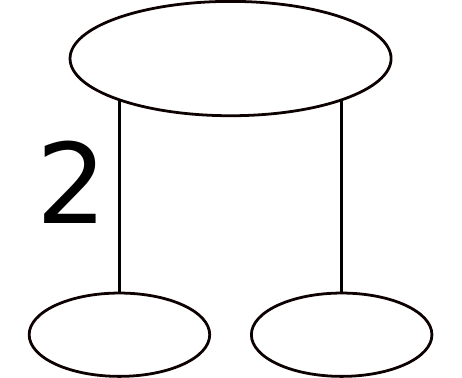}}&
$\kappa,\epsilon\backslash s$& 0 & 1 & 2
\\ \cline{2-5}
&  0  & 0& 0&0
\\ & 1  &0 & 0&0
\\ & 2  & 0& 0&0
\\ & 3  & 0& 0&0
\\ \cline{2-5}
& 0  & 0& 0&0
\\&  1  &0 & 0&0
\end{tabular} 
\\  $\sum \mu^\C(\D,m) = 60 $ &  $\sum \mu^\C(\D,m) = 48$&  $\sum \mu^\C(\D,m) =  144$

\\\\\begin{tabular}{cc||c|c|c}
\multirow{7}{*}{\includegraphics[width=1.6cm,
    angle=0]{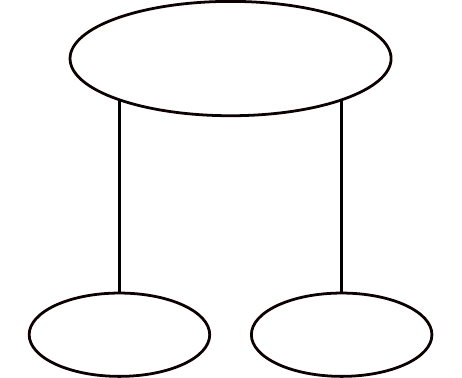}}&
$\kappa,\epsilon\backslash s$& 0 & 1 & 2
\\ \cline{2-5}
&  0  & 90& 30& 30
\\ & 1  & 36& 12& 16 
\\ & 2  & 6& 2& 10
\\ & 3  & 0& 0& 12
\\ \cline{2-5}
& 0  & 0& 0&0
\\&  1  &0 &0 &0
\end{tabular} &
\begin{tabular}{cc||c|c|c}
\multirow{7}{*}{\includegraphics[width=1.1cm,
    angle=0]{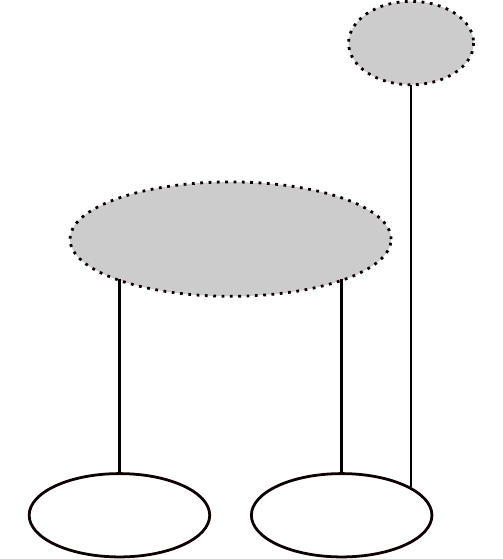}}&
$\kappa,\epsilon\backslash s$& 0 & 1 & 2
\\ \cline{2-5}
&  0  & 120& 48& 24
\\ & 1  & 80& 32&16
\\ & 2  & 40& 16&8
\\ & 3  & 0& 0&0
\\ \cline{2-5}
& 0  & 0& 0&0
\\&  1  &0 & 0&0
\end{tabular} &
\begin{tabular}{cc||c|c|c}
\multirow{7}{*}{\includegraphics[width=1.1cm,
    angle=0]{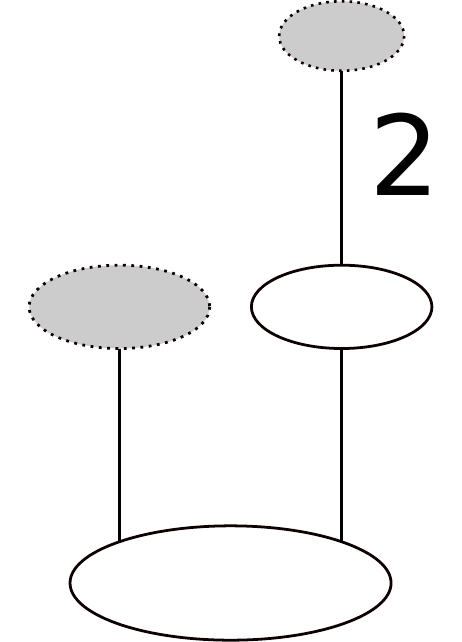}}&
$\kappa,\epsilon\backslash s$& 0 & 1 & 2
\\ \cline{2-5}
&  0  & 0& 0&0
\\ & 1  &0 & 0&0
\\ & 2  &0 & 0&0
\\ & 3  &0 & 0&0
\\ \cline{2-5}
& 0  & 0& 0&0
\\&  1  &0 & 0&0
\end{tabular} 
\\  $\sum \mu^\C(\D,m) =  90$ &  $\sum \mu^\C(\D,m) = 120$&  $\sum \mu^\C(\D,m) =  96$

\\\\ &
\begin{tabular}{cc||c|c|c}
\multirow{7}{*}{\includegraphics[width=1.1cm,
    angle=0]{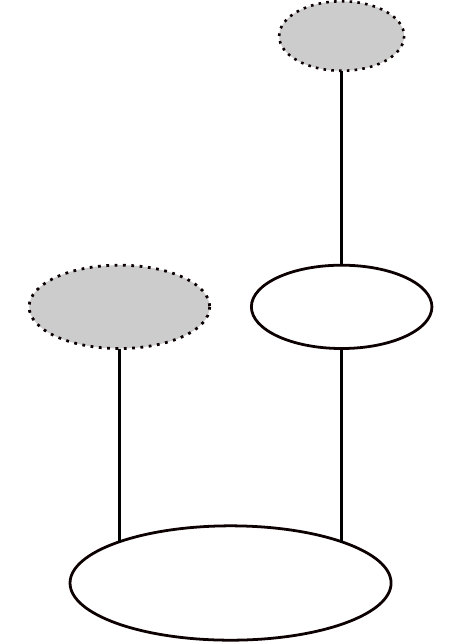}}&
$\kappa,\epsilon\backslash s$& 0 & 1 & 0
\\ \cline{2-5}
&  0  & 120& 120& 60
\\ & 1  & 48& 48& 24
\\ & 2  & 8& 8& 4
\\ & 3  & 0& 0& 0
\\ \cline{2-5}
& 0  & 0& 0&0
\\&  1  &0 & 0&0
\end{tabular}& 
\\  &$\sum \mu^\C(\D,m) =  120$&

\end{tabular}
\caption{$(6[D]-2\sum_{i=1}^6 [E_i])$-floor diagrams of genus 0 and type $(0,0)$, continued}
\label{fig:FD62}
\end{figure}

\subsection{Enumeration of real rational curves}\label{sec:real FD}
Let $(\mathcal D,m)$  be a $d$-marked
 floor diagram of  genus $0$, and let
 $\alpha^\Re,\beta^\Re,\alpha^\Im,\beta^\Im\in\Z_{\ge 0}^\infty$ such that
$$I\alpha^\Re+I\beta^\Re +2I\beta^\Im+2I\beta^\Im=d\cdot [E]. $$
Let $\zeta=d\cdot [D] -1 +|\alpha^\Re|+|\beta^\Re| +2|\alpha^\Im|+2|\beta^\Im|$, and
 choose two  integers $r,s\ge 0$ satisfying
 $\zeta =r +2s +|\alpha^\Re| +2|\alpha^\Im|$. 

The set $\{i,i+1\}\subset A_0$ is a called \textit{$s$-pair} if
 either $i=|\alpha^\Re|+2k-1$ with $1\le k\le |\alpha^\Im|$, or
$i=|\alpha^\Re| +2|\alpha^\Im|+2k-1$ with  $1\leq k \leq s$.
 Denote by
$\Im (m,s)$ the union of all the $s$-pairs $\{i,
i+1\}$ where $m(i)$ is not adjacent to  $m(i+1)$.
Let $\psi_{0,s}:\{1,\ldots, \zeta\} \rightarrow \{1,\ldots, \zeta\} $ be
the bijection
defined by  $\psi_{0,s} (i)=i$ if $i\notin
 \Im(m,s)$, and by $\psi_{0,s}(i)=j$ if
$\{i,j\}$ is a $s$-pair  contained in $ \Im(m,s)$.
Note that  $\psi_{0,s}$ is an involution, and that $\psi_{0,0}=Id$.

Now chose an integer $0\le \kappa\le \frac{n}{2}$ such that
$d\cdot [E_{2i-1}]=d\cdot [E_{2i}]$ for $i=1,\ldots,\kappa$.
For $i=2\kappa+1,\ldots, n$,  define $\psi_{i,\kappa}$ to be the identity on
$A_i$.
For  $i=1,\ldots, \kappa$, choose a bijection 
$\psi_{2i-1,\kappa}: A_{2i-1}\to A_{2i}$, 
and   define $\psi_{2i,\kappa}=\psi_{2i-1,\kappa}^{-1}$.
Finally  define the involution  $\rho_{s,\kappa}:\bigcup_{i=0}^n A_i\to
\bigcup_{i=0}^n A_i$ by setting $\rho_{s,\kappa|A_0}=
\psi_{0,s}$, and $\rho_{s,\kappa|A_i}=
\psi_{i,\kappa}$ for  $i=1,\ldots, n$.
Note that $\rho_{s,\kappa}=Id$ if $s=\kappa=0$.

\begin{defi}\label{defi real}
A $d$-marked floor diagram $(\mathcal D,m)$ of
genus 0
 is called $(s,\kappa)$-real if the two marked floor diagrams
$(\mathcal{D},m)$ and
$(\mathcal{D},m\circ \rho_{s,\kappa})$ are
equivalent.

A $(s,\kappa)$-real $d$-marked floor diagram $(\mathcal D,m)$ is said
to be of type $(\alpha^\Re,\beta^\Re,\alpha^\Im,\beta^\Im)$ if
\begin{enumerate}
\item the marked floor
diagram  $(\mathcal D,m)$ is of type
$(\alpha^\Re+2\alpha^\Im,\beta^\Re+2\beta^\Im)$; 
\item exactly  $2\beta^\Im_j$ edges of weight $j$ are contained in
  $Edge^{\infty}(\D)\cap 
m\left(\Im(m,s)\right)$ for any $j\ge 1$.
\end{enumerate}
\end{defi}

The set of $(s,\kappa)$-real
$d$-marked floor 
diagrams of genus $0$ and
of type $(\alpha^\Re,\beta^\Re,\alpha^\Im,\beta^\Im)$
is denoted by $\Phi^{\alpha^\Re,\beta^\Re,\alpha^\Im,\beta^\Im}(d,s,\kappa)$.
Note that the involution  $\rho_{s,\kappa}$ induces an involution,
 denoted by  $\rho_{m,s,\kappa}$, on the underlying floor
diagram of a real marked floor diagram.

The set of pairs of floors of
$\D$
 exchanged by $\rho_{m,s,\kappa}$ is denoted by $Vert_\Im(\D)$. The
 subset of $Vert_\Im(\D)$ formed by floors of degree $i$
is denoted by 
$Vert_{\Im,i}(\D)$.
To a pair  $\{v,v'\}\in Vert_\Im(\D)$, we associate the  following
numbers:
\begin{itemize}
\item $o_v$ is the sum of
the degree of $v$ and the number of its adjacent edges
which are in their turn adjacent to  $m\left(\bigcup_{i=2\kappa+1}^n
A_i\right)$;
\item $o'_v$ is the number of edges of weight  $2+ 4l$ adjacent to $v$.
\end{itemize}

The set of edges of $\D$ which are fixed (resp. exchanged) by
 $\rho_{m,s,\kappa}$ is
 denoted  by $Edge_\Re(\D)$ (resp. $Edge_\Im(\D)$). The number 
of edges contained in 
$m\left(\{\zeta -r+1,\ldots,\zeta\}\right)$ is denoted
 by $r_m$,
 and  the
 number of edges contained in 
$Edge_\Re(\D)\cap m\left(\{1,\ldots,\zeta-r\}\right)$
is denoted
 by $r'_m$.

If $n=2\kappa$ and  $\epsilon\in\{0,1\}$,  
a marked floor diagram $(\D,m)$ is said to be
\emph{$\epsilon$-sided} if  any edge in $Edge_\Re(\D)$ 
is of even weight, and, if 
$\epsilon=1$, any floor of degree 1 is contained in a pair
in $Vert_\Im(\D)$.
It is said to be 
\emph{significant} if 
it satisfies the three following additional conditions:
\begin{itemize}
\item any edge in $Edge_\Im(\D)\setminus m\left(\bigcup_{i=1}^nA_i\right)$ 
 is of even weight;
\item  any edge
 in $Edge_\Re(\D)\setminus Edge^\infty(\D)$ has   weight $2+4l$;
\item for any $\{v,v'\}\in Vert_{\Im}(\D)$ and  any $i=1,\ldots, n$, the vertex
$ v$ is adjacent to an 
   edge adjacent to  $m\left( A_i\right)$ if and only if so is $v'$.
\end{itemize}

Finally define 
$$E(\D)=\left(Edge(\D)\setminus Edge^\infty(\D)\right)\cap
m\left(\{1,\ldots,\zeta-r\} \right)
\quad \mbox{and}\quad 
\beta^\Re_{even}=\sum_{j\ge 0}\beta^\Re_{2j}.$$

\begin{defi}\label{def:real mult}
Let $(\D,m)$ be a  $(s,\kappa)$-real
$d$-marked floor 
diagram.
 The $(s,\kappa)$-real multiplicity of $(\D,m)$, denoted by
$\mu^\R_{s,\kappa}(\mathcal D,m)$, is defined by
$$\mu^\R_{s,\kappa}(\mathcal D,m)= 2^{\beta^\Re_{even}}\ 
I^{\beta^\Im} \prod_{\{v,v'\}\in Vert_\Im(\D)}(-1)^{o_{v}}
\prod_{e\in  E( \D)}w(e) $$
if  
$m(\Im(m,s))\bigcup Edge^\infty(\D)$ contains all edges of $\D$ of
even weight, and by 
$$\mu^\R_{s,\kappa}(\mathcal D,m)=0 $$
otherwise.

\medskip
If in addition $2\kappa=n$ and $(\D,m)$ is $\epsilon$-sided, 
we define an additional
$(s,\kappa)$-real multiplicities of $(\D,m)$ as follows
$$\nu^{\R,\epsilon}_{s}(\mathcal D,m)=(-1)^{\epsilon|Vert_{\Im,1}(\D)|} 
\ 2^{ 2r_m-r'_m+\beta^\Re_{even} }
\ I^{\beta^\Im}\prod_{\{v,v'\}\in Vert_{\Im,2}(\D)}(-1)^{o'_v}
\prod_{e\in E(\D)}w(e) 
$$
if $(\D,m)$ is significant,
and  by
$$\nu^{\R,\epsilon}_{s}(\mathcal D,m)=0 $$
otherwise.

\end{defi}

Next, choose $\alpha^\Re,\beta^\Re,\alpha^\Im,\beta^\Im$ and define
$$FW^{\alpha^\Re,\beta^\Re,\alpha^\Im,\beta^\Im}_{\X_n(\kappa)}(d,s)=\sum
\mu^\R_{s,\kappa}(\mathcal D,m) $$
where the sum is taken over all 
$(s,\kappa)$-real $d$-marked floor
diagrams of type $(\alpha^\Re,\beta^\Re,\alpha^\Im,\beta^\Im)$.

If in addition  $n=2\kappa$, 
 define the following numbers:
$$FW^{\alpha^\Re,\beta^\Re,\alpha^\Im,\beta^\Im}_{\X_n(\kappa),\epsilon}(d,s)=\sum
\mu^\R_{s,\kappa}(\mathcal D,m) $$
where the sum is taken over all $\epsilon$-sided 
$(s,\kappa)$-real $d$-marked floor
diagrams of type $(\alpha^\Re,\beta^\Re,\alpha^\Im,\beta^\Im)$, and 
$$FW_{\X_n(\kappa),\epsilon,\epsilon}^{\alpha^\Re,\beta^\Re,\alpha^\Im,\beta^\Im}(d,s)=\sum
\nu^{\R,\epsilon}_{s}(\mathcal D,m) $$
where the sum is taken over all significant $\epsilon$-sided
$(s,\kappa)$-real $d$-marked
 floor
diagrams of type $(\alpha^\Re,\beta^\Re,\alpha^\Im,\beta^\Im)$.
Note that by definition we have
$$FW_{\X_n(\kappa)}^{\alpha^\Re,\beta^\Re,\alpha^\Im,\beta^\Im}(d,s)=
FW_{\X_n(\kappa),\epsilon}^{\alpha^\Re,\beta^\Re,\alpha^\Im,\beta^\Im}(d,s)=
FW_{\X_n(\kappa),\epsilon,\epsilon}^{\alpha^\Re,\beta^\Re,\alpha^\Im,\beta^\Im}(d,s)=0$$
if $d\cdot [E]\ne I\alpha^\Re+I\beta^\Re+2I\alpha^\Im+2I\beta^\Im$.

\begin{lemma}\label{cor:vanish}
Given $n=2\kappa$, $\varepsilon\in\{0, 1\}$, and $r\ge |\beta^\Re|+2$,
we have 
$$FW^{\alpha^\Re,\beta^\Re,\alpha^\Im,\beta^\Im}_{\X_n(\kappa),\epsilon}(d,s)=0. $$
\end{lemma}
\begin{proof}
Let $(\D,m)$ be an $\epsilon$-sided 
$(s,\kappa)$-real $d$-marked floor
diagram of type $(\alpha^\Re,\beta^\Re,\alpha^\Im,\beta^\Im)$.
Since $\D$ is a tree, its subgraph formed by  elements fixed by
$\rho_{m,s,\kappa}$ is connected. In particular if $r\ge |\beta^\Re|+2$, the set
$Edges^\Re(\D)\setminus Edge^\infty(\D)$ is not empty. Since any edge
in this set has an even weight, the lemma follows from Definition
\ref{def:real mult}. 
\end{proof}

Next theorem relates the three series of numbers $FW$ 
to actual
enumeration of real 
curves in $\X_n(\kappa)$. Recall that when $n=2\kappa$, the connected
component of $\R\X_n(\kappa)\setminus \R E$ with Euler characteristic
$\epsilon$ is denoted by $\widetilde L_\epsilon$.

\begin{thm}\label{WFD}
Let  $\zeta_0,r,s,\kappa \ge 0$ be some integers such that  $\zeta_0=r+2s$.
Then
there exists a generic $(E,s)$-compatible
configuration $\x^\circ$ of 
$\zeta_0$ points in $\X_n$ such that:
 \begin{enumerate}
\item 
 for any  $d\in H_2(\X_n;\Z)$ with $d\cdot [D]\ge 1$,  any
 $\alpha^\Re,\beta^\Re,\alpha^\Im,\beta^\Im\in\Z_{\ge 0}^\infty$ such that
$$d\cdot [D]-1 +|\beta^\Re|+2|\beta^\Im|=\zeta_0\quad \mbox{and}\quad 
d\cdot [E]=I\alpha^\Re+ I\beta^\Re+2I\alpha^\Im +2I\beta^\Im, $$ 
and any generic real configuration $\x_E\subset E$ of type 
$(\alpha^\Re,\alpha^\Im)$, one has
$$W^{\alpha^\Re,\beta^\Re,\alpha^\Im,\beta^\Im}_{\X_n(\kappa)}(d,s,\x^\circ\sqcup \x_E)=FW^{\alpha^\Re,\beta^\Re,\alpha^\Im,\beta^\Im}_{\X_n(\kappa)}(d,s).$$

\item If moreover $n=2\kappa$,
and $\x^\circ$ is $(E,s,\widetilde L_\epsilon)$-compatible,
then
$$W^{\alpha^\Re,\beta^\Re,\alpha^\Im,\beta^\Im}_{\X_n(\kappa),\widetilde L_\epsilon,\R\X_n(\kappa)}(d,s,\x^\circ\sqcup \x_E)=FW^{\alpha^\Re,\beta^\Re,\alpha^\Im,\beta^\Im}_{\X_n(\kappa),\epsilon}(d,s),$$
and
$$W^{\alpha^\Re,\beta^\Re,\alpha^\Im,\beta^\Im}_{\X_n(\kappa),\widetilde L_\epsilon,,\widetilde L_\epsilon}
(d,s,\x^\circ\sqcup \x_E)=FW^{\alpha^\Re,\beta^\Re,\alpha^\Im,\beta^\Im}_{\X_n(\kappa),\epsilon,\epsilon}(d,s) .$$
\end{enumerate}
\end{thm}

\begin{exa}
If  $n\le 5$,  Theorem \ref{WFD}(1) computes
Welschinger invariants of $X_n$ equipped with a standard real structure.
In particular, applying Theorem \ref{WFD} with $n=0$, one verifies that 
$$W_{\C P^2}([D],s)=W_{\C P^2}(2[D],s)=1\quad \mbox{and} \quad W_{\C P^2}(3[D],s)=8-2s. $$ 
\end{exa}

\begin{exa}\label{ex:WFD 4 and 6}
Fix $n=6$ and $\zeta_0=5$. Given $0\le s\le 2$, let $\x^\circ_s$ be a
configuration whose 
existence is attested by Theorem \ref{WFD} with $r=6-2s$. 
Using Figures \ref{fig:ex FD} , \ref{fig:FD42}, \ref{fig:FD6},
and \ref{fig:FD62}, one computes all numbers
$W^{0,\beta^\Re,0,\beta^\Im}_{\X_6(\kappa)}(d_k,s,\x^\circ_s)$ for the classes
$d_k=6[D]-2\sum_{i=1}^6[E_i] -k[E]$ with $k=0,1,2$, as well as the numbers
$W^{0,\beta^\Re,0,\beta^\Im}_{\X_6(3),\widetilde L_\epsilon,\widetilde L_\epsilon}(d_0,s,\x^\circ_3)$.
In the case $k=2$, this value is 1 for $(\beta^\Re,\beta^\Im)$ given in Table
\ref{tab:comp41}a, and 0 otherwise.
In the case $k=1$, 
the numbers
$W^{0,\beta^\Re,0,\beta^\Im}_{\X_6(\kappa)}(d_1,s,\x^\circ_s)$ vanish for all
values of $\beta^\Re$ and $\beta^\Im$ not listed in Table \ref{tab:comp41}b.
In the case $k=0$, all $(s,3)$-real diagrams contributing to 
$W^{0,0,0,0}_{\X_6(3)}(d_0,s,\x^\circ_s)$
are
$\epsilon$-sided with $\epsilon\in\{0,1\}$, so we have
$W^{0,0,0,0}_{\X_6(3)}(d_0,s,\x^\circ_s)=
W^{0,0,0,0}_{\X_6(3),\widetilde L_\epsilon,\R \X_6(3)}(d_0,s,\x^\circ_s)$.  
  In Figures \ref{fig:FD42},
 \ref{fig:FD6}, and \ref{fig:FD62}, beside all floor diagrams  
is written the sum of  $(s,\kappa)$-multiplicity of all  corresponding 
$(s,\kappa)$-real marked floor
diagrams of type $(0,\beta^\Re,0,\beta^\Im)$ with this underlying floor diagram. 
The numbers $W^{0,0,0,0}_{\X_6(\kappa)}(d_0,0,\x^\circ_0)$ were first
computed in {\cite[Proposition 3.1]{Br20}}.

\begin{table}[!h]
\begin{center}
\begin{tabular}{ccc}
\begin{tabular}{ c |c|c}
 $s  $  &  $\beta^\Re$&$\beta^\Im$
\\\hline
  $0$   &  $4u_1$ & 0   
\\\hline
$1$ &  $2u_1$&  $u_1$  
\\\hline
$2$ & 0 & $2u_1$ 
\end{tabular}
&\hspace{5ex} &
\begin{tabular}{ c|c|c| c|c|c}
 \multicolumn{2}{c|}{$s \backslash \kappa$}  &0& 1& 2& 3
\\\hline
  $0$ & $\beta^\Re=2u_1$  &  236&  140 & 76 & 36 
\\\hline
\multirow{2}{*}{$1$}& $\beta^\Re=2u_1$  & 80   & 50   &28  & 14 
\\\cline{2-6}& $\beta^\Im=u_1$  &  62   & 28  & 10  &0
\\\hline
$2$ & $\beta^\Im=u_1$&   74 &   36 & 14 & 0
\end{tabular}
\\
\\ a) $W^{0,\beta^\Re,0,\beta^\Im}_{\X_6(\kappa)}(d_2,s,\x^\circ_s)=1$ & & 
b) $W^{0,\beta^\Re,0,\beta^\Im}_{\X_6(\kappa)}(d_1,s,\x^\circ_s)$
\end{tabular}
\end{center}
\caption{}
\label{tab:comp41}
\end{table}

\begin{table}[!h]
\begin{center}
\begin{tabular}{ c|c| c|c|c||c|c}
 $s \backslash \kappa,\epsilon $  & $0$& $1$& $2$& $3$& $0$&$1$
\\\hline
  $0$   & 522 & 236 & 78 & 0 & 160 & 96
\\\hline
$1$ &390 &  164&  50 & 0 &  64& 32 
\\\hline
$2$ & 286 & 128 &  50  & 20 & 24 & 8
\\ 
\end{tabular}

\begin{tabular}{c}
\end{tabular}
\end{center}
\caption{$W^{0,0,0,0}_{\X_6(\kappa)}(d_0,s,\x^\circ_s)$ and 
$W^{0,0,0,0}_{\X_6(3),\widetilde L_\epsilon,\widetilde L_\epsilon}(d_0,s,\x^\circ_s)$}
\label{tab:comp62}
\end{table}

\end{exa}

\section{Absolute invariants of $X_6$}\label{sec:X6}
\subsection{Gromov-Witten invariants}
When $n=6$, Theorem \ref{NFD} combined with 
{\cite[Theorem 4.5]{Vak2}} allows one to computes
Gromov-Witten invariants of $X_6$.

\begin{thm}\label{thm:NFD2}
For any  $d\in H_2(X_6;\Z)$ such that $d\cdot [D]\ge 1$, and
 any genus $g\ge 0$, one has
$$GW_{X_6}(d,g)=\sum_{k\ge 0}\left(\begin{array}{c}
 d\cdot [E] +2k \\ k\end{array}\right) \sum \mu^\C(\mathcal D,m)$$
where the second sum is taken over all $(d-k[E])$-marked floor diagrams of 
genus $g$ and type $(0,(d\cdot [E] +2k)u_1)$.
\end{thm}

\begin{exa}
Theorem \ref{thm:NFD2} together with Examples \ref{ex1} and \ref{ex:NFD 4 and 6}
implies that 
$$GW_{X_6}(2c_1(X_6),0)=2002 + {2 \choose 1}\times 616 +  {4 \choose
  2}\times 1=3240.$$
 Performing analogous computations
in genus up to $4$, we obtain the value listed in Table \ref{tab:comp X6}. 
\begin{table}[!h]
\begin{center}
\begin{tabular}{ c|c| c|c|c|c}
 $g $  & $0$& $1$& $2$& $3$& $4$
\\\hline
 $GW_{X_6}(2c_1(X_6),g)$   & 3240  & 1740 & 369& 33 & 1 

\end{tabular}
\end{center}
\caption{$GW_{X_6}(2c_1(X_6),g)$}\label{tab:comp X6}
\end{table}
The value in the rational case has been
first  computed  by G\"ottsche and
Pandharipande in {\cite[Section 5.2]{PanGot98}}. The cases of higher genus have been first
treated in \cite{Vak2}.
\end{exa}

\subsection{Welschinger invariants}
Applying Theorem \ref{WFD} with $n=6$, one can also  compute
Welschinger invariants of $X_6$ with any real
structure.
Denote by $X_{6}(\kappa)$ with $\kappa=0,\ldots,4$ the surface
$X_6$ equipped with the real structure such that
$$\chi(\R X_{6}(\kappa))=-5+2\kappa.$$
Denote also by $L_\epsilon$
the connected
component of $\R X_6(4)$
with Euler characteristic
$\epsilon$. 
Next theorem is an immediate corollary of Theorems \ref{WFD} and
{\cite[Theorem 2.2]{Br20}} or {\cite[Theorem 4]{BP14}} (see also Section \ref{sec:WX7} or
\cite{IKS13} for a proof in the algebraic setting). 

\begin{thm}\label{thm:W X6}
For any  $d\in H_2(X_6;\Z)$ such that $d\cdot [D]\ge 1$, any $r,s\ge 0$ such
that $c_1(X_6)\cdot d-1=r+2s$, any $\kappa\in \{0,\ldots, 3\}$,
and any
$\epsilon\in\{0,1\}$, one has 
 \[\begin{aligned}
& W_{X_{6}(\kappa)}(d,s)=\sum_{k\ge 0}\ \ 
\sum_{k=r'+2s'}\ \ 
\sum_{\beta^\Re_1+2\beta^\Im_1=d\cdot [E]+2k}
{\beta^\Re_1 \choose r'}{\beta^\Im_1\choose
  s'}  FW_{\X_6(\kappa)}^{0,\beta^\Re_1u_1,0,\beta^\Im_1u_1}(d,s),\\  
& W_{X_{6}(\kappa+1)}(d,s)= \sum_{k\ge 0}(-2)^k
FW_{\X_6(\kappa)}^{0,0,0,ku_1}(d,s) \quad \mbox{if }\kappa\le 2,\\ 
&W_{X_{6}(4),L_{1+\epsilon}, \R X_6(4)}(d,s)= \sum_{k\ge 0}(-2)^k
FW_{\X_6(3),\epsilon}^{0,0,0,ku_1}(d,s)\quad \forall \epsilon\in\{0,1\},\\
&W_{X_{6}(4),L_{1+\epsilon},L_{1+\epsilon}}(d,s)=
FW_{\X_6(3),\epsilon,\epsilon}^{0,0,0,0}(d,s)\quad \forall \epsilon\in\{0,1\}.
\end{aligned}\]

\end{thm}

Theorem \ref{thm:W X6} has the two following  corollaries. 
\begin{cor}\label{cor:X6 1}
For any $d\in H_2(X_6;\Z)$, one has
$$W_{X_{6}(4), L_1,L_1}
(d,0)\ge W_{X_{6}(4), L_2,L_2}
(d,0)\ge 0.$$
Moreover both invariants are divisible by $4^{\left[\frac{d\cdot [D]}{2}\right]-1}$.
\end{cor}
\begin{proof}
The first assertion follows immediately from Theorem \ref{thm:W X6},
 Definition \ref{defi real}, and the fact that 
$\psi_{0,0}=Id$. The second assertion follows from Definition
\ref{defi real} and the observation that any marked floor diagram which contributes to 
 $FW_{\X_6(4),\epsilon,\epsilon}^{0,0,0,0}(d,s)$ has at least
$\left[\frac{d\cdot [D]}{2}\right]$ vertices,  hence at least 
$\left[\frac{d\cdot  [D]}{2}\right]-1$ edges in $Edge(\D)\setminus Edge^\infty(\D)$. 
\end{proof}
The non-negativity of $W_{X_{6}(4), L_1,L_1}(d,0)$ has been first
proved in \cite{IKS11}.
Next corollary is a particular case of {\cite[Proposition
    3.3]{Br20}} and {\cite[Theorem 2]{BP14}}. The proof presented here is slightly
different and easier than the one used in  {\cite[Theorem 2]{BP14}},
which covers a more general situation. 

\begin{cor}\label{cor:X6 2}
For any $d\in H_2(X_6;\Z)$ and any $\epsilon\in\{1,2\}$,  one has
$$W_{X_{6}(4),L_\epsilon,\R X_6(4)}
(d,s)= 0$$
as soon as $r \ge 2$.
\end{cor}
\begin{proof}
This is a consequence of  Theorem \ref{thm:W X6} 
and Lemma \ref{cor:vanish}.
\end{proof}

\begin{exa}
Theorem \ref{thm:W X6} and Example \ref{ex:WFD 4 and 6} imply that 
Welschinger invariants of the surface $X_6$ for  the
 class $2c_1(X_6)$ are the one listed in Table \ref{tab:W X6}.
\begin{table}[!h]
\begin{center}
\begin{tabular}{ccc}
\begin{tabular}{ c|c| c|c|c|c|c}
 $s \backslash \kappa $  & $0$& $1$& $2$& $3$& $4$ & $4$
\\ & & & & & $L=L_1$ & $L=L_2$
\\\hline
  $0$   & 1000 & 522 & 236 & 78 &0 &0
\\\hline
$1$ & 552 &  266 & 108 & 30  & 0&0
\\\hline
$0$ & 288 & 130& 52 & 22& 24&24
\end{tabular}
&\hspace{5ex} &
\begin{tabular}{ c|c|c}
 $s \backslash \epsilon $  &  $1$&$2$
\\\hline
  $0$   & 160 & 96
\\\hline
$1$ & 64 & 32  
\\\hline
$2$ &  24 & 8
\end{tabular}
\\ \\ $W_{X_6(\kappa),L,\R X_6(\kappa)}(2c_1(X_6),s)$ &&
$W_{X_6(4),L_{\epsilon},L_{\epsilon}}(2c_1(X_6),s)$
\\ & &
\end{tabular}
\end{center}
\caption{Welschinger invariants of $X_6$ for the class $2c_1(X_6)$}
\label{tab:W X6}
\end{table}
I first computed the numbers $W_{X_6(0)}(2c_1(X_6),s)$  
 \cite{Br31}. The numbers
  $W_{X_6(\kappa)}(2c_1(X_6),0)$  with $\kappa=1,\ldots,3$, as well as 
 $W_{X_6(4),L_1,L_1}(2c_1(X_6),0)$
have been first computed by Itenberg, Kharlamov and Shustin in
\cite{IKS11}. 
The values
$W_{X_6(4),L_\epsilon,\R X_6(4)}(2c_1(X_6),2)$ have been first computed by
Welschinger in \cite{Wel4}.
\end{exa}

\section{Proof of Theorems and \ref{NFD} and \ref{WFD}}\label{sec:proof}
Here I apply the strategy detailed in Section \ref{sec:CH
  FD}. Recall that $\N=\mathbb P(\N_{E/\X_n}\oplus \C)$, 
$E_\infty=\mathbb P(\N_{E/\X_n}\oplus \{0\})$, and
$E_0=\mathbb P(E\oplus \{1\})$.

Let us go back to the steps $(1)-(3)$ mentioned in 
Section \ref{sec:CH FD}.
The degeneration of $\X_n$ performed in step $(1)$ is standard, see
{\cite[Chapter 5]{F}} for example. Consider
the complex variety $\mathcal Y$ obtained by blowing up $\X_n\times \C$
along $E\times\{0\}$. Then $\mathcal Y$ admits a natural flat
projection $\pi:\mathcal Y\to \C$ such that 
\begin{itemize}
\item $\pi^{-1}(t)=\X_n$ for
$t\ne 0$;
\item  $\pi^{-1}(0)=\X_n \cup \N$, the surfaces $\X_n$ and 
$\N$ intersecting transversely along $E$
  in $\X_n$, and  $E_\infty$ 
in $\N$.
\end{itemize}
 If   $\mathcal E$  denotes the Zariski closure of
 $E\times \C^*$ in $\mathcal Y$, then $\mathcal E\cap\pi^{-1}(0)=E_0.$

\subsection{Degeneration formula applied to $\mathcal Y$}\label{sec:Li deg}
Choose  $\x^\circ(t)$ 
 (resp. $\x_E(t)$) 
a  set of $d\cdot
[D]-1+g+|\beta|$ (resp.  $|\alpha|$) 
holomorphic sections $ \C \to \mathcal Y$ 
(resp.  $\C \to \mathcal E$), and denote  $\x(t)=\x^\circ(t)\sqcup \x_E(t)$. 
Define $\CC^{\alpha,\beta}(d,g,\x(0))$ to be
the set $\left\{\overline f:\overline C\to \X_n \cup \N\right\}$
of limits, as stable maps, 
 of
  maps in $\CC^{\alpha,\beta}(d,g,\x(t))$ as $t$
  goes to $0$, and 
 $\CC_*^{\alpha,\beta}(d,g,\x(0))$ as in Section \ref{sec:relative Xn}.
Recall that $\overline C$ is a connected nodal curve with arithmetic
genus $g$ such that
\begin{itemize}
\item $\x(0)\subset \overline f(\overline C)$;
\item any point $p\in \overline f^{\ -1}(\X_n\cap \N)$ is a node of 
$\overline C$ which is the intersection of two 
irreducible components $\overline C'$ and $\overline C''$ of
$\overline C$, with  $\overline f(\overline
C')\subset \X_n$ and $\overline f(\overline
C'')\subset \N$;

\item if in addition neither $\overline f(\overline
C')$ nor $\overline f(\overline
C'')$ is entirely mapped to $\X_n\cap \N$, then
$p$ appears with the same multiplicity, denoted by $\mu_p$, in both
$\overline f_{|\overline C'}^{\ *}(E)$ and $\overline f_{|\overline C''}^{\ *}(E_\infty)$.
\end{itemize}
If  $\overline C_1,\ldots, \overline C_k$ denote
the irreducible components of $\overline C$ and if none of them
is entirely mapped to $\X_n\cap \N$,   define
$$\mu(\overline f)=\prod_{p\in\overline f^{\ -1}(\X_n\cap \N)} \mu_p
\ \prod_{p\in \x^\circ(0)}|\overline f^{\ -1}(p)| \ \prod_{i=1}^k\left(\frac{1}{|Aut(\overline f_{|\overline C_i})|}\right) . $$
Note that some points in $\overline f^{\ -1}\left(\X_n\cap \N\right)$ might have the same
image by  $\overline f$.
Recall that given  a map $f_t$ in $\CC^{\alpha,\beta}(d,g,\x(t))$ with
$t\ne 0$, the multiplicity $\mu(f_t,\x^\circ(t))$ 
has been defined in Section \ref{sec:relative Xn}.

\begin{prop}\label{prop:degeneration}
Suppose that $\x(0)$ is generic. Then the set
$\CC_*^{\alpha,\beta}(d,g,\x(0))$ is finite, and only depends on
$\x(0)$. Moreover if $\overline f:\overline C\to \X_n
\cup \N$ is an element of  $\CC^{\alpha,\beta}(d,g,\x(0))$, then
no irreducible component of $\overline C$ is
  entirely mapped to $\X_n\cap \N$. 
If in addition we assume that $\overline C$ has no 
 component $\overline C'$ such that $\overline
f(\overline C')=lE_i$ in $\X_n$ with $l\ge 2$, then 
$$\sum \mu(f_t,\x^\circ(t))=\mu(\overline f) $$
where the sum is taken over all morphisms which converge to $\overline
f$ as $t$ goes to $0$.
\end{prop}
\begin{proof}
Thanks to Propositions \ref{prop:shoshu} and \ref{prop:generic N}, the
proof reduces to standard dimension estimations.
The fact no component of $\overline C$ is entirely mapped to $X_n\cap
\N$
 follows from
{\cite[Example 11.4 and Lemma 14.6]{IP00}}.

 Denote by $g_i$ the
arithmetic genus of $\overline C_i$, by $d_i$ the homology class realized by 
$\overline f(\overline C_i)$
 in either $\X_n$ or
$\N$, and by $a_i$ the number of its
intersection points with $\X_n\cap \N$. 
In the case  $\overline f(\overline C_i)\subset \N$,  denote
by $b_i$ the number of its intersection points with $E_0$  not
contained in $\x_E(0)$.
By 
 Propositions \ref{prop:shoshu} and \ref{prop:generic N}, if  $\overline
f(\overline C_i)$ contains $\zeta_i$ points of $\x^\circ(0)$,  we have
$$d_i\cdot [D] -1 +g_i + a_i \ge \zeta_i $$
if $\overline f(\overline C_i)\subset \X_n$, and
$$2d_i\cdot [F] -1 + g_i + a_i+b_i\ge \zeta_i $$
if $\overline f(\overline C_i)\subset \N$.
Moreover, the curves  in $\X_n$ and in $\N$ have to match along $\X_n\cap\N$,
which in regard to Propositions \ref{prop:shoshu} and
\ref{prop:generic N} provide
 $a:=\left|\overline f^{\ -1}\left(\X_n\cap\N\right)\right|$ 
 additional independent
conditions. Altogether we obtain
$$\sum_{\overline f(\overline C_i)\subset \X_n}\left( d_i\cdot [D] -1
+g_i+ a_i \right)
+ \sum_{\overline f(\overline C_i)\subset \N}\left( 2d_i\cdot [F] -1 +
g_i + a_i + b_i \right)\ge d\cdot [D] -1 + g +|\beta|
+a.  $$
We clearly have the equalities
$$\sum_{i=1}^k a_i =2a, \quad \sum
b_i=|\beta|,\quad \mbox{ and } 
\sum_{\overline f(\overline C_i)\subset \X_n}d_i
+ \sum_{\overline f(\overline C_i)\subset \N}\left( d_i\cdot
[F]\right)[E]=d,$$
and an Euler characteristic computation gives
$$a +1- k +\sum g_i=g.$$
Those latter equalities imply that
$$\sum_{\overline f(\overline C_i)\subset \X_n}\left( d_i\cdot [D] -1
+g_i+ a_i \right)
+ \sum_{\overline f(\overline C_i)\subset \N}\left( 2d_i\cdot [F] -1 +
g_i + a_i+ b_i \right) = d\cdot [D] -1 + g +|\beta|
+a.  $$
In particular 
 all the above inequalities are in fact equalities. Together with 
Proposition \ref{prop:shoshu} and \ref{prop:generic N}, this implies that
 the set $\CC_*^{\alpha,\beta}(d,g,\x(0))$ is finite and only depends
 on $\x(0)$. The rest of the proposition follows now from
 {\cite[Theorem 17]{Li04}} (see also {\cite[Lemma 2.19]{Shu13}}).
\end{proof}

\begin{rem}
The assumption that $\overline f$ does not contain a
ramified covering of a $(-1)$-curve in $\X_n$ is needed since
the set $\CC^{0,u_l}(lE_i,0,\emptyset)$ has not the
expected dimension. Again, one could remove this assumption by
 replacing  $\CC^{0,u_l}(lE_i,0,\emptyset)$ by its
virtual fundamental class. 
\end{rem}

From now on, we assume that $\x(0)$ is generic, so that we can apply
Proposition \ref{prop:degeneration}, and we fix an element
$f:\overline C\to \X_n \cup \N$  of
$\CC^{\alpha,\beta}(d,g,\x(0))$.
Next corollary is an immediate consequence of Propositions
\ref{prop:shoshu} and \ref{prop:degeneration}.
\begin{cor}\label{cor:node on E}
 If $p$ and $p'$ are two nodes
of $\overline C$ mapped to the same points of $\X_n\cap \N$, then
$\{\overline f(p)\}= E\cap E_i$ for some $1\le i\le n$.
\end{cor}

\begin{cor}\label{cor:no sEi}
Suppose that $d\ne l[E_i]$ with $l\ge 2$. Then
any irreducible component of $\overline C$ 
entirely mapped to  $E_i$ is  isomorphically mapped to $E_i$.
\end{cor}
\begin{proof}
Let $l'_i$ be the number of connected components of $\overline C$
mapped to $E_i$, and let $l_i$ be the total
multiplicity under which the curve $E_i$ appears in $\overline
f(\overline C)$. 
Note that $l_i\ge l'_i$ with equality if and only if 
any irreducible component of $\overline C$ 
mapped to  $E_i$ is  isomorphically mapped to $E_i$.
We denote by $\overline C_{\N}$ the union of all
irreducible components of $\overline C$ mapped to $\N$, and by
$\overline C_{\X_n}$ the union of those which are mapped to $\X_n$ but
not entirely to $ E_i$. 

Suppose that $d=d_D[D] -\sum_{j=1}^n \mu_j[E_j]$, and that $\overline f(\overline
C_\N)$ realizes the homology class $d_E[E_\infty] + d_F[F]$. Then the
curve $\overline f_*(\overline C_{\X_n})$ realizes the homology class
$$(d_D-2d_E)[D] - \sum_{j\ne i} (\mu_j-d_E)[E_j] -  (\mu_i-d_E+l_i)[E_i].$$

In the degeneration of $\X_n$ to $\X_n\cup \N$, the curve $E_i$
degenerates to  the union $\overline E_i$ of $E_i$ and the fiber
of $\N$ passing 
through  $E_i\cap E$. The sum of multiplicity of intersections over
intersection points
of $f(\overline C_{\X_n}\cup\overline C_\N)$ with 
 $\overline E_i$ and  not contained in  $\X_n\cap \N$
 is then
$$\mu_i-d_E+l_i + (d_E-l'_i) =\mu_i +l_i -l'_i.$$
On the other hand, all those intersections deform to $\mathcal Y$,
hence
we must have
$$d\cdot [E_i] =\mu_i \ge  \mu_i +l_i -l'_i.$$
In conclusion $l_i=l'_i$ and we are done.
\end{proof}

Next corollary is  an immediate combination of  Propositions
\ref{prop:degeneration}
and \ref{prop:initial values N}.

\begin{cor}\label{prop:CHXn}
Let $\x_\N^\circ=\x^\circ(0)\cap \N$, and 
let $\overline C'$ an
irreducible component of $\overline C$ mapped to $\N$.
If $|\x^\circ_\N\cap\overline f(\overline C')|\le 2$,
then  $\overline f(\overline C')$ realizes either 
 the class $d_F[F]$ or $[E_\infty]+d_F[F]$ in
   $H_2(\N;\Z)$. Moreover we have 
$|\x^\circ_\N\cap\overline f(\overline C')|\le 1$ in the former
 case, and $1\le |\x^\circ_\N\cap\overline f(\overline C')|\le 2$ in the latter
 case.
\end{cor}

When enumerating real curves, we need to study carefully 
 the possible limits of the nodes of the  curves.
Given  a map  $f:C\to X$  from a (possibly singular) complex
algebraic curve, we say that $\{p,p'\}\subset C$ is 
a \emph{nodal pair} if it is an isolated solution of the
 equation $f(x)=f(y)$. In particular if $\{p,p'\}$ is a nodal pair,
 then $p\ne p'$.

Denote by $\mathcal P(\overline f) $ the set of points
 $p\in
\overline f^{\ -1}\left(\X_n\cap\N\right)$ 
such that none of the restrictions of $\overline f$ on
the local branches of $\overline C$ at $p$ is a non-trivial
ramified covering onto its image. 

\begin{prop}[{\cite[Lemmas 2.10 and 2.19]{Shu13}}]\label{prop:node on E}
Let 
$p\in\mathcal P(\overline f) $, and
let $U_p$ be a small neighborhood of $p$ in
$\overline C$. Then when deforming $\overline f$ to an element of 
 $\CC^{\alpha,\beta}(d,g,\x(t))$, exactly $\mu_p-1$ nodal pairs appear
in the deformation of $U_p$. 
\end{prop}

\begin{cor}\label{cor:total node}
Let $f: C\to \X_n\in \CC^{\alpha,\beta}(d,g,\x(t))$, with $|t|<<1$,
 be a deformation of 
$\overline f:\overline C\to \pi^{-1}(0)$.
 Then any
nodal pair of $C$ is either the deformation of a nodal pair of
$\overline C$, or is contained
 in the deformation of a neighborhood of a point
$p\in\mathcal P(\overline f)$. 
\end{cor}
\begin{proof}
Let us decompose the curve $\overline C$ as follows
$$\overline C= \overline C'_{\X_n} \bigcup \overline C'_{\N} 
\bigcup \left(\bigcup_{i=1}^n\overline C^{(i)}\right)
\bigcup \overline G $$
where 
\begin{itemize}
\item  $\overline C'_{\X_n}$ is the union of irreducible components of
  $\overline C$ which are mapped to $\X_n$, but not entirely to one of the
  curves $E_i$;
\item $\overline C'_{\N}$ is the union of irreducible components of
  $\overline C$ which are mapped to $\N$ and whose image does not
  realize a multiple of the fiber class;
\item $\overline C^{(i)}$ is the union of irreducible components of
  $\overline C$ which are mapped to $E_i$; we denote by $l_i$ the
  number of such curves;
\item $\overline G$ is the union of irreducible components of
  $\overline C$ which are mapped to $\N$ and whose image
  realizes a multiple of the fiber class; we denote by $l$ the sum
  of all those multiplicities.
\end{itemize}

Let us denote by $d_{\X_n}=d_{D}[D] -\sum_{i=1}^n \mu_{i}[E_i]$ 
the homology class realized by $\overline f(\overline
C'_{\X_n})$ in $\X_n$, and by $d_\N=d_E[E_\infty] + d_F[F]$
 the one realized by $\overline f(\overline
C'_{\N})$ in $\N$.  
By the adjunction formula, the number of nodal pairs of $\overline
f_{|\overline C'_{\X_n}}$ is  
$$a_{\X_n}=\frac{d_{\X_n}^2 - c_1(\X_n)\cdot d_{\X_n} +\chi(\overline C'_{\X_n})}{2}. $$
Moreover, according to Propositions \ref{prop:degeneration} and
\ref{prop:shoshu}, none of those pairs is mapped to $E$. 
Similarly, 
Corollary
\ref{cor:no sEi} implies that
 the number of nodal pairs of $\overline
f_{|\overline C'_{\N}}$ which are not mapped to $E_\infty$ is exactly
$$a_\N=\frac{d_\N^2 - c_1(\N)\cdot d_\N +\chi(\overline C'_{\N}) - \sum_{i=1}^n
  l_i(l_i-1)}{2}. $$ 
Furthermore we have
$$d= (d_{D} +2d_E)[D] -\sum_{i=1}^n \left(\mu_{i}+d_E-l_i
\right)[E_i],$$
$$(d_\N +l [F])\cdot [E_\infty] = (d_{\X_n} +\sum_{i=1}^nl_i[E_i])\cdot [E],  
\quad 
\mbox{ and } 
\quad \chi(\overline C'_{\X_n} )+ \chi(\overline C'_{\N})=2-2g
+2a , $$
where $a$ is the number of intersection points of 
$\overline C'_{\X_n} $ and $\overline C'_{\N}$.
Thus we deduce that the total number of nodal pairs of $\overline f$
which are not mapped to $\X_n\cap \N$ is exactly
$$a_{\X_n}+ \sum_{i=1}^nl_i\mu_i+ a_\N + l d_E
=\frac{d^2 - c_1(\X_n)\cdot d + 2 - 2g}{2}  - (d_{\X_n}\cdot [E] -l - a).$$
Each of these nodal pairs deform to a unique nodal pair of $f$.
Combining this with Proposition \ref{prop:node on E} and the fact that
$$d_{\X_n}\cdot [E] -l -a =\sum_{p\in \mathcal P(\overline f)}(\mu_p -1),$$
 the corollary now follows from the adjunction formula.
\end{proof}

 Thanks to Proposition
 \ref{prop:degeneration}, we know how many
  elements of 
$\CC^{\alpha,\beta}(d,g,\x(t))$  converge, as $t$ goes to 0, to a
  given element $\overline f$ of 
$\CC^{\alpha,\beta}(d,g,\x(0))$. In the case when the situation is
  real,  we now determine how many of those
  complex maps are real.
So let us assume 
 that $\X_n$ is endowed with the real structure $ \X_n(\kappa)$.
The previous degeneration $\mathcal Y\to \C$ has a canonical real
structure compatible with the one of $\X_n$,
 and  let us choose the set of sections 
$\x:\C\to \mathcal Y$ to be real. 

Given $p\in \R(\X_n\cap\N)$,  choose a  neighborhood $V_p$ of $p$ in
$\R(\X_n\cup\N)$ homeomorphic to the union of two disks. The set
 $V_p\setminus \R (\X_n\cap\N)$ has four connected
components $V_{p,1},V_{p,2}\subset \R \X_n$ and 
$V_{p,3}, V_{p,4}\subset \R \N$, labeled so that 
when smoothing $\R(\X_n\cup\N)$ to $\R\pi^{-1}(t)$ with $t>0$,
 the
components $V_{p,1}$ and $V_{p,3}$ one hand hand, and $V_{p,2}$ and $V_{p,4}$ on the
other hand, glue together, see Figure \ref{fig:real part 1}a. 
Denote respectively by $V_{p,1,3}$ and $V_{p,2,4}$ a deformation of
$V_{p,1}\cup V_{p,3}$ and $V_{p,2}\cup V_{p,4}$ in $\R \pi^{-1}(t)$
with $t>0$.
 \begin{figure}[h]
\centering
\begin{tabular}{cc}
\includegraphics[height=2.6cm, angle=0]{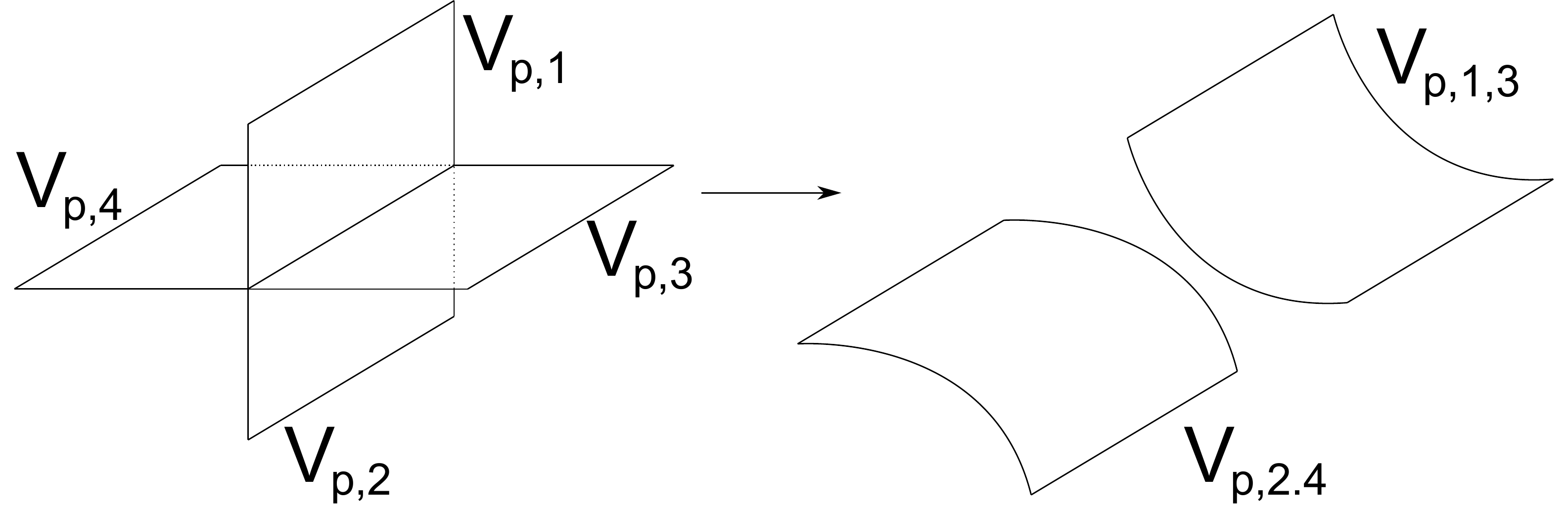}
&
\includegraphics[height=2.6cm, angle=0]{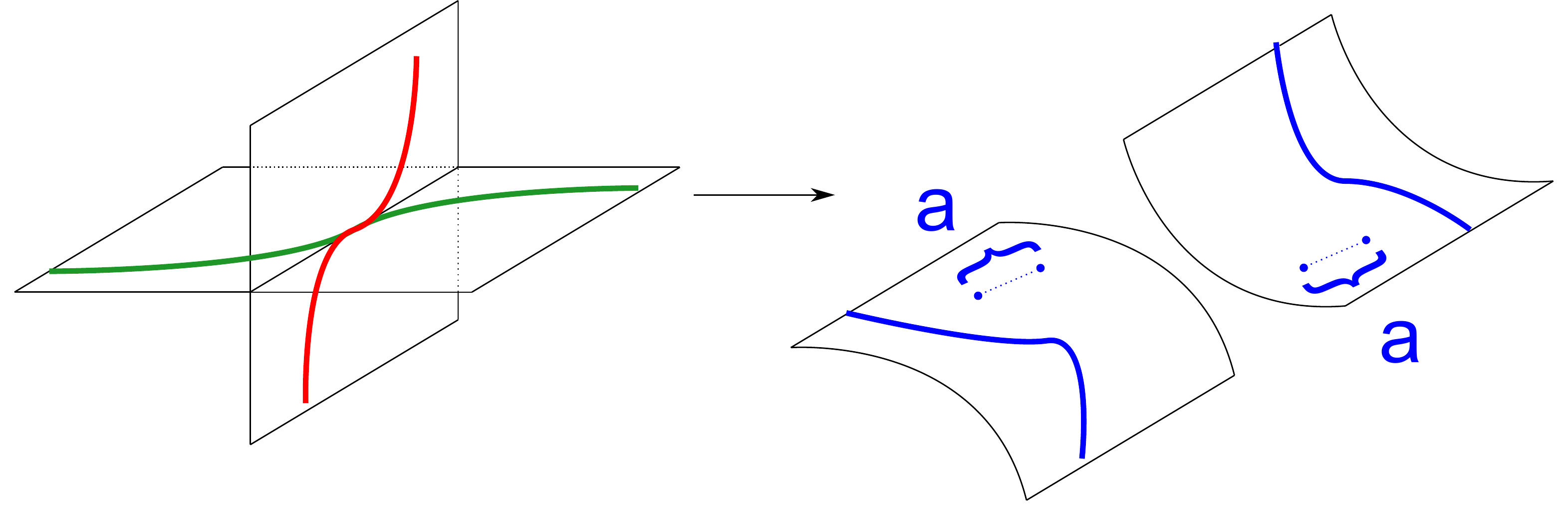}
\\
\\ a) & b) 
\end{tabular}
\caption{Real deformations of a real map 
$\overline f:\overline C\to\X_n\cup \N$}
\label{fig:real part 1}
\end{figure}

Let us fix  
  a real element $\overline f:\overline C\to \X_n\cup\N$
of $\CC^{\alpha,\beta}(d,g,\x(0))$.
Given $q\in   \overline f\ ^{-1} (\X_n\cap \N)$,
denote by
  $\overline C'_q$ 
the
 irreducible
 component
 of $\overline C$ containing $q$ and mapped to $\N$. 

Given a pair $\{q,\overline q\}$ of 
conjugated elements in $\overline f\ ^{-1} (\X_n\cap \N)$,
 define $\mu_{\{q,\overline q\}}=1$ if  $\overline f(\overline
 C'_q)\cap \x(0)= \emptyset$, and $\mu_{\{q,\overline q\}}=\mu_q$ otherwise.
Note that
$\mu_q=\mu_{\overline q}$ so
$\mu_{\{q,\overline q\}}$ is well defined.
Recall that $\overline f(\overline C'_q)\cap \x(0)= \emptyset$ implies
that $\overline f(\overline C'_q)$ is a multiple fiber of $\N$.
We denote by $\xi_0$  the product of the $\mu_{\{q,\overline q\}}$ where 
$\{q,\overline q\}$ ranges over all pairs of conjugated
elements in $\overline f\ ^{-1} (\X_n\cap \N)$.

 Given  $q\in \R  \left(\overline f\ ^{-1} (\X_n\cap \N)\right)$,  denote by
 $U_{q}$ a
 small neighborhood of $q$ in 
$\R \overline C$. 
If $\mu_q$ is even,   define  the integer $\xi_q $ as follows:
\begin{itemize}
\item  if 
$\overline f(U_{q}  )\subset V_{\overline f(q),1}\cup V_{\overline f(q),4}$ 
or $\overline f(U_{q}  )\subset 
V_{\overline f(q),2}\cup V_{\overline f(q),3}$, then $\xi_q =0$;
\item if 
$\overline f(U_{q}  )\subset V_{\overline f(q),1}\cup V_{\overline f(q),3}$ 
or $\overline f(U_{q}  )\subset 
V_{\overline f(q),2}\cup V_{\overline f(q),4}$, then
    $\xi_q =1$ if $\overline f(\overline C'_q)\cap \x(0)= \emptyset$,
and  $\xi_q =2$ otherwise.
\end{itemize}
Finally, define $\xi(\overline f)$ as the product of $\xi_0$ with all the $\xi_q$ where
$q$ ranges over all points in $ \R \left( \overline f\ ^{-1} (\X_n\cap \N)\right)$
with $\mu_q$ even.

\begin{prop}[
{\cite[Lemma 17]{IKS13}}]\label{prop:real degeneration}
The real map $\overline f$
 is the limit of exactly 
$\xi(\overline f)$ 
real maps in 
$\CC^{\alpha,\beta}(d,g,\x(t))$ with $t>0$.
Moreover for each $q\in \mathcal P(\overline f)$, one has
\begin{itemize}

\item if  $\mu_q$ is
  odd, then any real deformation
of $\overline f$ has exactly $a$
  solitary nodes in both $V_{\overline f(q),1,3}$ and $V_{\overline
    f(q),2,4}$, with $a =\frac{\mu_q -1}{2}$ or $a=0$ (see Figure
  \ref{fig:real part 1}b); 

\item if $\mu_q$ is
  even, then half of the
 real deformations of $\overline f$ have exactly $\frac{\mu_q-2}{2}$
  solitary nodes  in $V_{\overline f(q),1,3}$ and $\frac{\mu_q}{2}$
  solitary nodes  in $V_{\overline f(q),2,4}$, while the other half of
real deformations of $\overline f$ have no
  solitary nodes  in $V_{\overline f(q),1,3}\cup V_{\overline f(q),2,4}$
 (see Figure   \ref{fig:real part 2}).
\end{itemize}

\end{prop}
 \begin{figure}[h]
\centering
\begin{tabular}{c}
\includegraphics[height=2.6cm, angle=0]{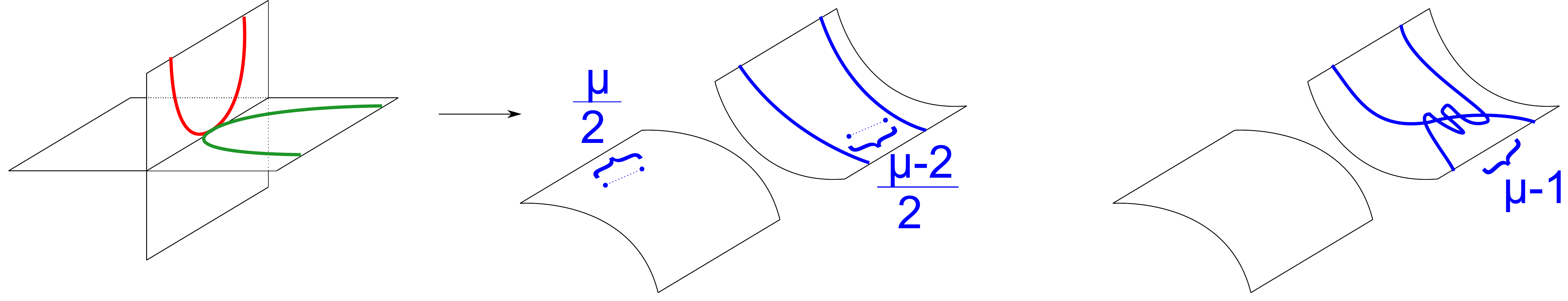}
\\
\\ a) 
\end{tabular}
\caption{Real deformations of a real map 
$\overline f:\overline C\to\X_n\cup \N$, continued}
\label{fig:real part 2}
\end{figure}

\subsection{Proof of Theorem \ref{NFD}}\label{sec:proof NFD}
 Theorem \ref{NFD} is proved by  a recursive use of
Proposition \ref{prop:degeneration}.
The fact that such a recursion is indeed possible is ensured by
Corollary \ref{cor:no sEi}: if $d\ne l[E_i]$ for all $i=1,\ldots,n$ and
$l\ge 2$,
then the same holds for the class realized by the image
of any irreducible
component of $\overline C$.

The union $Y$ of finitely many irreducible algebraic varieties
$Y_1,\ldots, Y_k$
intersecting transversely is called a \emph{chain} if $Y_i\cap Y_j\ne\emptyset $
only when $|i-j|=1$. In this case  denote by $Z^+_i$ (resp. $Z^-_i$)
the intersection $Y_i\cap Y_{i+1}$ viewed as a subvariety of $Y_i$
(resp. $Y_{i+1}$),
and  write 
 $$Y= Y_k\ _{Z^-_{k-1}}\cup_{Z^+_{k-1}} Y_{k-1} \ _{Z^-_{k-2}}\cup_{Z^+_{k-2}}  \ldots
\ _{Z^-_{1}}\cup_{Z^+_{1}} Y_1. $$

Assume now that $d\cdot [D]\ge 1$ and $d\cdot [D]-1+g +|\beta|\ge 1$, 
 all
remaining cases being covered by Proposition \ref{prop:initial values}.
Recall that we have chosen two non-negative integers $r$ and $s$ such
that
$$d\cdot [D] -1+g +|\beta|=r+2s, $$
and that $s>0$ implies that $g=0$.
By iterating the degeneration process of $\X_n$ described in Section
\ref{sec:Li deg}, we construct a flat morphism $\pi: \mathcal Z\to \C$
such that 
\begin{itemize}
\item $\pi^{-1}(t)=\X_n$ for
$t\ne 0$;
\item  $\pi^{-1}(0)$ is a chain of $X_n$ and 
$r+s+1$ copies of $\N$:
$$\pi^{-1}(0)=\X_n \ _{E}\cup_{E_\infty}  \N_{s+r} \ _{E_0}\cup_{E_\infty}  \N_{s+r-1}  \ _{E_0}\cup_{E_\infty} \ldots 
 \ _{E_0}\cup_{E_\infty}  \N_{0}. $$
\end{itemize}

Choose $\x^\circ(t)$ 
a generic  set of $d\cdot [D]-1+g+|\beta|$ 
holomorphic sections $ \C \to \mathcal Z$ 
such that $\x^\circ(0)$ contains exactly one point (resp. two points)
in each $\N_i$ with $i\ge s+1$ (resp. $1\le i\le s$). 
Choose $\x_E(t)$ a generic set of $|\alpha|$ holomorphic sections 
$ \C \to \mathcal Z$ such that $\x_E(t)\in E$ for any $t\in\C^*$.  
In particular $\x_E(0)$ is contained in the divisor
$E_0$ of $\N_0$. Define $\x(t)=\x^\circ(t)\sqcup \x_E(t)$.
Let $\overline f:\overline C\to \pi^{-1}(0)$ be a limit of
  maps in $\CC^{\alpha,\beta}(d,g,\x(t))$ as $t$
  goes to $0$. 
It follows from Propositions \ref{prop:initial values}, 
\ref{prop:initial values N}, \ref{prop:degeneration}, \ref{cor:no sEi},
 and \ref{prop:CHXn}  that 
 $$\overline C =\left(\bigcup_{i=1}^{k_1} \overline C_i
\bigcup_{i=k_1+1}^{k_2} \overline C_i
\ \ \bigcup_{i=k_2+1}^{k_3} \overline C_i\right) 
\left(\bigcup_{i=1}^{l} \overline C'_i \ \ \bigcup_{i=l+1}^{l+|\alpha|+|\beta|}
\overline C'_i \right) 
\left(\bigcup_{i=1}^n \ \ \bigcup_{j=1}^{l_i}
\overline C^{(i)}_j\right)$$
where:
\begin{enumerate}
\item  the curves $ \overline C_i$ are pairwise disjoint
 irreducible rational curves; 
if $i\le k_1$ (resp. $k_1+1\le i\le k_2$, $i\ge k_2+1$), then
$\overline f(\overline C_i)$ 
realizes the class $[E_\infty] +d_{F,i}[F]$ in $\N$
(resp. $[D]-[E_{j_i}]$ in $\X_n$, $[D]$ in $\X_n$);

\item  the curves $\overline C'_i$ are pairwise disjoint
 chains of rational curves, and 
$\overline f(\overline C'_i)$ is a chain of (equimultiple) fibers of
 the $\N_j$'s; if $i\le l$ (resp. $ i\ge l+1$), then
$\overline C_i'$   intersects  exactly two
 (resp. one)
   curves $ \overline C_j$; 

\item the curves $\overline C^{(i)}_j$ are pairwise disjoint
chains of rational curves; exactly one component of 
$\overline C^{(i)}_j$ is mapped (isomorphically) to $E_i$, while  all the others
  are mapped to a chain of (simple) fibers; moreover $\overline C^{(i)}_j$
intersects a unique $ \overline C_a$ with $a\le k_1$, and does not
intersect any curve $\overline C'_{b'}$ nor  $ \overline C_b$ with $b\ge k_1+1$.

\end{enumerate}
Note that by Proposition \ref{prop:initial values}, 
we have
$$\overline f\left(\overline C\right)\cap \N_0= 
\overline f\left(\bigcup_{i=l+1}^{l+|\alpha|+|\beta|}
\overline C'_i\right)\cap \N_0.$$

\begin{lemma}\label{lem:euler graph}
One has 
$$k_3=d\cdot [D]-2k_1 -k_2,\quad l_i=k_1+k_2-d\cdot [E_i],\quad \mbox{and}\quad 
l=d\cdot [D] -k_1 -1+g.$$
\end{lemma}
\begin{proof}
By counting intersection points of  $\overline f(\overline C)$ with
the  degeneration in $\X_n\cup\N$ of a
generic curve in $[D]$ or $[D]-[E_i]$,  we get
$$d\cdot [D]=2k_1 + k_2+k_3 \quad \mbox{and}\quad 
d\cdot ([D]-[E_i])=l_i+k_1 +k_3, $$
which gives the first two equalities.

Computing the Euler characteristic of
$\overline C$ in two ways yields
$$2\left(k_1+k_2+k_3+l+l'+\sum_{i=1}^6l_i\right) -2l-l'- \sum_{i=1}^6l_i= 2-2g
+ 2l+l'+ \sum_{i=1}^6l_i, $$
which provides the third equality.
\end{proof}

Given $i\le k_1$, Denote by $a_i$ the integer such that $ \overline
f(\overline C_i)\subset \N_{a_i}$. 
Without loss of generality, we may assume that  
$a_i\ge a_j$ if $i>j$.
Denote also by $A_{\overline C}$ the set composed of couples $(\overline
C_a, i)$ such that either 
$$a\le k_1 \mbox{ and }
\overline C_a \textrm{ is disjoint
    from }
\bigcup_{j=1}^{l_i}
\overline C^{(i)}_j,$$
or
$$k_1+1\le a\le k_2\mbox{ and }
[\overline f(\overline C_a)]=[D]-[E_i].$$
Let us construct an oriented weighted graph $\D_{\overline C}$ as follows:

\begin{itemize}
\item the vertices of $\D_{\overline C}$ are the elements the disjoint union of two sets 
$Vert^\circ(\D_{\overline C})$ and $Vert^\infty(\D_{\overline C})$:
\begin{itemize}
\item  Vertices  in $Vert^\circ(\D_{\overline C})$ are in one-to-one
  correspondence with the curves  $ \overline C_i$, and are denoted by
 $v_i$. 

\item Vertices  in
  $Vert^\infty(\D_{\overline C})$ are in one-to-one 
  correspondence with elements of $A_{\overline
    C}\cup\{1,\ldots,|\alpha|+|\beta|\}$, and are denoted respectively by
  $v^\infty_{\overline C_a, i}$ and $v^\infty_i$.
\end{itemize}

\item edges of $\D_{\overline C}$ are in one-to-one
  correspondence with elements of $A_{\overline
    C}\cup\{1,\ldots,|\alpha|+|\beta|+ l\}$:

\begin{itemize}
\item The edge $e_{\overline C_a, i}$ corresponding to 
  $(\overline C_a, i)\in A_{\overline C}$ is adjacent to the vertices
    $v^\infty_{\overline C_a, i}$ and $v_a$, oriented
  from the former to the latter.  The weight of $e_{\overline
    C_a, i}$ is equal to 1.

\smallskip
\item The edge $e_{i}$ corresponding to 
$i\in\{1+l,\ldots,|\alpha|+|\beta|+l\}$ is adjacent 
  to the vertices $v^\infty_{i-l}$ and $v_a$, oriented
  from the former to the latter, where $a$ is 
such that $\overline C_a\cap\overline C'_i \ne\emptyset$.
 The weight of $e_{i}$ is equal to the degree of the covering
 $\overline f_{|\overline C'_i}$.

\smallskip
\item The edge $e_{i}$ corresponding to
  $i\in\{1, l\}$ is adjacent to the vertices 
 $v_a$ and $v_b$ with $a<b$, oriented
  from the former to the latter, where $a$ and $b$ are such that 
 $\overline C'_i$ intersects $\overline C_a$ and $\overline C_b$;
The weight of $e_{i}$ is equal to the degree of the covering
 $\overline f_{|\overline C'_i}$.
\end{itemize}

\end{itemize}

\begin{lemma}\label{lem:proof N1}
The oriented weighted graph $\D_{\overline C}$ is a floor diagram of
degree $d\cdot [D]$ and genus $g$. 
\end{lemma}
\begin{proof}
We only have to compute $b_1(\D_{\overline C})$ and
$\sum_{v\in Vert^\infty(\D_{\overline C})}div(v)$, the other 
properties  of a floor diagram following immediately from the construction of
$\D_{\overline C}$.
We have
$$\sum_{v\in Vert^\infty(\D_{\overline C})}div(v)=d\cdot [E] + n(k_1 + k_2) -
\sum_{i=1}^nl_i = d\cdot [E] +\sum_{i=1}^nd\cdot
[E_i]=2d\cdot [D]. $$
According Lemma \ref{lem:euler graph}, we have
$$b_1(\D_{\overline C})=1 - |Vert(\D_{\overline C})| + l= g,  $$
and we are done.
\end{proof}

The floor diagram $\D_{\overline C}$ 
is naturally equipped 
with a
marking. 
Let
$A_0,A_1,\ldots,A_n$ be some disjoint sets such that $|A_i|=d\cdot
[E_i]$ for $i=1,\ldots n$,
and $A_0=\{1,\ldots, d\cdot [D] -1+g+|\alpha|+|\beta| \}$. Denote
by $p_{2s+i}$ with $1\le i\le r$
(resp. $p_{2i-1}$ and $p_{2i}$ with $1\le i\le s$) the
point in $\x^\circ(0)\cap \N_{2s+i}$  (resp the two points
 in $\x^\circ(0)\cap \N_{i}$), and define
$A'_0$ to be the union of all pairs $\{2a_i-1, 2a_i\}$
which are contained in $\overline f(\overline C_i)$.
Denote also $\x_E(0)=(p_{i,j})_{0\le j\le\alpha_{i,j}, i\ge 1}$.

Define a map $m_{\overline C}:\left(A_0\setminus A'_0\right)\bigcup_{i=1}^n A_i\to
\D_{\overline C}$ as follows: 
\begin{itemize}
\item for $i=1,\ldots,n$, the restriction of $m_{\overline C}$ on $A_i$ is a
  bijection to the set 
$\{ e_{\overline C_a,i}\ | \ (\overline
  C_a,i)\in A_{\overline C}\};$
\item for $i\in A_0\setminus A'_0$ and $i\ge|\alpha|+1$,  set
 $m_{\overline C}(i)=v_j $ (resp. $e_j$) if $p_i\in \overline
  f(\overline C_j)$ (resp. $p_i\in \overline
  f(\overline C'_j)$);
\item for $0\le j\le \alpha_i$, we set
 $m_{\overline C}\left(\sum_{a=1}^{i-1}\alpha_a + j\right)=v^\infty_b$ if  
$p_{i,j}\in \overline f(\overline C'_b)$.
\end{itemize}
The map $m_{\overline C}$ is injective and increasing by construction.

It follows from Propositions \ref{prop:initial values N} and
\ref{prop:degeneration} that
$\D'=\D_{\overline C}\setminus \left(Im(m_{\overline C})\cup Vert^\infty(\D)\right) $ contains as many
vertices  as edges, and that any
element  of $\D'$ is adjacent to at least one other element of $\D'$.
Suppose that there exist an element $u$ of $\D'$ which is adjacent to only one
other element $v$ of $\D'$. Then either $u$ or $v$ corresponds to a
component $\overline C_i$ of $\overline C$.
Extend the map $m_{\overline C}$ on $\{2a_i-1,2a_i\}$ to
$\{u,v\}$ in
the unique way so that it remains an increasing map.
Extend inductively $m_{\overline C}$  in this way as long as 
$\D'$ contains an element
 adjacent to only one
other element of $\D'$.

\begin{lemma}\label{lem:proof N2}
The resulting map  $m_{\overline C}$ is a $d$-marking of
$\D_{\overline C}$ of type $(\alpha,\beta)$. 
\end{lemma}
\begin{proof}
The only thing  to show is
 that $m_{\overline C}$ is surjective.
The set $\D'$ still contains as many
vertices  as edges, and any  edge in $\D'$ is adjacent to two vertices in 
$\D'$. Hence it is either
empty or a
disjoint collection of loops. Since $g=0$ if $s>0$, this latter case cannot
occur.
\end{proof}

\begin{rem}\label{rem:s=0}
The above lemma is one of the two places where the assumption
that $g=0$ if $s>0$ is used. With a little extra-care, one can  adapt this
construction also in the case where $g>0$ and $s>0$.
\end{rem}

Given a $d$-marked floor diagram $(\D,m)$ of 
degree $d\cdot [D]$  and genus
$g$, we define $B_{\D,m}$ to be the set of edges $e$ of $\D$
which have an adjacent floor $v$ such that
$m(\{2i-1,2i\})=\{v,e\}$ for some $i=1,\ldots, s$. 
Note that $B_{\D_{\overline C},m_{\overline C}}= Edge(\D_{\overline
  C})\cap m(A'_0)$.
Theorem \ref{NFD} now follows from the two following lemmas.

\begin{lemma}\label{lem:proof N3}
The map  $\overline f \mapsto (\D_{\overline C},m_{\overline C})$ is 
surjective from the set $\CC^{\alpha,\beta}(d,g,\x(0))$
 to the set of $d$-marked floor diagram of type $(\alpha,\beta)$, 
degree $d\cdot [D]$, and genus $g$.
More precisely, such a marked floor diagram $(\D,m)$  is the
image of exactly 
$\prod_{e\in B_{\D,m}}w(e)$ elements of
$\CC^{\alpha,\beta}(d,g,\x(0))$. 
\end{lemma}
\begin{proof}
The number of possibilities  to reconstruct $\overline f$ out of
$(\D,m)$ is given by  Propositions \ref{prop:initial values} and 
\ref{prop:initial values N}.
\end{proof}

\begin{lemma}\label{lem:proof N4}
A morphism $\overline
f:\overline C\to\pi^{-1}(0)$ is the limit of exactly
$$\frac{\mu^\C(\D_{\overline C})}{\prod_{e\in B_{\D_{\overline C},m_{\overline C}}}w(e)}$$
  elements of 
$\CC^{\alpha,\beta}(d,g,\x(t))$ as $t$ goes to 0.
\end{lemma}
\begin{proof}
This  follows from Proposition \ref{prop:degeneration}.
\end{proof}

\subsection{Proof of Theorem \ref{WFD}}\label{sec:proof WFD}
Now let us suppose that $\X_n$ is equipped with the real structure
$\X_n(\kappa)$, and that $g=0$.
The previous degeneration $\mathcal Z\to \C$ can obviously be equipped with a
real structure such that $\pi^{-1}(t) = \X_n(\kappa)$ if $t\in\R^*$, and  
 $\pi^{-1}(t) = \X_n(\kappa)\cup \N_{s+r}\cup\ldots \cup \N_0$ where $\N$ is equipped with the real
structure for which $\pi_E:\N\to E_\infty$ is real. Note that
$\R E\ne\emptyset$  and
$\R \N\ne\emptyset$.
Choose the configuration $\x_E(t)$ to be of type $(\alpha^\Re,\alpha^\Im)$ with
$\alpha^\Re+2\alpha^\Im=\alpha$,
the configuration $\x^\circ(t)$ to be $(E,s)$-compatible for $t\ne 0$,
and 
$\x(0)\cap\N_i$ to be a pair of complex conjugated points (resp. to be a
real point) if $1\le i\le s$
(resp. $i\ge s+1$).
There remain two steps to end the proof of Theorem
\ref{WFD}:
\begin{enumerate}
\item identify which elements of $\CC^{\alpha,\beta}(d,0,\x(0))$
  are real;

\item for each such element $\overline f$, determine how many real
elements of 
$\CC^{\alpha,\beta}(d,g,\x(t))$ converge to $\overline f$ as $t$
  goes to 0, and determine their different real multiplicities.
\end{enumerate}

The first step is  straightforward. 
\begin{lemma}\label{lem:real FD}
An element $\overline f:\overline C\to \pi^{-1}(0)$
 of $\CC^{\alpha,\beta}(d,0,\x(0))$
  is real if and only if the marked floor diagram 
$(\D_{\overline C},m_{\overline C})$ is $(s,\kappa)$-real.

In this case $(\D_{\overline C},m_{\overline C})$ is of
type $(\alpha^\Re, \beta^\Re,\alpha^\Im,\beta^\Im)$ 
if and only if
exactly $\alpha^\Re_i+\beta^\Re_i$ 
(resp. 
$2\alpha^\Im_i+2\beta^\Im_i$)
irreducible components 
of  
$\overline C$ are mapped with degree $i$ 
to a real (resp. non-real) fiber of $\N_0$.
\end{lemma}
\begin{proof}
Suppose that $\overline f$ is real. Then the action of the complex
conjugation on  $\x(0)$ and on 
the irreducible components of $\overline C$ induces an
involution on  $\bigcup_{i=0}^n A_i$ and  
$\D_{\overline C}$ which turns $(\D_{\overline
  C},m_{\overline C})$ into a  $(s,\kappa)$-real marked floor diagram.

Conversely, if $(\D,m)$ is a  $(s,\kappa)$-real marked floor diagram,
 Propositions \ref{prop:initial values} 
and \ref{prop:initial values N} together with Lemma
\ref{lem:m real} implies that any map $\overline f$ in 
$\CC^{\alpha,\beta}(d,0,\x(0))$ such that 
$(\D,m)=(\D_{\overline C},m_{\overline C})$ is real.

The last statement follows from the construction of $(\D_{\overline
  C},m_{\overline C})$.
\end{proof}

Let us fix a real element $\overline f:\overline C\to \pi^{-1}(0)$
of $\CC^{\alpha,\beta}(d,0,\x(0))$, and denote by 
$(\alpha^\Re, \beta^\Re,\alpha^\Im,\beta^\Im)$ the type of 
$(\mathcal D_{\overline C},m_{\overline C})$.
 The number  of  nodal pairs of $\overline f$
 composed of two complex
conjugated elements
is
 denoted by  $m^\circ(\overline f)$. Recall that the integer $o_v$
 associated to an element
$\{v,v'\}\in Vert^\Im(\D)$ has been defined in Section \ref{sec:real FD}.

\begin{lemma}\label{lem:node f 1}
One has
$$m^\circ(\overline f)= \prod_{\{v,v'\}\in Vert_\Im(\D_{\overline C})}(-1)^{o_v}. $$
\end{lemma}
\begin{proof}
If follows from Proposition \ref{prop:initial values},
\ref{prop:initial values N}, and Corollary \ref{prop:CHXn} that a
nodal pair contributing to $m^\circ(\overline f)$ contains two points
in two conjugated irreducible components of $\overline C$. Since two
complex conjugated fibers in $\N$ do not intersect, these two
components  must correspond to two floors in $Vert_\Im(\D_{\overline C})$.
For homological reasons $o_v$ has the same parity 
as the number of nodal pairs contained in these two components and
mapped to $\R\N\setminus \left(E_0\cup E_\infty \right)$.
\end{proof}

For  $t$ a small enough non-null real number, 
 denote respectively by $\R \CC_{\overline f}(d,0,\x(t))$
and $\R\CC_{\overline
  f,L_\epsilon}(d,0,\x(t))$  the sets
of  elements of $\R\CC^{\alpha^\Re,
  \beta^\Re,\alpha^\Im,\beta^\Im}(d,0,\x(t))$ 
and
$\R\CC^{\alpha^\Re, \beta^\Re,\alpha^\Im,\beta^\Im}_{L_\epsilon}(d,0,\x(t))$
which converge to
$\overline f$ as $t$ goes to $0$.

\begin{lemma}\label{lem:sign 1}
One has
$$\sum_{f\in \R\CC_{\overline
     f}(d,0,\x(t))}(-1)^{m_{\R\X_n(\kappa)}(f(C))}=
\frac{\mu^\R_{r,\kappa}(\mathcal D_{\overline C},m_{\overline C})}
{\prod_{e\in B_{\mathcal D_{\overline C},m_{\overline C}}}w(e)}. $$
\end{lemma}
\begin{proof}
This follows from Lemma \ref{lem:node f 1}, Corollary \ref{cor:total node}, and
Proposition \ref{prop:real degeneration}. 
\end{proof}

Combining Lemmas \ref{lem:sign 1} and \ref{lem:proof N3}, we obtain
Theorem \ref{WFD}$(1)$. 

\medskip
Let us suppose now that $n=2\kappa$. Recall that
$\widetilde L_\epsilon$ denotes the connected component of 
$\R\X_n(\kappa)\setminus \R E$ with Euler characteristic $\epsilon$.
The surface 
$\R \N_i\setminus \left(\R E_0\cup \R E_\infty\right)$ 
has two connected components that are denoted  by $N_i^{0,1}$ in such a way 
that $N_i^1$ deforms to the
interior of $\R E$ for $t\in\R^*$. Define also $\overline
N_i^\epsilon$ to be the topological closure of $N_i^\epsilon$ in $\N_i$.
Note that $\x^\circ(0)$ deforms to an $(E,s,\widetilde L_\epsilon)$-compatible
configuration if and only if  $\x^\circ(0)\subset \bigcup_i
N^\epsilon_i$. In this case  $\x^\circ(0)$  is said to be
 $(E,s,\widetilde L_\epsilon)$-compatible.

\begin{lemma}\label{lem:sign 1bis}
If   $\R\CC_{\overline f,L_\epsilon}(d,0,\x(t))\ne\emptyset$,
then $(\D_{\overline C},m_{\overline C})$ is
$\epsilon$-sided and $\x^\circ(0)$  is
 $(E,s,\widetilde L_\epsilon)$-compatible. 
Moreover one has
$\R\CC_{\overline f,L_\epsilon}(d,0,\x(t))=
\R\CC_{\overline f}(d,0,\x(t)). $
\end{lemma}
\begin{proof}
If $\R \CC_{\overline
  f,L_\epsilon}(d,0,\x(t))\ne\emptyset$, then $\overline f(
\R \overline C)\subset \widetilde 
L_\epsilon\bigcup_i \overline N_i^\epsilon $. This
implies that $\x^\circ(0)$  is
 $(E,s,\widetilde L_\epsilon)$-compatible, and 
that any edge in $Edge_\Re(\D_{\overline C})$ has even weight.
If in addition $\epsilon=1$, then no curve $\overline C_i$ with $i\ge
k_1+1$ can be a real component of $\overline C$ since 
the real part of a real line in $\C P^2$
 cannot be contained in the interior of a
real ellipse. Conversely, if
$\overline f(\R \overline C)\subset \widetilde 
L_\epsilon\bigcup_i \overline N_i^\epsilon  $, then 
any map $f$ in $\R\CC_{\overline f}(d,0,\x(t))$ must 
satisfy $f(\R C)\subset \widetilde L_\epsilon\cup \R E$.
\end{proof}

Combining Lemmas \ref{lem:sign 1bis} and \ref{lem:sign 1}, we obtain
the first  assertion of Theorem \ref{WFD}$(2)$. 

\medskip

From now on, let us assume that $\x^\circ(0)$ is
$(E,s,\widetilde L_\epsilon)$-compatible, and that 
  $\overline f:\overline C\to \pi^{-1}(0)$
satisfies
$\overline f(\R \overline C)\subset \widetilde 
L_\epsilon\bigcup_i \overline N_i^\epsilon  $.
Denote by $m_\epsilon^\circ(\overline f)$ the number  of   nodal
pairs of $\overline f$
composed of two complex
conjugated points and  mapped
to 
$L_\epsilon \bigcup_i N^\epsilon_i$.
Consider the three following situations (recall that the involution
  $\rho_{s,\kappa}$ is defined at the
   beginning of Section \ref{sec:real FD}):
\begin{enumerate}
\item There exists $\{v,v'\}\in Vert_{\Im}(\D_{\overline C})$ and  
$i=1,\ldots, \kappa$ such that 
$ v$ is adjacent to an 
   edge adjacent to  $m\left( A_{2i-1}\right)$ and $v'$ is not. 
In this case let $j\in A_{2i-1}$ such that $ v$ is adjacent to the
   edge adjacent to $m_{\overline C}(j)$.
Since
   $(\D_{\overline C},m_{\overline C})$ is $(s,\kappa)$-real,
the vertex $v'$ is adjacent to the
   edge adjacent to  $m_{\overline C}(\rho_{s,\kappa}(j))$. 

\item We are not in the above situation, and
$Edge_\Im(\D_{\overline C})\setminus 
\left(\bigcup_{i=1}^n m_{\overline C}(A_i)\right)$ contains an edge of
odd weight. In this case, since $\D_{\overline C}$ is a tree, there
exists 
$j\in m_{\overline C}^{-1}\left(Edge_\Im(\D_{\overline C})\right)$ 
such that $m_{\overline C}(j)$ is of odd weight and
adjacent to 
a floor of $\D_{\overline C}$
 fixed by $\rho_{s,\kappa}$.

\item None of the two above situations occur.
\end{enumerate}

\begin{rem}\label{rem:s bis}
  Note that the assumption that $g=0$ whenever $s>0$ appears in case $(2)$. Again,
  one could adapt the arguments to avoid this assumption, at the cost
  of 
some extra work.
\end{rem}

In case $(3)$ above, set $\Delta_{\overline C}=\{(\D_{\overline C},m_{\overline C})\}$.
In the case $(1)$ and $(2)$, 
 define $m_{\overline C}'$ to
be the marking of $\D_{\overline C}$ which coincide with $m_{\overline
  C}$
 outside $\{j,\rho_{s,\kappa}(j)\}$ and
with
$m_{\overline C}'(j)=m_{\overline C}(\rho_{s,\kappa}(j))$ and 
$m_{\overline C}'(\rho_{s,\kappa}(j))=m(j)$. 
Clearly $(\D_{\overline C},m_{\overline C}')$
 is a $(s,\kappa)$-real $d$-marked floor diagram of
the same type as $(\D_{\overline C},m_{\overline C})$,
and set 
$\Delta_{\overline C}=\{(\D_{\overline C},m_{\overline C}),(\D_{\overline
  C},m_{\overline C}')\}$. 
Note that neither $i$ nor $j$ might  be
unique, however this does not matter in what follows.
Recall that the integer $o'_v$
 associated to an element
$\{v,v'\}\in Vert^\Im(\D)$ has been defined in Section \ref{sec:real FD}.

\begin{lemma}\label{lem:node f 2}
If $\Delta_{\overline C}=\{(\D_{\overline C},m_{\overline C})\}$, then 
$$m_{\epsilon}^\circ(\overline f)=  (-1)^{\epsilon|Vert_{\Im,1}(\D_{\overline C})|}
\prod_{\{v,v'\}\in Vert_{\Im,2}(\D_{\overline C})} (-1)^{o'_v}. $$
Otherwise
 let $\overline f':\overline C'\to \pi^{-1}(0)$ be a real element
 of $\CC^{\alpha,\beta}(d,0,\x(0))$ such that
$(\mathcal D_{\overline C'},m_{\overline C'})=(\mathcal D_{\overline
   C},m'_{\overline C})$. Then one has
$$m_{\epsilon}^\circ(\overline f) + m_{\epsilon}^\circ(\overline f')= 0. $$
\end{lemma}
\begin{proof}
This follows from Lemmas \ref{lem:line conic} and \ref{lem:even}.
\end{proof}

\begin{cor}\label{lem:sign 2}
If  $\Delta_{\overline C}=\{(\D_{\overline C},m_{\overline C})\}$, then
$$\sum_{f\in \R\CC_{\overline
     f, L_\epsilon}(d,0,\x(t))}(-1)^{m_{L_\epsilon}(f(C))}=
\frac{\nu^{\R,\epsilon}_{r}(\mathcal D_{\overline C},m_{\overline C})}
{\prod_{e\in B_{\mathcal D_{\overline C},m_{\overline C}}}w(e)}. $$

Otherwise let $\overline f':\overline C'\to \pi^{-1}(0)$ be a real element
 of $\CC^{\alpha,\beta}(d,0,\x(0))$ such that
$(\mathcal D_{\overline C'},m_{\overline C'})=(\mathcal D_{\overline
   C},m'_{\overline C})$. Then one has
$$\sum (-1)^{m_{L_\epsilon}(f(C))}=0, $$
 where the sum is taken over all elements $f:C\to \X_n$ in 
$\R \CC_{\overline
     f,L_\epsilon}(d,0,\x(t))\cup  \R\CC_{\overline
     f',L_\epsilon}(d,0,\x(t))$.
\end{cor}
\begin{proof}
This follows from Lemmas \ref{lem:m real} and \ref{lem:node f 1},
Corollary \ref{cor:total node}, and 
Propositions \ref{prop:real degeneration}. 
\end{proof}

Now the second identity in  Theorem \ref{WFD}$(2)$ follows immediately
from a combination of Corollary \ref{lem:sign 2} with Lemmas
\ref{lem:proof N3} and \ref{lem:real FD}.

\section{Absolute invariants of $X_7$}\label{sec:X7}

\subsection{Strategy}\label{sec:X7 strategy}
In this section,  absolute
invariants of $X_7$ are expressed in terms of invariants of $\X_8$ by
applying
Li's degeneration formula to the degeneration of $X_7$ described in
next proposition.

\begin{prop}\label{prop:degen X7}
There exists a flat degeneration $\pi:\YY\to \C$
of $X_7$ with $\pi^{-1}(0)=\X_6\cup \X_2$, where 
$\X_6\cap \X_2$ is the distinguished curve $E$ in both 
$\X_6$ and $\X_2$.
\end{prop}
\begin{proof}
Start with the classical flat
degeneration
$\pi_0:\YY_0\to \C$ of $X_6$ with $\pi^{-1}(0)=\X_6\cup (\C
P^1\times \C P^1)$, where 
$\X_6\cap (\C
P^1\times \C P^1)$ is the curve $E$ in $\X_6$, and is a hyperplane
section
 in $\C P^1\times \C P^1$. 
The 3-fold $\YY_0$ can be obtained by blowing up at the point
$([0:0:0:1],0)$  the singular hypersurface in
$\C P^3\times \C$ with equation $P_3(x,y,z) + wP_2(x,y,z) + w^3t^2=0$,
where 
\begin{itemize}
\item $P_i(x,y,z)$ is a homogeneous polynomial of degree $i$;

\item the curves defined in $\C P^2$ by $P_2$ and $P_3$ are smooth and
  intersect transversely.
\end{itemize}
Note that the $\X_6$ component of  $\pi^{-1}(0)$ is precisely the blow
up of $\C P^2$ at the six points in $\{P_2(x,y,z)=0\}\cap\{P_3(x,y,z)=0\}$.

Next,  choose a holomorphic section $p_0:\C \to \YY_0$ such that
$p_0(0)\in (\C P^1\times \C P^1)\setminus \X_6$, and blow up the
divisor $p_0(\C)$ in $\YY_0$. The obtained 3-fold $\YY$ is a
flat degeneration of $X_7$ with the desired properties (recall that 
$\C P^1\times \C P^1$ blown up at a point is also $\CP^2$ blown up
at two points).
\end{proof}

\begin{rem}
Although we started with $\C P^2$ blown up in seven points, 
the degeneration $\pi:\YY\to \C$ distinguishes eight special points on
$E$. Namely, these points are the intersection points of $E$ with a
$(-1)$-curve contained in
either $\X_6$ or $\X_2$. This explains why absolute
invariants of $X_7$ can be expressed in terms of invariants of $\X_8$
rather than those of $\X_7$
(see the proof of Theorem \ref{thm:GWX7} for details).
\end{rem}

\begin{rem}\label{rem:real structure}
By choosing suitable real polynomials $P_i(x,y,z)$, one constructs
 a  flat degeneration $\YY_0$ as above with a real structure such that 
$\pi^{-1}(t)=X_6(\kappa)$ for $t\in\R^*$, and 
\begin{itemize}
\item $\pi^{-1}(0)=\X_6(\kappa)\cup Q(0)$ if $\kappa\le 3$;

\item $\pi^{-1}(0)=\X_6(\kappa-1)\cup Q(2)$ if $\kappa\ge 1$;
\end{itemize}
where $Q(\epsilon)$ denotes $\C P^1\times \C P^1$ equipped with the
real structure satisfying $\R Q(\epsilon)\ne \emptyset$ and  $\chi(\R Q(\epsilon))=\epsilon$.
\end{rem}

All the results from this section are obtained by applying Li's degeneration
formula and its real counterpart 
to $\YY$ and a set of sections $\x:\C\to \YY$ satisfying 
$\x(0)\subset \X_6\setminus \X_2$.
As mentioned in the introduction, no non-trivial covering appears
during this degeneration.

\subsection{Gromov-Witten invariants}
Some additional notation are needed to state Theorem \ref{thm:GWX7}. 
Given $a\in\Z$ and 
$\{a_i\}_{i\in I}$ a finite set of
integer numbers,  define 
$$\binom{a}{\{a_i\}_{i\in I}}=\frac{a!}{\left(a- \sum_{ i\in I}a_i \right)!\prod_{i\in I}a_i!}.$$
Recall also that $(2l)!!=(2l-1)(2l-3)\ldots 1$.

Given a graph $\Gamma$, denote by $\lambda_{v,v'}$  the number of
edges between the distinct vertices 
$v$ and $v'$ of $\Gamma$, by  $\lambda_{v,v}$ twice the number of
loops of $\Gamma$ based at the vertex $v$,
and by $k^\circ_\Gamma$ the number of edges of $\Gamma$.

In this section, we consider curves in $X_7$ and $\X_8$ (and even in
$\X_2$ in the proofs of Theorems \ref{thm:GWX7} and \ref{thm:WX7}). In
order to avoid confusions, let us use the following notation: 
$D$ denotes the pullback of a generic line in both surfaces, and
$E_1,\ldots E_7$ (resp. $\widetilde E_1,\ldots \widetilde E_8$) denote
the $(-1)$-curves coming from the presentation of $X_7$ (resp. $\X_8$)
as a blow up
of $\C P^2$ (resp. of $\C P^2$ at eight points on a conic).
Finally, let $V_8\subset H_2(\X_8; \Z)\setminus \{0\}$ be the set of
effective classes $d\ne l\widetilde E_i$ with $l\ge
2$ or $i=7,8$.
\begin{defi}
A \emph{$X_7$-graph} is a connected graph $\Gamma$
together with 
 three quantities $d_v\in V_8$, $g_v\in\Z_{\ge
  0}$, and 
$\beta_v=\beta_{v,1}u_1+\beta_{v,2}u_2\in\Z^\infty_{\ge  0}$
 associated to each vertex $v$ of $\Gamma$, such 
 that $I\beta_v=d_v\cdot [E]$.

An isomorphism between $X_7$-graphs is  an isomorphism of graphs preserving the
 three quantities associated to each vertex.
\end{defi}
An  $X_7$-graphs is always considered up to isomorphism.
Given a $X_7$-graph $\Gamma$,  define 
$$d_\Gamma= \sum_{v\in Vert(\Gamma)}d_v,\quad    \mbox{and}\quad 
\beta_{\Gamma} = \sum_{v\in Vert(\Gamma)}\beta_{v}.$$
Given $g,k\in\Z_{\ge 0}$ and $d\in H_2(X_7;\Z)$ such that $d\cdot [D]\ge 1$
(if not  the corresponding Gromov-Witten invariants are straightforward to
compute), let $\SS_7(d,g,k)$  be the set of all pairs 
$(\Gamma,P_\Gamma)$ where 
\begin{itemize}
\item $\Gamma$ is a $X_7$-graph  such that 
$$\sum_{v\in Vert(\Gamma)}g_v +b_1(\Gamma)=g $$ 
and 
$$d= (d_\Gamma\cdot [D] + 2k)[D] -
\sum_{i=1}^6 \left(d_\Gamma\cdot [\widetilde E_i]  + k\right)[E_i] - 
\left(k^\circ_\Gamma+ \beta_{\Gamma,2} + d_\Gamma\cdot ([\widetilde E_7]+ [\widetilde E_8] )\right)[E_7];$$

\item $P_\Gamma=\bigcup_{v\in Vert(\Gamma)} U_v $ is a partition of the set
  $\{1,\ldots ,c_1(X_7)\cdot d-1+g \}$ such that
$|U_v|=  d_v\cdot [D] -1 +g_v+|\beta_v|$.

\end{itemize}

Given $(\Gamma,P_\Gamma)\in\SS_7(d,g,k)$, define
$$k^{\circ\circ}= k-\beta_{\Gamma,2} -k^\circ_\Gamma -d_\Gamma\cdot [\widetilde E_7].$$
Denote by $\sigma(\Gamma)$ the number of bijections of
$Vert(\Gamma)$ to itself which are induced by an automorphism of the
graph $\Gamma$.
Define the following complex multiplicities
for  $(\Gamma,P_\Gamma)\in\SS_7(d,g,k)$ and $v\in
Vert(\Gamma)$:
$$\mu^\C(v)= \lambda_{v,v} !!
\binom{\beta_{v,1}}{\{\lambda_{v,v'}\}_{v'\in Vert(\Gamma)}}
GW_{\X_8}^{0,\beta_v}(d_v,g_v),$$
and
$$\mu^\C(\Gamma,P_\Gamma)=\frac{I^{\beta_{\Gamma}}}{\sigma(\Gamma)}
\binom{\beta_{\Gamma,1} -2k^\circ_\Gamma} 
{ k^{\circ\circ}}
\prod_{v\ne v'\in Vert(\Gamma)}\lambda_{v,v'}!
\ \prod_{v\in Vert(\Gamma)}  \mu^\C(v).$$

Note that given $d$ and $g$, there exists only finitely 
elements in $\bigcup_{k\ge 0}\SS_7(d,g,k)$ with a positive multiplicity. 
Also given $(\Gamma,P_\Gamma)\in\SS_7(d,0,k)$, we have
$\lambda_{v,v'}\le 1$ (resp. $\lambda_{v,v}=0$) for each pair of distinct vertices
(resp. each vertex) of $\Gamma$.

\begin{exa}\label{exa:X7 2c1}
There exists element(s) 
in $\SS_7(2c_1(X_7),0,k)$ with a positive multiplicity in the following cases:

\begin{tikzpicture}
  [scale=.8,auto=left,vert/.style={circle,fill=blue!20, text
      centered, minimum width=25pt},
leg/.style={circle,fill=white, text centered}]
  \node[vert] (n6) at (0,1) {$v$};
  \node[leg] (n1) at (-1.5,1) {$\Gamma=$};

  \foreach \from/\to in {}
    \draw (\from) -- (\to);
\end{tikzpicture}

\begin{itemize}
\item $k=1, d_v=4[D]-\sum_{1}^8[\widetilde E_i]$: 
$$\mu^\C(\Gamma,P_\Gamma)=392;$$ 

\item $k=2, d_v=2[D]-a_7[\widetilde E_7] -a_8[\widetilde E_8]$, with  $a_7+a_8 +\beta_{v,2}=2$: 
$$\sum \mu^\C(\Gamma,P_\Gamma)=34;$$ 
\end{itemize}

  \begin{tikzpicture}
  [scale=.8,auto=left,vert/.style={circle,fill=blue!20, text
      centered, minimum width=25pt},
leg/.style={circle,fill=white, text centered}]
  \node[vert] (n6) at (0,1) {$v$};
  \node[vert] (n4) at (3,1)  {$v'$};
  \node[leg] (n1) at (-1.5,1) {$\Gamma=$};

  \foreach \from/\to in {n6/n4}
    \draw (\from) -- (\to);
\end{tikzpicture}

\begin{itemize}
\item$k=2, d_v=2[D]-[\widetilde E_i]-a_7[\widetilde E_7]
  -a_8[\widetilde E_8], d_{v'}=[\widetilde E_i]$, with $a_7+a_8
  +\beta_{v,2}=1$ and $1\le i\le 6$: 
$$\sum \mu^\C(\Gamma,P_\Gamma)=72;$$ 

\item $k=2, d_v=[D]
-a_7[\widetilde E_7] -a_8[\widetilde E_8], d_{v'}=[D]$, with $a_7+a_8 =1$: 
$$\sum \mu^\C(\Gamma,P_\Gamma)=12;$$
 
\end{itemize}

  \begin{tikzpicture}
  [scale=.8,auto=left,vert/.style={circle,fill=blue!20, text
      centered, minimum width=25pt},
leg/.style={circle,fill=white, text centered}]
  \node[vert] (n6) at (0,1) {$v'$};
  \node[vert] (n4) at (3,1)  {$v$};
  \node[vert] (n2) at (6,1)  {$v''$};
  \node[leg] (n1) at (-1.5,1) {$\Gamma=$};

  \foreach \from/\to in {n6/n4, n4/n2}
    \draw (\from) -- (\to);
\end{tikzpicture}

\begin{itemize}

\item $k=2, d_v=2[D] -[\widetilde E_i] -[\widetilde E_j], d_{v'}=[\widetilde E_i], d_{v''}=[\widetilde E_j]$, with $1\le i\ne
  j\le 6$: 
$$\sum \mu^\C(\Gamma,P_\Gamma)=30;$$

\item $k=2, d_v=[D] , d_{v'}=[\widetilde E_i], d_{v''}=[D]-[\widetilde E_i]$, with $1\le i\le 6$: 
$$\sum \mu^\C(\Gamma,P_\Gamma)=36.$$
 \end{itemize}

To get the above sum of multiplicities, I used Theorem \ref{NFD}
and Figure \ref{fig:X7} to compute the following numbers:
$$GW_{\X_8}(4[D]-\sum_{i=1}^8[\widetilde E_i],0)=392 \quad 
 \mbox{and}\quad  
GW^{0,u_1+u_2}_{\X_1}(2[D]-[\widetilde E_1],0)=GW^{0,2u_2}_{\C P^2}(2[D],0)=4.$$
 \begin{figure}[h]
\centering
\begin{tabular}{cc}
\begin{tabular}{ccc}
\includegraphics[width=1.1cm,
    angle=0]{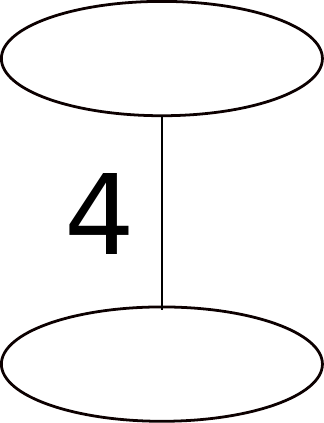}& \hspace{5ex} &
 $\sum \mu^\C(\D,m) = 16 $
\end{tabular} &
\begin{tabular}{ccc}
\includegraphics[width=1.1cm,
    angle=0]{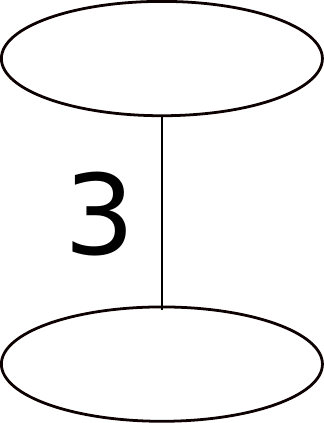}& \hspace{5ex}&
 $\sum \mu^\C(\D,m) = 72 $
\end{tabular} 
\\ \\
\begin{tabular}{c|c|c|c|c|c || c | c }
$s\backslash  \kappa,\epsilon$ & 0 & 1& 2 & 3 & 4 & 0& 1
\\ \hline
0 & 0& 0& 0& 0& 0& 0&0
\\ \hline
1 & 0&0 &0 &0 &0 & 0&0 
\end{tabular}&
\begin{tabular}{c|c|c|c|c|c || c | c}
$s\backslash  \kappa,\epsilon$ & 0 & 1& 2 & 3 & 4 & 0& 1
\\ \hline
0 & 8 &6& 4& 2& 0 & 0&0
\\ \hline
1 & 24& 18& 12& 6&0 & 0&0
\end{tabular}

\\ \\
\begin{tabular}{ccc}
\includegraphics[width=1.1cm,
    angle=0]{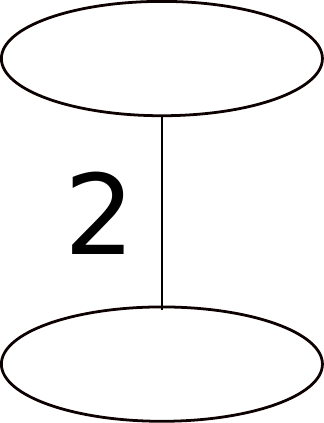}& \hspace{5ex} &
 $\sum \mu^\C(\D,m) = 112 $
\end{tabular} &
\begin{tabular}{ccc}
\includegraphics[width=1.1cm,
    angle=0]{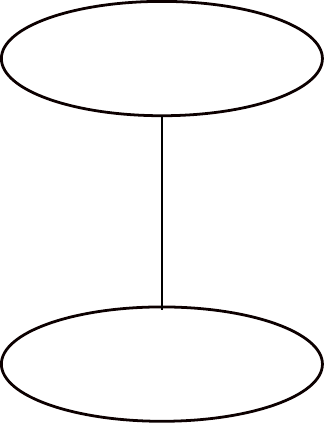}& \hspace{5ex}&
 $\sum \mu^\C(\D,m) = 56 $
\end{tabular} 
\\ \\
\begin{tabular}{c|c|c|c|c|c || c | c}
$s\backslash  \kappa,\epsilon$ & 0 & 1& 2 & 3 & 4& 0& 1
\\ \hline
0 & 0&0 &0 &0 &0 & 16& 16
\\ \hline
1 & 0&0 & 0& 0& 0& 8& 8
\end{tabular}&
\begin{tabular}{c|c|c|c|c|c|| c | c }
$s\backslash  \kappa,\epsilon$ & 0 & 1& 2 & 3 & 4& 0& 1
\\ \hline
0 & 56& 26& 12& 6& 0& 0&0 
\\ \hline
1 & 56& 26& 12& 6& 0& 0&0 
\end{tabular}

\\ \\
\begin{tabular}{ccc}
\includegraphics[width=1.1cm,
    angle=0]{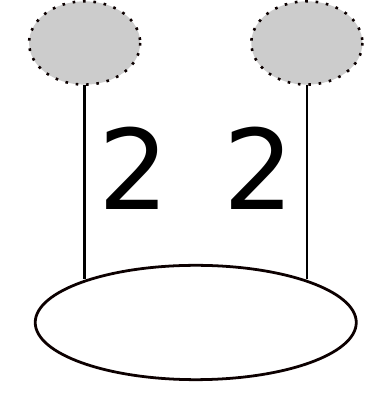}& \hspace{5ex} &
 $\sum \mu^\C(\D,m) = 16 $
\end{tabular} &
\begin{tabular}{ccc}
\includegraphics[width=1.1cm,
    angle=0]{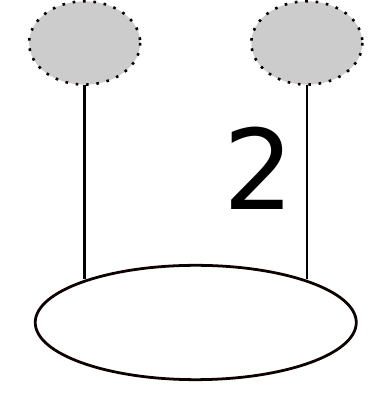}& \hspace{5ex}&
 $\sum \mu^\C(\D,m) = 64 $
\end{tabular} 
\\ \\
\begin{tabular}{c|c|c|c|c|c || c | c }
$s\backslash  \kappa,\epsilon$ & 0 & 1& 2 & 3 & 4& 0& 1
\\ \hline
0 & 0& 0& 0& 0& 0& 16& 0
\\ \hline
1 & 0& 0& 0& 0& 0& 8& 0 
\end{tabular}&
\begin{tabular}{c|c|c|c|c|c || c | c }
$s\backslash  \kappa,\epsilon$ & 0 & 1& 2 & 3 & 4 & 0& 1
\\ \hline
0 & 0& 0& 0& 0& 0& 0& 0
\\ \hline
1 & 0&0 &0 &0 &0 &0 & 0
\end{tabular}

\\ \\
\begin{tabular}{ccc}
\includegraphics[width=1.1cm,
    angle=0]{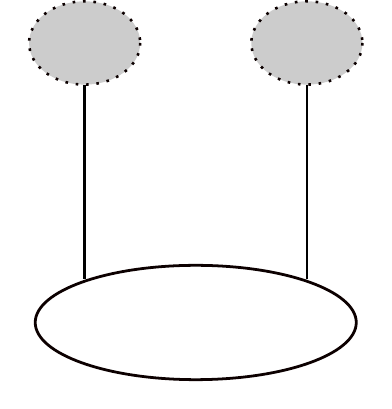}& \hspace{5ex} &
 $\sum \mu^\C(\D,m) = 56 $
\end{tabular} &
\begin{tabular}{ccc}
\includegraphics[width=1.1cm,
    angle=0]{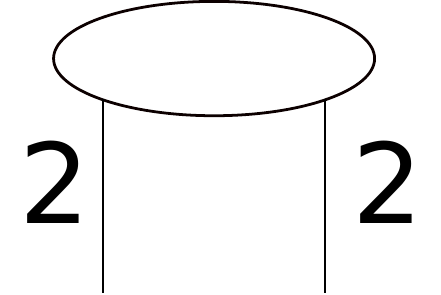}& \hspace{5ex}&
 $ \mu^\C(\D,m) = 4 $
\end{tabular} 
\\ \\
\begin{tabular}{c|c|c|c|c|c || c | c }
$s\backslash  \kappa,\epsilon$ & 0 & 1& 2 & 3 & 4& 0& 1
\\ \hline
0 & 56& 30& 12& 2&0 & 0&0
\\ \hline
1 & 56& 30&12 &2 & 0& 0& 0
\end{tabular}&
\begin{tabular}{ c|c|c| c|c|c|c||c|c}
\multicolumn{2}{c|}{$s \backslash \kappa,\epsilon $}  & $0$& $1$& $2$& $3$& $4$& $0$&$1$
\\\hline
$0$ &  $\beta^\Re_2=2$&  4& 4  &4  &4  &4&  4&4
\\\hline
$1$ & $\beta^\Im_2=1$ &  2&2    &2  &2  &2& 2&2
\\ 
\end{tabular}

\\ \\
\begin{tabular}{ccc}
\includegraphics[width=1.1cm,
    angle=0]{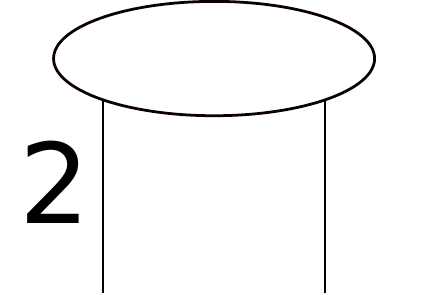}& \hspace{5ex} &
 $ \mu^\C(\D,m) = 4 $
\end{tabular} &
\begin{tabular}{ccc}
\includegraphics[width=1.1cm,
    angle=0]{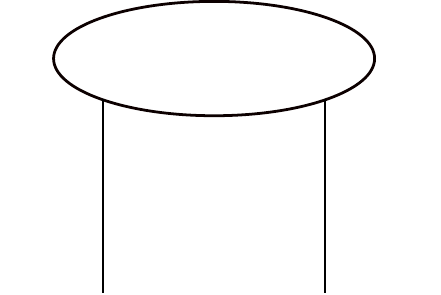}& \hspace{5ex}&
 $ \mu^\C(\D,m) = 1 $
\end{tabular} 
\\ \\
\begin{tabular}{ c|c| c|c|c|c||c|c}
 $s \backslash \kappa,\epsilon $  & $0$& $1$& $2$& $3$& $4$& $0$&$1$
\\\hline
$0$ & 4&4  &4   &4  &4&  0&0  
\\\hline
$1$ & 0 &0  &0    &0  &0&  0& 0 
\\ 
\end{tabular}&
\begin{tabular}{ c|c|c| c|c|c|c||c|c}
\multicolumn{2}{c|}{$s \backslash \kappa,\epsilon $}  & $0$& $1$& $2$& $3$& $4$& $0$&$1$
\\\hline
$0$ &  $\beta^\Re_1=2$&  1&1   & 1 &1  &1&  0&0
\\\hline
$1$ & $\beta^\Im_1=1$ & 1 &1    &1  &1  &1& 1&1
\\ 
\end{tabular}
\end{tabular}
\caption{Marked floor diagrams of genus 0 and degree 4 and 2 used in
  Example \ref{exa:X7 2c1}}
\label{fig:X7}
\end{figure}
\end{exa}

Next theorem reduces the computation of  $GW_{X_7}$ to
the computation of $GW_{\X_8}$.
\begin{thm}\label{thm:GWX7}
Let $g\ge 0$ and $d\in H_2(X_7;\Z)$ such that $d\cdot [D]\ge 1$. Then one
has
$$GW_{X_7}(d,g)=\sum_{k\ge 0}\ \ \sum_{(\Gamma,P_\Gamma)\in\SS_7(d,g,k)}\mu^\C(\Gamma,P_\Gamma) .$$
\end{thm}
\begin{proof}
The proof follows the lines of Section \ref{sec:Li deg}. 
 Let $E=\X_6\cap \X_2$, and 
$\widetilde E'_1$ and $\widetilde E'_2$ be the two $(-1)$-curves in
 $\X_2$
intersecting  $E$. Denote  respectively by  $p_7$ and $p_8$ the
corresponding intersection points.
Let $T_1,\ldots, T_6$ be the six rational curves  in $\X_2$ 
 such that $T_i^2=T_i\cdot \widetilde E'_2=0$, and
such that $T_i$ passes through $\widetilde E_i\cap E$.
Denote also by 
$\widetilde E'_7$ the $(-1)$-curve arising from the blown up of $\C
P^1 \times \C P^1$ at the point $p_0(0)$ 
(alternatively, $\widetilde E'_7$ is the $(-1)$-curve in $\X_2$ which
does not intersect $E$).
We may further assume that we chose $\widetilde E'_1$ and $E_1,\ldots,
E_7$ 
such that
the seven curves $\widetilde E_1\cup T_1,\ldots \widetilde
E_6\cup T_6, $ and 
$\widetilde  E'_7$ in $\pi^{-1}(0)$ respectively deform  to 
$E_1,\ldots, E_7$ in $X_7$. 

Let us choose
$\x(t)$ a generic set of $c_1(X_7)\cdot d-1+g$ sections $\C\to\YY$
such that $\x(0)\subset \X_6\setminus \X_2$. 
For each $t\ne 0$, we denote by $\CC(d,g,\x(t))$
 the set of maps $f:C\to X_7$ with $C$ an irreducible curves of geometric 
genus $g$,  such that $f(C)$ realizes
 the class $d$ in $H_2(X_7;\Z)$, and contains  all points in $\x(t)$.
We denote by $\CC(d,g,\x(0))$ the set of limits, as $t$ goes to $0$, of
elements $\CC(d,g,\x(t))$.

Exactly as in the proof of Proposition \ref{prop:degeneration}, we
have that the set $\CC(d,g,\x(0))$ is finite, and that its cardinal 
does not depend on
$\x(0)$ as long as this latter is generic. We also deduce that if 
$\overline f:\overline C\to\pi^{-1}(0)$ is an element of 
$\CC(d,g,\x(0))$, and $p$ and $p'$ are two points on $\overline C$ with
the same image on $E$, then 
$\{\overline f(p)\}=E \cap \widetilde E_i$ or $\{\overline
f(p)\}=E \cap \widetilde E'_i$.
Denote by $\overline C_{\X_6}$ the union of  irreducible
components of $\overline C$ mapped to $\X_6$, and  by $\overline C_{\X_2}$ the union
of those mapped to $\X_2$.

The same proof
as for Corollary \ref{cor:no sEi} combined with Proposition
\ref{prop:shoshu} yields that the restriction of $\overline f$
on each irreducible component of $\overline C$
 is
birational onto its image.
Proposition \ref{prop:initial values} applied to curves in
$\X_2$ gives that if $\overline C'$ is an irreducible component of
$\overline C_{\X_2}$, then one of the four following situations occurs:
\begin{enumerate}
\item $\overline f(\overline C')$ realizes the class $[D]$, and
  intersect $E$ in two points determined by
  $\overline f(\overline C_{\X_6})$, distinct from $p_7$ and $p_8$;
 
\item $\overline f(\overline C')$ realizes the class $[D]$, and
 is tangent to $E$ at a point  determined by
  $\overline f(\overline C_{\X_6})$, distinct from $p_7$ and $p_8$;

\item   $\overline f(\overline C')$ realizes the class $[D]-[\widetilde
  E'_i]$, $i=1,2$, and
 intersects $E$ in a point  determined by
  $\overline f(\overline C_{\X_6})$, distinct from $p_7$ and $p_8$;

\item   $\overline f(\overline C')$ realizes the class $[\widetilde
  E'_i]$, $i=1,2$.
\end{enumerate}
Let $a$ be the number of components in 
cases $(1)$ and $(2)$, let $b$ be the number of components in  case $(3)$ with
$i= 1$, let $c$ be the number of components in  case $(4)$ with
$i= 1$, and let $d_{\X_6}$ be the homology class realized by 
$\overline f(\overline C_{\X_6})$ in $H_2(\X_6;\Z)$. Then one has
$$d=\left( d_{\X_6}\cdot [D] -2(a+b+c)\right)[D] -
\sum_{i=1}^6\left(d_{\X_6}\cdot [\widetilde E_i]
 +a+b+c \right)[E_i] + (a+c)[E_7]  .$$

Let us construct  an $X_7$-graph
$\Gamma_{\overline f}$ out of an element 
$\overline f$ of  $\CC(d,g,\x(0))$ as follows:
\begin{itemize}
\item  vertices $v$ of $\Gamma_{\overline f}$ are in one-to-one correspondence with irreducible components $\overline C_v$ 
of $\overline
C_{\X_6}$; the quantities $d_v$, $g_v$, and $\beta_v$ record
respectively the homology class realized by 
$\overline f(\overline C_v)$ in $\X_6$ blown up at $p_7$ and $p_8$,
the genus of $\overline C_v$, and the intersections  of
    the strict transforms of 
$\overline f(\overline C_v)$ and $E$ in $\X_8$;
\item edges of $\Gamma_{\overline f}$ are in one-to-one correspondence with irreducible components of $\overline
C_{\X_2}$ in case $(1)$ above; 
each such component  $\overline C'$
correspond to an edge between $v$ and $v'$, where 
$\overline C_v$ and
$\overline C_{v'}$ are the components of $\overline
C_{\X_6}$ intersecting $\overline C'$ (note that we may have $v=v'$).
\end{itemize}

If $\x(0)=\{p_1,\ldots,p_{c_1(X_7)\cdot d-1+g}\}$,
then 
define 
$U_v\subset\{1,\ldots, c_1(X_7)\cdot d-1+g\}$ for $v\in
Vert(\Gamma_{\overline f})$ 
as the set corresponding to
points in $\x(0)$ contained in $\overline f(\overline C_v)$. 
Note that $|U_v|= [D]\cdot d_v-1+g_v +[\beta_v|$ by the same arguments as
in the proof of Proposition \ref{prop:degeneration}.
Finally
 denote by $P_{\overline f}$  the partition of $\{1,\ldots, c_1(X_7)\cdot
d-1+g\}$ defined by the sets $U_v$, where $v$ ranges over all vertices
of $\Gamma$.

It follows from the above arguments that
 $(\Gamma_{\overline f},P_{\overline f})$ is an element of
$\SS_7(d,g,a+b+c)$. The theorem now follows from the  fact that 
for any $(\Gamma,P_{\Gamma})\in \SS_7(d,g,a+b+c)$, 
the multiplicity $\mu^\C(\Gamma,P_{\Gamma})$ is
precisely the number of elements of $\CC(d,g,\x(t))$ converging, as $t$
goes to 0, to an element $\overline f$ in $\CC(d,g,\x(0))$ with
$(\Gamma_{\overline f},P_{\overline f})=(\Gamma,P_{\Gamma})$.
\end{proof}

\begin{rem}
I consider the degeneration $\YY$ having in mind the enumeration  of real
curves, see Section \ref{sec:WX7}. If one is only interested in the
computation of Gromov-Witten invariants of $X_7$, then it is probably
 simpler
to consider the degeneration of $X_7$ to 
$\X_7\cup \X_1$, the resulting formula being the same. 
In this perspective, Theorem \ref{thm:GWX7} is
then analogous to {\cite[Theorem 2.9, Example 2.11]{Br18}}.
\end{rem}

\begin{exa}
Thanks to Theorem \ref{thm:GWX7} and Example \ref{exa:X7 2c1}, one
verifies that 
$$GW_{X_7}(2c_1(X_7),0)=576.$$
 Performing analogous computations
in genus up to $3$, we obtain the value listed in Table \ref{tab:comp
  X7}.
 The value in the rational case has been
first  computed  by G\"ottsche and
Pandharipande in {\cite[Section 5.2]{PanGot98}}. The cases of higher genus have been first
treated in \cite{Shu13}.
\begin{table}[!h]
\begin{center}
\begin{tabular}{ c|c| c|c|c}
 $g $  & $0$& $1$& $2$& $3$
\\\hline
 $GW_{X_7}(2c_1(X_7),g)$   & 576  & 204  & 26 & 1 

\end{tabular}
\end{center}

\medskip
\caption{$GW_{X_7}(2c_1(X_7),g)$}
\label{tab:comp X7}
\end{table}
The value $GW_{X_7}(2c_1(X_7),1)=204$ corrects  the incorrect value
announced in {\cite[Example 3.2]{Shu13}}. 
\end{exa}

\subsection{Welschinger invariants}\label{sec:WX7}

Denote by $X_7(\kappa)$ with $\kappa= 0,\ldots, 3$,
and by $X_7^\pm(4)$ the surface $X_7$ 
equipped with the real structure such that:
$$\R X_7(\kappa)=\R P^2_{7-2\kappa},\quad 
\R X_7^-(4)=\R P^2\sqcup \R P^2,
\quad \R X_7^+(4)=S^2\sqcup \R P^2_1.$$ 
Recall  that these are all real structures on $X_7$ with a
non-orientable real part, and represent half of the possible real
structures on $X_7$.
Note that
$$\chi(\R X_7^\pm(\kappa))=-6+2\kappa .$$

For $\kappa=0,\ldots,3$,  define the two involutions $\tau^0_\kappa$ and
$\tau^1_\kappa$ on $H_2(\X_8;\Z)$ as follows:
$\tau^0_\kappa$ (resp. $\tau^1_\kappa$) fixes the elements $[D]$ and
$[\widetilde E_i]$ with $i\in\{ 2\kappa+1,\ldots,8\}$ (resp. $i\in\{
2\kappa+1,\ldots,6\}$),
and exchanges the elements $[\widetilde E_{2i-1}]$ and $[\widetilde E_{2i}]$ with 
$i\in\{1,\ldots,\kappa\}$ (resp. $i\in\{1,\ldots,\kappa,4\}$).
These two involutions  take into account that for each real
structure on $\X_6$, there are two possible real structures on $\X_2$,
depending on the real structure on $\pi^{-1}(0)$.

\medskip
From now on, let us  fix $g=0$,  an integer $\kappa\in\{0,\ldots,3\}$, and a
class
 $d\in H_2(X_7;\Z)$.
Set 
 $\zeta=  c_1(X_7)\cdot d-1$ and
$A=\{1,\ldots, \zeta\}$,
and choose two integer  $r,s\ge 0$ such that 
$\zeta=r+2s$.
Define the involution $\rho_s$ on $A$ as follows:
$\rho_{s|\{2i-1,2i\}}$ is the non-trivial transposition for $1\le i\le
s$, and  $\rho_{s|\{2s+1,\zeta \}}=Id$.

Given $\epsilon\in\{0,1\}$,
  denote by $\R\SS_7^\epsilon(d,k,s,\kappa)$
 the set of triples $(\Gamma,P_\Gamma,\tau)$
where
\begin{itemize}
\item  $(\Gamma,P_\Gamma)\in\SS_7(d,0,k)$;
\item $\tau: \Gamma\to \Gamma$ is an involution such that
for any vertex $v$ of $\Gamma$, one has
$\beta_{v}=\beta_{\tau(v)}$,
$d_{\tau(v)}=\tau_\kappa^\epsilon(d)$, and  $\rho_s(U_v)=U_{\tau(v)}$;
\item to each vertex $v$ fixed by $\tau$ is associated a decomposition
$\beta_v=\beta^\Re_{v}+ 2\beta^\Im_v$ with
 $\beta^\Re_{v},\beta^\Im_v\in\Z_{\ge0}^\infty$. 
\end{itemize}

Given $(\Gamma,P_\Gamma,\tau)\in \R\SS_7^{\epsilon}(d,k,s,\kappa)$, 
denote by $\sigma(\Gamma,\tau)$ the number of bijections of
$Vert(\Gamma)$ to itself which are induced by an automorphism of the
graph $\Gamma$ commuting with $\tau$.
Note that $\tau=Id$ if $s=0$.
Denote also by $Vert_\Im(\Gamma)$ (resp. $k^{\circ,\Im}_{\Gamma}$) 
the set of pairs of vertices 
 (resp. the number of pairs of edges) exchanged by $\tau$,
 and by $Vert_\Re(\Gamma)$ (resp. $k^{\circ,\Re}_{\Gamma}$) 
the set of vertices (resp. the number of  edges) 
fixed by $\tau$. 
Next, define
$$
\beta^\Re_{\Gamma}=\sum_{v\in Vert_\Re(\Gamma)}\beta^\Re_{v},\quad\mbox{and}\quad
\beta^\Im_{\Gamma}= \sum_{v\in Vert_\Re(\Gamma)}\beta^\Im_{v} + \sum_{\{v,v'\}\in
    Vert_\Im(\Gamma)}\beta_{v}.$$
Let us associate different real multiplicities to elements of
$\R\SS_7^{\epsilon}(d,k,\kappa)$, acounting all possible
smoothings  of $\R\pi^{-1}(0)$.

Given $\{v,v'\}\in Vert_\Im(\Gamma)$,  define
$$\mu^{\R}(\{v,v'\})=(-1)^{d_v\cdot d_{v'}} 
\binom{\beta_{v,1}}{\{\lambda_{v,v''}\}_{v''\in Vert(\Gamma)}}
GW_{\X_8(\kappa)}^{0,\beta_v}(d_v,0). $$

Let $v\in Vert_\Re(\Gamma)$. Denote  respectively by $r_v$ and $s_v$ the
number of points in $U_v$ fixed by $\rho_{s}$ and the number of
 pairs of points in $U_v$  exchanged by $\rho_{s}$. By definition we
 have
 $|U_v|=r_v+2s_v$. Denote also by $k_v^{\circ,\Im}$ the number of pairs of
 edges of $\Gamma$ adjacent to $v$ and exchanged by $\tau$.
Define 
$$\mu^{\R,\epsilon}_{s,\kappa}(v)=
2^{k_v^{\circ,\Im}}
\binom{\beta^\Re_{v,1}}{\{\lambda_{v,v'}\}_{v'\in Vert_\Re(\Gamma)}}
\binom{\beta^\Im_{v,1}}{\{\lambda_{v,v'}\}_{\{v',v''\}\in Vert_\Im(\Gamma)}}
FW_{\X_8(\kappa+\epsilon)}^{0,\beta^\Re_v,0,\beta_v^\Im}(d_v^\epsilon,s_v), $$
where $d_v^0=d_v$, and $d_v^1$ is obtained from $d_v$ by
exchanging\footnote{This  additional complication is purely formal and comes from
  the convention used to define the numbers $FW$ in section
  \ref{sec:real FD}.}
 the
coefficients of $\widetilde E_{2\kappa-1}$ and $\widetilde E_7$, and $\widetilde E_{2\kappa}$ and $\widetilde E_8$.
Define also
$$\eta^{\R}_{s,\epsilon}(v)=
2^{k_v^{\circ,\Im}}
\binom{\beta^\Re_{v,1}}{\{\lambda_{v,v'}\}_{v'\in Vert_\Re(\Gamma)}}
\binom{\beta^\Im_{v,1}}{\{\lambda_{v,v'}\}_{\{v',v''\}\in Vert^\Im(\Gamma)}}
FW_{\X_8(4),\epsilon}^{0,\beta^\Re_v,0,\beta_v^\Im}(d_v,s_v), $$
and
$$\nu^{\R}_{s,\epsilon}(v)=
FW_{\X_8(4),\epsilon,\epsilon}^{0,\beta^\Re_v,0,\beta_v^\Im}(d_v,s_v). $$

Let $\R\SS^{0}_{7,m}(d,k,s,\kappa)$ be the subset of
$\R\SS^{0}_7(d,k,s,\kappa)$ formed by elements with 
$\beta^\Re_{\Gamma,2}=0$.
Given $(\Gamma,P_\Gamma,\tau)\in \R\SS^{0}_{7,m}(d,k,s,\kappa)$, 
 define the following multiplicity:
$$\mu^{\R,0}_{s,\kappa}(\Gamma,P_\Gamma,\tau)=
\frac{(-1)^{k^{\circ,\Im}_\Gamma+\beta^\Im_{\Gamma,2}}
\ I^{\beta^\Im_\Gamma}}{\sigma(\Gamma,\tau)}
\prod_{v\in
  Vert_\Re(\Gamma)}\mu^{\R,0}_{s,\kappa}(v)
\prod_{\{v,v'\}\in
Vert_\Im(\Gamma)}\mu^{\R}(\{v,v'\}) \times$$
$$\times 
\sum_{k^{\circ\circ}=r'+2s'}
\binom{\beta^\Re_{\Gamma,1} -2k^{\circ,\Re}_\Gamma}{r'}
\binom{\beta^\Im_{\Gamma,1}-2k^{\circ,\Im}_\Gamma}{s'}. $$

Let $\R\SS^{1}_{7,m}(d,k,s,\kappa)$ be the subset of
$\R\SS_7^{1}(d,k,s,\kappa)$ formed by elements with 
$\beta^\Re_{\Gamma}=2k_\Gamma^{\circ,\Re}u_1$ and $k^{\circ\circ}=\beta_{\Gamma,1}^\Im-2k_\Gamma^{\circ,\Im}$.
Given $(\Gamma,P_\Gamma,\tau)\in \R\SS^{1}_{7,m}(d,k,s,\kappa)$,
 define the following multiplicity
$$\mu^{\R,1}_{s,\kappa}(\Gamma,P_\Gamma,\tau)=
\frac{(-1)^{k^{\circ,\Im}_\Gamma} \ 
(-2)^{|\beta^\Im_{\Gamma}|-2k^{\circ,\Im}_\Gamma}}
{\sigma(\Gamma,\tau)}
\prod_{v\in
  Vert_\Re(\Gamma)}\mu^{\R,1}_{s,\kappa}(v)
\prod_{\{v,v'\}\in
Vert_\Im(\Gamma)}\mu^{\R}(\{v,v'\}).$$

Note that $\R\SS^{1}_{7,m}(d,k,s,3)$ is composed of elements with 
$\beta^\Re_{\Gamma}=k_\Gamma^{\circ,\Re}=0$.
Given $(\Gamma,P_\Gamma,\tau)\in \R\SS^{1}_{7,m}(d,k,s,3)$ and
$\epsilon\in\{0,1\}$, define the following multiplicity
$$\eta^{\R}_{s,\epsilon}(\Gamma,P_\Gamma,\tau)=
\frac{(-1)^{k^{\circ,\Im}_\Gamma} \ 
(-2)^{|\beta^\Im_{\Gamma}|-2k^{\circ,\Im}_\Gamma}}
{\sigma(\Gamma,\tau)}
\prod_{v\in
  Vert_\Re(\Gamma)}\eta^{\R}_{s,\epsilon}(v)
\prod_{\{v,v'\}\in
Vert_\Im(\Gamma)}\mu^{\R}(\{v,v'\}).$$

Let $\R\SS^1_{7,2}(d,k,s,3)$ (resp. $\R\SS^1_{7,3}(d,k,s,3)$) be the subset of
$\R\SS^{1}_{7}(d,k,s,3)$ formed by elements with 
$k^\circ_\Gamma=\beta_{\Gamma,1}=0$ (resp.
$k^\circ_\Gamma=\beta_{\Gamma,1}=\beta^\Re_{\Gamma,2}=0$).
Note that any element of  $\R\SS^1_{7,2}(d,k,s,3)$ or 
$\R\SS^1_{7,3}(d,k,s,3)$ has a single vertex.

In the following theorem, the connected component of $\R X_7^+(4)$
with Euler characteristic $\epsilon$ is denoted by $L_\epsilon$.

\begin{thm}\label{thm:WX7}
Let    $d\in H_2(X_7;\Z)$ such that
$d\cdot [D]\ge 1$, and
 $r,s\in\Z_{\ge 0}$ such that
 $c_1(X_7)\cdot d-1 =r+2s$. Then  one has
\[\begin{aligned}
& W_{X_7(\kappa)}(d,s) =
\sum_{k\ge 0}\ \ \sum_{(\Gamma,P_\Gamma,\tau)\in\R\SS^0_{7,m}(d,k,s,\kappa)}\mu^{\R,0}_{s,\kappa}(\Gamma,P_\Gamma,\tau)
\quad \mbox{if } \kappa\in\{0,\ldots, 3\},\\&
W_{X_7(\kappa+1)}(d,s)=
\sum_{k\ge 0}\ \ \sum_{(\Gamma,P_\Gamma,\tau)\in\R\SS^1_{7,m}(d,k,s,\kappa)}\mu^{\R,1}_{s,\kappa}(\Gamma,P_\Gamma,\tau)
\quad \mbox{if }\kappa\in\{0,\ldots, 2\},\\&
W_{X_7^-(4),\R P^2,\R X_7^-(4)}(d,s)=
\sum_{k\ge 0}\ \ \sum_{(\Gamma,P_\Gamma,\tau)\in\R\SS^1_{7,m}(d,k,s,3)}\eta^{\R}_{s,\epsilon}(\Gamma,P_\Gamma,\tau)
\quad \forall \epsilon\in\{0,1\},\\&
W_{X_7^+(4),L_{2\epsilon},\R X_7^+(4)}(d,s)=
\sum_{k\ge
  0}\ \ \sum_{(\Gamma,P_\Gamma,\tau)\in\R\SS^1_{7,m}(d,k,s,3)}\eta^{\R}_{s,\epsilon}(\Gamma,P_\Gamma,\tau)
\quad \forall \epsilon\in\{0,1\}, \\&
W_{X_7^-(4),L_1,L_1}(d,s)=
\sum_{k\ge
  0}\ \ \sum_{(\Gamma,P_\Gamma,\tau)\in\R\SS^1_{7,2}(d,k,s,3)}2^{\beta^\Re_{\Gamma,2}+\beta^\Im_{\Gamma,2}}
\nu^{\R}_{s,1}(v), \\&
W_{X_7^-(4),L_1,L_1}(d,s)=
\sum_{k\ge
  0}\ \ \sum_{(\Gamma,P_\Gamma,\tau)\in\R\SS^1_{7,3}(d,k,s,3)}(-2)^{\beta^\Im_{\Gamma_2}}
\nu^{\R}_{s,0}(v), 
\\&
W_{X_7^+(4),L_{0},L_0}(d,s)=
\sum_{k\ge
  0}\ \ \sum_{(\Gamma,P_\Gamma,\tau)\in\R\SS^1_{7,2}(d,k,s,3)}2^{\beta^\Re_{\Gamma,2}+\beta^\Im_{\Gamma,2}}
\nu^{\R}_{s,0}(v), 
\\&
W_{X_7^+(4),L_2,L_2}(d,s)=
\sum_{k\ge
  0}\ \ \sum_{(\Gamma,P_\Gamma,\tau)\in\R\SS^1_{7,3}(d,k,s,3)}(-2)^{\beta^\Im_{\Gamma_2}}
\nu^{\R}_{s,1}(v). 
\end{aligned}\]
\end{thm}
\begin{proof}
We use the notations introduced in  Section
\ref{sec:X7 strategy}, and in the proof of Theorem \ref{thm:GWX7}.
Denote by $\widetilde L'_\epsilon$ the connected component of
$\R\X_2(\kappa)\setminus \R E$ with Euler characteristic $\epsilon$.
In what follows, by a real structure on $\YY$, I mean a real structure
turning 
$\pi:\YY\to \C$ into a
real map.

According to Remark \ref{rem:real structure}, there exists a flat
degeneration $ \pi:\YY\to\C$ of $X_7$ as in Proposition \ref{prop:degen X7}
endowed with a real structure such that one of the following holds:
\begin{itemize}
\item $\R\pi^{-1}(0)=\X_6(\kappa)\cup\X_2(0)$ with  $0\le \kappa\le
  3$;
in this case the two points $p_7$ and $p_8$ are real, and 
$\R \pi^{-1}(t)= X_7(\kappa)$ for $t\ne 0$. 

\item  $\R\pi^{-1}(0)=\X_6(\kappa)\cup\X_2(1)$ with $0\le \kappa\le
  2$;
in this case the two points $p_7$ and $p_8$ are complex conjugated, and 
$\R \pi^{-1}(t)= X_7(\kappa+1)$ for $t\ne 0$. 

\item $\R\pi^{-1}(0)=\X_6(3)\cup\X_2(1)$ and the component
  $\widetilde L'_\epsilon$ of $\X_2(1)$ is glued to the 
  component $\widetilde L_0$ of $\X_6(3)$ in the smoothing of
  $\R\pi^{-1}(0)$;
in this case the two points $p_7$ and $p_8$ are complex conjugated,
and for $t\ne 0$ one has
$\R \pi^{-1}(t)= X_7^-(4)$ if $\epsilon=1$, and 
$\R \pi^{-1}(t)= X_7^+(4)$ if $\epsilon=0$. 
\end{itemize}

Now  choose the configuration $\x(t)$ to be real, with $r$ real
points and $s$ pairs of complex conjugated points, and such that $\x(0)$
is a real configuration whose existence is attested by Theorem \ref{WFD}. 
Let $\overline f:\overline C\to \pi^{-1}(0)$ be a real element of 
$\CC(d,0,\x(t))$. If $\R\pi^{-1}(0)=\X_6(\kappa)\cup \X_2(\epsilon)$,
then the
 complex conjugation induces an involution
$\tau_{\overline f}$ on $(\Gamma_{\overline f},P_{\overline f})$ such
that 
$(\Gamma_{\overline f},P_{\overline f},\tau_{\overline f})\in
\R\SS_7^\epsilon(d,k,s,\kappa) $.

Let $(\Gamma,P_{\Gamma},\tau)\in \R\SS_7^\epsilon(d,k,s,\kappa) $, and
let $D(t)$ be the set of real elements in   $\CC(d,g,\x(t))$ converging, as $t$
goes to 0, to a real element $\overline f$ in $\CC(d,g,\x(0))$ with
$(\Gamma_{\overline f},P_{\overline f},\tau_{\overline
  f})=(\Gamma,P_{\Gamma},\tau)$.
It remains us to 
 compute the total contribution
to the various Welschinger invariants of 
all   elements of $D(t)$.
To do so, we just note that, exactly as in Corollary
\ref{cor:total node}, a nodal pair of any deformation in
$\CC(d,0,\x(t))$ of $\overline f$ is either a deformation of a nodal pair of 
$\overline f$ not mapped to $E$, or is contained in the
deformation of a small
neighborhood of a point in $\overline
f^{-1}\left(E\setminus 
\left(\widetilde E'_1\cup\widetilde E'_2\cup_{i=1}^6 \widetilde E_i \right) \right)$.

\medskip
For the first four identities of the theorem, the only thing to
compute is the parity of the number $a_{\overline f}$ of nodal pairs
not mapped to $E$ and
 contained in two
complex conjugated irreducible components
of $\overline C$.
Let  $\overline C_{v}$ and  $\overline C_{v'}$ be  two such  complex
conjugated irreducible components of $\overline C$. 
Suppose that 
$\overline f(\overline C_{v})$ is not a real $(-1)$-curve, and 
passes $b$ times through  
$\R E\cap 
\left(\widetilde E'_1\cup\widetilde E'_2\cup_{i=1}^6 \widetilde E_i
\right)$, 
then 
  the contribution 
of $\overline C_{v}$ and  $\overline C_{v'}$ 
to $a_{\overline f}$ is equal to $d_v\cdot d_{v'} - b$ modulo two.
Note that to any real  intersection point of $\overline f(\overline C_{v})$ with
$E\cap 
\left(\widetilde E'_1\cup\widetilde E'_2\cup_{i=1}^6 
\widetilde E_i \right)$ will correspond 
a pair of complex conjugated irreducible
components  $\overline C_{w}$ and  $\overline C_{w'}$ of $\overline C$ such that 
$\overline C_{w}\cap \overline C_{v}\ne\emptyset $ and $\overline
f(\overline C_{w})=f(\overline C_{w'})=\widetilde E_i$ or $\widetilde E'_i$. 
Moreover if  $\overline C_{v}\subset \overline C_{\X_2}$, then 
 $d_v\cdot d_{v'}= 1$ if $d_v=[D]$, and  $d_v\cdot d_{v'}= 0$ if
$d_v=[D]-[\widetilde E_i]$.
Altogether we obtain
$$a_{\overline f}=\sum_{\{v,v'\}\in Vert_\Im(\Gamma)}d_v\cdot d_{v'} +
k^{\circ,\Im}_\Gamma+\beta^\Im_{\Gamma_2}\ \quad \mbox{mod}\quad  2$$
which   only depends on  $(\Gamma_{\overline f},P_{\overline
  f},\tau_{\overline f})$. Now the first four identities of Theorem
\ref{thm:WX7} follow from Proposition \ref{prop:real degeneration}.

\medskip
Suppose now that $\R \pi^{-1}(t)= X_7^\pm(4)$.
Clearly if $k^{\circ,\Re}\ne 0$, then $D(t)=\emptyset$. Next,
if $k^{\circ,\Im}\ne 0$, 
  Lemma \ref{lem:line conic} implies that the total contribution
to $W_{X^\pm_7(4),L_\epsilon,L_\epsilon}(d,s)$
of  elements of $D(t)$ is equal to 0. 
To end the proof of Theorem \ref{thm:WX7}, it remains to notice that
 in $\X_2(1)$, the real part of the line tangent to $E$ at a real point is
 contained in $\widetilde L'_0\cup\R E$, and to use one more time Lemma
 \ref{lem:line conic}. 
\end{proof}

Theorem \ref{thm:WX7} has the following  corollaries.

\begin{cor}\label{cor:X8 dec}
For any $d\in H_2(X_7;\Z)$, one has
\[
W_{X_7^+(4),L_{0},L_0}(d,0)\ge W_{X_7^+(4),L_{2},L_2}(d,0)\ge 0.
\]
Moreover both invariant are divisible by $4^{\left[\frac{d\cdot
      [D]}{2}\right]-\min(d\cdot E_i)-1}$,
as well as $W_{X_7^-(4),\R P^2,\R P^2}(d,0)$.
\end{cor}
\begin{proof}
By choosing $E_7$ such that $d\cdot E_7=\min(d\cdot E_i)$, we obtain
 $k\le \min(d\cdot E_i)$, and so 
$d_\Gamma\cdot[D]\ge d\cdot [D] - 2\min(d\cdot E_i)$. The rest of the proof
is similar to that for Corollary \ref{cor:X6 1}.
\end{proof}
The non-negativity of Welschinger invariants of $X_7$ when $s=0$ has been first
established in \cite{IKS13}. 

\begin{cor}\label{cor:X7 van}
For any $d\in H_2(X_7;\Z)$, one has
\[
W_{X_7^+(4),L_{0},\R X_7^+(4)}(d,s)= W_{X_7^+(4),L_{2},\R
  X_7^+(4)}(d,s)=W_{X_7^-(4),\R P^2,\R X_7^-(4)}(d,s).
\]
Moreover this number equals 0 as soon as $r\ge 2$.
\end{cor}
\begin{proof}
The only thing to prove is the last assertion of the corollary. Since
$k_\Gamma^{\circ, \Re}=0$ and $\Gamma$ is a tree, there exists a unique vertex in
$Vert_\Re(\Gamma)$. Now the corollary follows from Theorem \ref{thm:WX7}
and Lemma \ref{cor:vanish}
\end{proof}

\begin{exa}\label{exa:WX7 2c1}
Theorem \ref{thm:WX7} together with Example \ref{exa:X7 2c1} implies
that Welschinger invariants of the real surfaces $X_7^\pm(\kappa)$ are the
one listed in Table \ref{tab:comp WX7}. 
\begin{table}[!h]
\begin{center}
\begin{tabular}{ c|c| c|c|c|c|c|c}
 $s \backslash \kappa $  & $0$& $1$& $2$& $3$& $4-$ &$4+$ & $4+$
\\ &   & & & & & $L=\R P^2_1$ & $L=S^2$
\\\hline
$0$ &  224 & 128  & 64 &24   & 0 & 0 &0
\\\hline
$1$ &  132  & 68&  28& 4& -12& -12 & -12 
\end{tabular}

\vspace{2ex}
$W_{X_7^\pm(\kappa),L,\R X_7^\pm(\kappa)}(2c_1(X_7),s)$
\vspace{3ex}

\begin{tabular}{ccc}
\begin{tabular}{ c|c}
 $s  $  &  
\\\hline
$0$ &   32
\\\hline
$1$ &  12  
\end{tabular}

&\hspace{5ex} &
\begin{tabular}{ c|c|c}
 $s \backslash \epsilon $  &  $0$&$2$
\\\hline
$0$ &  48&  16
\\\hline
$1$ & 20  & 4
\end{tabular}
\\ \\
$W_{X_7^-(4),\R P^2,\R P^2}(2c_1(X_7),s)$&&
$W_{X_7^+(4),L_{\epsilon},L_{\epsilon}}(2c_1(X_7),s)$
\\ & &
\end{tabular}

\end{center}
\caption{Welschinger invariants of $X_7$ for the class $2c_1(X_7)$}
\label{tab:comp WX7}
\end{table}
The invariants $W_{X_7(\kappa)}(2c_1(X_7),s)$ have been first computed 
in \cite{HorSol12}. In addition to the present text, 
the invariants $W_{X_7^+(4),L_\epsilon,L_\epsilon}(2c_1(X_7),0)$ and
$W_{X_7^-(4),\R P^2,\R P^2}(2c_1(X_7),0)$ have also been computed in
\cite{IKS13}.

To get the above sum of multiplicities, I used Theorem \ref{NFD}
and Figure \ref{fig:X7} to compute the  numbers 
$FW^{0,0,0,0}(4[D ]-\sum_{i=1}^8 [E_i],s)$
 listed in Table \ref{tab:comp WXX8}. As in Example \ref{ex:WFD 4 and 6}, we
have $FW^{0,0,0,0}_{\X_8(4)}(4[D]-\sum_{i=1}^8[E_i],s)=
FW^{0,0,0,0}_{\X_8(4),\epsilon}(4[D]-\sum_{i=1}^8[E_i],s)$.
\begin{table}[!h]
\begin{center}
\begin{tabular}{ c|c| c|c|c|c||c|c}
 $s \backslash \kappa,\epsilon $  & $0$& $1$& $2$& $3$& $4$& $0$&$1$
\\\hline
$0$ & 120 & 62 & 28  & 10 & 0 & 32 & 16
\\\hline
$1$ & 136 & 74 & 36 &  14&0  &  16& 8 
\\ 
\end{tabular}
\end{center}
\caption{$FW^{0,0,0,0}_{\X_8(\kappa)}(4[D]-\sum_{i=1}^8[E_i],s)$ and
 and
$FW^{0,0,0,0}_{\X_8(\kappa),\epsilon,\epsilon}(4[D ]-\sum_{i=1}^8 [E_i],s)$ }
\label{tab:comp WXX8}
\end{table}

\end{exa}

\begin{rem}
The invariant 
$W_{(X,c),L,L'}(d,s)$ is said to be \emph{sharp} if there 
 exists a real configuration $\x$ with $s$ pairs of complex conjugated points
such that $|\R \CC(d,0,\x)|=|W_{(X,c),L,L'}(d,s)|$.
It follows from the above computations that  
$W_{X_7^+(4),L_0,\R X_7^+(4)}(2c_1(X_7),1)$ is not sharp. This shows
that {\cite[Theorem 1.1]{Wel4}} does not extend to all real
structures on $X_7$ (see Section \ref{sec:sharpness}). 
The invariant $W_{X_7^-(4),\R P^2,\R X_7^-(4)}(2c_1(X_7),1)$ is sharp,
see Section \ref{sec:sharpness}.
\end{rem}

\section{Absolute  invariants of $X_8$}\label{sec:X8}

\subsection{Strategy}
Let us start with the degeneration $\YY$ of $X_7$ considered in
Section \ref{sec:X7 strategy}. 
Choose an additional generic holomorphic section $p_0':\C\to
\YY$ such that $p_0'(0)\in \X_6\setminus \X_2$,
 and denote by $\ZZ$ the blow up $\YY$ along the divisor $p'_0(\C)$. The map $\pi:\YY\to
 \C$ naturally extends to a flat map
$\pi:\ZZ\to \C$, which is a degeneration of $X_8$ to
 $\pi^{-1}(0)=\X_{8,1}\cup\X_2$. Recall that $\X_{n,1}$ denotes $\C P^2$
 blown-up at $n$ points lying on a conic, and at one additional point
 outside the conic.

Exactly as Gromov-Witten and Welschinger invariants of $X_7$ 
haven been
computed by enumerating curves in $\X_8$, 
Gromov-Witten and Welschinger invariants of $X_8$ 
are
reduced here to  enumeration of curves in $\X_{8,1}$. 
This enumeration is performed in \cite{Shu13} in the case of
complex curves, and in \cite{IKS13} in the case of real curves passing
through a configuration of real points in $CH$ position. The important
properties of such type of configurations are summarized in
Proposition \ref{prop:all real}.

Although I do not see any obstruction to enumerate real and complex
 curves in $\X_{8,1}$ 
using the floor diagrams techniques, I chose not to do it for the sake
 of shortness. I  refer instead to \cite{Shu13,IKS13} for
 details. Hence I compute here Welschinger invariants only for
 configurations of real points.
For the same shortness reason, I decided to restrict to standard real
 structures on $\X_{8,1}$. In particular, with some additional work
 one should be able to generalize Theorem \ref{thm:WX8} to compute
$W_{(X_8,c),L,L'}(d,s)$ for $s>0$ and 
more real
 structures on $X_8$.

\subsection{Gromov-Witten invariants}
I use here notations introduced in Section \ref{sec:X7}, with the
following adjustment in a choice of basis for $H_2(X_{8}; \Z)$ and 
$H_2(\X_{8,1}; \Z)$: 
$E_1,\ldots E_8$ (resp. $\widetilde E_1,\ldots \widetilde E_9$) denote
the $(-1)$-curves coming from the presentation of $X_8$ (resp. $\X_{8,1}$)
as a blow up
of $\C P^2$ (resp. of $\C P^2$ at eight points lying on a conic and
one point outside this conic, $\widetilde E_9$ being the $(-1)$-curve
corresponding to this latter point).
Let us also
denote by
$V_{8,1}\subset H_2(\X_{8,1}; \Z)\setminus\{0\}$ the set of effective classes
$d$ such that 
$d_v\ne l[\widetilde E_i]$ with $l\ge
2$ or $i=7,8$, and $d_v\ne l([D]-[\widetilde E_9])$ with $l\ge 3$.
\begin{defi}
A \emph{$X_8$-graph} is a connected graph $\Gamma$
together with 
 three quantities $d_v\in V_{8,1}$, $g_v\in\Z_{\ge
  0}$, and 
$\beta_v=\beta_{v,1}u_1+\beta_{v,2}u_2\in\Z^\infty_{\ge  0}$
 associated to each vertex $v$ of $\Gamma$ such 
 that $I\beta_v=d_v\cdot [E]$.

An isomorphism between $X_8$-graphs is  an isomorphism of graphs preserving the
 three quantities associated to each vertex.
\end{defi}
An  $X_8$-graphs is always considered up to isomorphism.
Given $g,k\in\Z_{\ge 0}$ and $d\in H_2(X_8;\Z)$ such that $d\cdot [D]\ge 1$,
let $\SS_8(d,g,k)$  be the set of all pairs 
$(\Gamma,P_\Gamma)$ where 
\begin{itemize}
\item $\Gamma$ is a $X_8$-graph  such that 
$$\sum_{v\in Vert(\Gamma)}g_v +b_1(_\Gamma)=g $$
and 
\[\begin{split}
d=& (d_\Gamma\cdot [D] + 2k)[D] -
\sum_{i=1}^6 \left(d_\Gamma\cdot [\widetilde E_i]  + k\right)[E_i]  - 
\left(k^\circ_\Gamma+ \beta_{\Gamma,2} + d_\Gamma\cdot ([\widetilde E_7]+ [\widetilde E_8]
)\right)[E_7]\\
&
- (d_\Gamma\cdot  [\widetilde E_9]) [E_8];
\end{split}\]

\item $P_\Gamma=\bigcup_{v\in Vert(\Gamma)} U_v $ is a partition of the set
  $\{1,\ldots ,c_1(X_8)\cdot d-1+g \}$ such that
$|U_v|=  d_v\cdot [D] -1 +g_v +|\beta_v|$.

\end{itemize}

Given  $(\Gamma,P_\Gamma)\in\SS_8(d,g,k)$ and $v\in
Vert(\Gamma)$, define the complex multiplicities
\[
\mu^\C(v)= \lambda_{v,v} !!
\binom{\beta_{v,1}}{\{\lambda_{v,v'}\}_{v'\in Vert(\Gamma)}}
GW_{\X_{8,1}}^{0,\beta_v}(d_v,g_v),
\]
and
\[
\mu^\C(\Gamma,P_\Gamma)=\frac{I^{\beta_{\Gamma}}}{\sigma(\Gamma)}
\binom{\beta_{\Gamma,1} -2k^\circ_\Gamma} 
{ k^{\circ\circ}}
\prod_{v\ne v'\in Vert(\Gamma)}\lambda_{v,v'}!
\ \prod_{v\in Vert(\Gamma)}  \mu^\C(v).
\]

Next theorem reduces the computation of the numbers $GW_{X_8}(d,g)$ to
the computation of  the numbers $GW_{\X_{8,1}}(d,g)$. 
\begin{thm}\label{thm:GWX8}
Let $g\ge 0$ and $d\in H_2(X_8;\Z)$ such that $d\cdot [D]\ge 1$. Then one
has
$$GW_{X_8}(d,g)=\sum_{k\ge 0}\ \ \sum_{(\Gamma,P_\Gamma)\in\SS_8(d,g,k)}\mu^\C(\Gamma,P_\Gamma) .$$
\end{thm}
\begin{proof}
The proof is word by word the proof of Theorem \ref{thm:GWX7},
using 
 {\cite[Propositions 2.1 and 2.5]{Shu13}} applied to $\X_{8,1}$
instead of Propositions \ref{prop:shoshu} and \ref{prop:initial values}. 
\end{proof}

\begin{exa}
Using Theorem \ref{thm:GWX8} one
computes 
$$GW_{X_8}(2c_1(X_8),0)= 90.$$
Analogous computations
in genus up to $2$ provide the values listed in Table \ref{tab:comp
  X8}. 
The rational case has been first computed in
{\cite[Section 5.2]{PanGot98}}.
\begin{table}[!h]
\begin{center}
\begin{tabular}{ c|c| c|c}
 $g $  & $0$& $1$& $2$
\\\hline
 $GW_{X_8}(2c_1(X_8),g)$   &  90 & 18  & 1 

\end{tabular}

\end{center}

\medskip
\caption{$GW_{X_8}(2c_1(X_8),g)$}
\label{tab:comp X8}
\end{table}
The computation of $GW_{X_8}(2c_1(X_8),0)$ can be detailed as follows.
There exists element(s) 
in $\SS_8(2c_1(X_8),0,k)$ with a positive multiplicity in the following cases:

\begin{tikzpicture}
  [scale=.8,auto=left,vert/.style={circle,fill=blue!20, text
      centered, minimum width=25pt},
leg/.style={circle,fill=white, text centered}]
  \node[vert] (n6) at (0,1) {$v$};
  \node[leg] (n1) at (-1.5,1) {$\Gamma=$};

  \foreach \from/\to in {}
    \draw (\from) -- (\to);
\end{tikzpicture}

\begin{itemize}
\item $k=1, d_v=4[D]-\sum_{1}^8[\widetilde E_i]-2[\widetilde E_9]$: 
$$\mu^\C(\Gamma,P_\Gamma)=70;$$ 

\item $k=2, d_v=2[D]-2[\widetilde E_9]$,  $\beta^\C_{v,2}=2$: 
$$\sum \mu^\C(\Gamma,P_\Gamma)=4;$$ 
\end{itemize}

  \begin{tikzpicture}
  [scale=.8,auto=left,vert/.style={circle,fill=blue!20, text
      centered, minimum width=25pt},
leg/.style={circle,fill=white, text centered}]
  \node[vert] (n6) at (0,1) {$v$};
  \node[vert] (n4) at (3,1)  {$v'$};
  \node[leg] (n1) at (-1.5,1) {$\Gamma=$};

  \foreach \from/\to in {n6/n4}
    \draw (\from) -- (\to);
\end{tikzpicture}

\begin{itemize}
\item $k=2, d_v=[D]-a_7[\widetilde E_7] -a_8[\widetilde E_8]-[\widetilde E_9], d_{v'}=[D]-[\widetilde E_9]$, with $a_7+a_8 =1$: 
$$\sum \mu^\C(\Gamma,P_\Gamma)=4;$$
 
\end{itemize}

  \begin{tikzpicture}
  [scale=.8,auto=left,vert/.style={circle,fill=blue!20, text
      centered, minimum width=25pt},
leg/.style={circle,fill=white, text centered}]
  \node[vert] (n6) at (0,1) {$v'$};
  \node[vert] (n4) at (3,1)  {$v$};
  \node[vert] (n2) at (6,1)  {$v''$};
  \node[leg] (n1) at (-1.5,1) {$\Gamma=$};

  \foreach \from/\to in {n6/n4, n4/n2}
    \draw (\from) -- (\to);
\end{tikzpicture}

\begin{itemize}

\item $k=2, d_v=[D] -[\widetilde E_9], d_{v'}=[\widetilde E_i], d_{v''}=[D]-[\widetilde E_i]-[\widetilde E_9]$, with $1\le i\le 6$: 
$$\sum \mu^\C(\Gamma,P_\Gamma)=12.$$
 \end{itemize}

Only the first above multiplicity is not trivial to compute, and 
 I used 
{\cite[Theorem 2.1]{Shu13}} to
get
$$GW_{\X_{8,1}}^{0,0}(4[D]-\sum_{i=1}^8[\widetilde E_i]-2[\widetilde E_9],0)=70 .$$
\end{exa}

\subsection{Welschinger invariants}\label{sec:WX8}

Denote by $X_8(\kappa)$ with $\kappa= 0,\ldots, 3$,
and $X_8^\pm(4)$ the surface $X_8$ 
equipped with the real structure such that
$$\R X_8(\kappa)=\R P^2_{8-2\kappa},\quad 
\R X_8^-(4)=\R P^2_1\sqcup \R P^2,
\quad \R X_8^+(4)=S^2\sqcup \R P^2_2.$$
These real structures on $X_8$ represent 6  of the 11 deformation
classes of real Del Pezzo surfaces of degree 1. 
Note that
$$\chi(\R X_8^\pm(\kappa))=-7+2\kappa .$$

For $\kappa=0,\ldots,3$,  define  two involutions $\tau^0_\kappa$ and
$\tau^1_\kappa$ on $H_2(\X_{8,1};\Z)$ as follows:
$\tau^0_\kappa$ (resp. $\tau^1_\kappa$) fixes the elements $[D]$,
$[\widetilde E_9]$, and
$[\widetilde E_i]$ with $i\in\{ 2\kappa+1,\ldots,8\}$ (resp. $i\in\{
2\kappa+1,\ldots,6\}$),
and exchanges the elements $[\widetilde E_{2i-1}]$ and $[\widetilde E_{2i}]$ with 
$i\in\{1,\ldots,\kappa\}$ (resp. $i\in\{1,\ldots,\kappa,4\}$).
Denote by $\X_{8,1}(\kappa)$ the surface $\X_{8,1}$ equipped with
the real structure induced by the blowing up of $\C P^2$ at $\kappa$ pairs of
complex conjugated points on a conic $E$, $8-2\kappa$ real points on $\R E$,
and a real point in the exterior of $\R E$.
Given $\epsilon\in\{1,-1\}$,  
 the connected component of 
$\R \X_{8,1}(4)\setminus \R E$ with Euler characteristic $\epsilon$
is denoted $\overline L_\epsilon$.

Let $\kappa=0,\ldots,4$, and let $L$ be a connected component of 
$\R\X_{8,1}(\kappa)\setminus E$.
Given $d\in V_{8,1}\setminus\{2([D]-[\widetilde E_9])\}$
and
 $\x$  a generic configuration
of $  d\cdot  c_1(\X_{8,1}) -1$ points 
in $L$,  denote by $\R\CC_{8,1}(d,\kappa,\x)$ the set of rational
real curves
in $\X_{8,1}(\kappa)$, realizing the class $d$ and containing $\x$. 
Denote by  $\R\CC_{8,1,L}(d,\kappa,\x)$ the subset of
$\R\CC_{8,1}(d,\kappa,\x)$ consisting of curves whose real part,
except maybe their solitary nodes, is contained in $L\cup \R E$. 
Define
 $$W_{\X_{8,1}(\kappa)}(d,\x) =
\sum_{C\in\R\CC_{8,1}(d,\x)}(-1)^{m_{\R \X_{8,1}(\kappa)}(C)} ,\quad \mbox{and}\quad
W_{\X_{8,1}(4),L}(d,\x) =
\sum_{C\in\R\CC_{8,1,L}(d,\x)}(-1)^{m_{\R \X_{8,1}(4)}(C)} .$$

\begin{prop}\label{prop:all real}
Let $\kappa=0,\ldots,4$, and $L$ a connected component of 
$\R\X_{8,1}(\kappa)\setminus E$.
Then for any $\zeta_0\in\Z_{\ge 0}$,
there exists a generic configuration
 $\x$ of $\zeta_0$ points 
in
$L$ with the following property:
for any $d\in V_{8,1}\setminus\{2(D-E_0)\}$
 and any subset $\x'$ of $\x$  such that 
$|\x'|=  d\cdot  c_1(\X_{8,1}) -1$, one has
\begin{itemize}
\item $W_{\X_{8,1}(\kappa)}(d,\x')\ge 0$ and $W_{\X_{8,1}(4),L}(d,\x')\ge 0$;
\item  given any curve $C\in
 \R\CC_{8,1}(d,\kappa,\x')$, all intersection points of $C$ and
 $E$ are transverse and real.
\end{itemize}
Moreover once $d$ and $L$ are chosen, the numbers
$W_{\X_{8,1}(\kappa)}(d,\x')$ and $W_{\X_{8,1}(4),L}(d,\x')$ 
do not depend on the choice  of $\x'$.
\end{prop}
\begin{proof}
The proof is entirely analogous to the proof of {\cite[Theorem 3]{IKS13}}.
\end{proof}
Let us choose once for all $d\in H_2(X_8;\Z)$ such that $[D]\cdot d\ge
1$, and
 a configuration $\x$ of $ d\cdot c_1(X_8) -1$ points in $L$ 
whose existence is attested by Proposition 
\ref{prop:all real}.
In particular when $\x'\subset \x$, it is safe to use the shorter notation
$W_{\X_{8,1}(\kappa)}(d)$ and $W_{\X_{8,1}(4),L}(d)$
instead of $W_{\X_{8,1}(\kappa)}(d,\x')$ and $W_{\X_{8,1}(4),L}(d,\x')$.

Let $\epsilon\in\{0,1\}$ and $\kappa=0,\ldots,3$,
  and denote by $\R\SS_8^\epsilon(d,k,\kappa)$
 the set of couples $(\Gamma,P_\Gamma)\in\SS_8(d,0,k)$ such that
$d_{v}=\tau_\kappa^\epsilon(d_v)$ 
and $\beta_v=\beta_{v,1}u_1$
for any $v\in Vert(\Gamma)$.
Given $(\Gamma,P_\Gamma)\in\R\SS_8^\epsilon(d,k,\kappa)$ and
$v\in Vert(\Gamma)$,  define
$$\mu^{\R,\epsilon}_{\kappa}(v)=
\binom{\beta_{v,1}}{\{\lambda_{v,v'}\}_{v'\in Vert(\Gamma)}}
W_{\X_{8,1}(\kappa+\epsilon)}(d_v^\epsilon), $$
where $d_v^0=d_v$, and $d_v^1$ is obtained from $d_v$ by exchanging the
coefficients of $E_{2\kappa-1}$ and $E_7$, and $E_{2\kappa}$ and
$E_8$,
and
$$\eta^{\R}_{\epsilon}(v)=
\binom{\beta_{v,1}}{\{\lambda_{v,v'}\}_{v'\in Vert(\Gamma)}}
W_{\X_{8,1}(4),\widetilde L_{2\epsilon-1}}(d_v). $$

Given $(\Gamma,P_\Gamma)\in \R\SS^0_{8}(d,k,\kappa)$, 
 define the following multiplicity
$$\mu^{\R,0}_{\kappa}(\Gamma,P_\Gamma)=
\frac{1}{\sigma(\Gamma)}
\binom{\beta_{\Gamma,1} -2k^{\circ}_\Gamma}{k^{\circ\circ}}
\prod_{v\in
  Vert(\Gamma)}\mu^{\R,0}_{\kappa}(v). $$

Let $\R\SS^{1}_{8,m}(d,k,\kappa)$ be the subset of
$\R\SS^{1}_8(d,k,\kappa)$ formed by elements with 
$\beta_{\Gamma}=k^{\circ}u_1$ and $k^{\circ\circ}=0$.
Given $(\Gamma,P_\Gamma)\in \R\SS^{1}_{8,m}(d,k,\kappa)$,
define the following multiplicity
$$\mu^{\R,1}_{\kappa}(\Gamma,P_\Gamma)=
\frac{1}{\sigma(\Gamma)}
\prod_{v\in
  Vert(\Gamma)}\mu^{\R,1}_{s,\kappa}(v)
.$$

Note that $\R\SS^{1}_{8,m}(d,k,4)$ is composed of graphs which are
reduced to a vertex.
As usual, $L_\epsilon$ denotes the connected component of $\R
X_8^\pm(4)$ with Euler characteristic $\epsilon$.

\begin{thm}\label{thm:WX8}
Given  $d\in H_2(X_8;\Z)$ with
$d\cdot [D]\ge 1$, one has
\[\begin{aligned}
&W_{X_8(\kappa)}(d,0)=
\sum_{k\ge 0}\ \ \sum_{(\Gamma,P_\Gamma)\in\R \SS^0_8(d,k,\kappa)}\mu^{\R,0}_{\kappa}(\Gamma,P_\Gamma)
 \quad \mbox{if }\kappa\le 3,\\
&W_{X_8(\kappa+1)}(d,0)=
\sum_{k\ge u0}\ \ \sum_{(\Gamma,P_\Gamma)\in\R \SS^1_{8,m}(d,k,\kappa)}\mu^{\R,1}_{\kappa}(\Gamma,P_\Gamma) 
\quad \mbox{ if }\kappa\le 2,\\
&W_{X_8^-(4),L_\epsilon}(d,0)=W_{X_8^+(4),L_{3\epsilon-1}}(d,0)=
\sum_{k\ge 0}\ \ \sum_{(\Gamma,P_\Gamma)\in\R \SS^1_{8,m}(d,k,3)}\eta^{\R}_{\epsilon}(v)
\quad \forall \epsilon\in\{0,1\}.
\end{aligned}\]
\end{thm}
\begin{proof}
The proof is entirely analogous to the proof of Theorem
\ref{thm:WX7}, while Proposition \ref{prop:all real} ensures that all
involutions $\tau$ are trivial.
\end{proof}

Theorems \ref{thm:GWX8} and \ref{thm:WX8} have the following usual 
corollaries. 
\begin{cor}\label{cor:X8 positive}
For any $d\in H_2(X_8;\Z)$ and $\kappa\in\{0,\ldots, 3\}$, one has
$$W_{X_8(\kappa)}(d,0)\ge 0 \quad\mbox{and}\quad
W_{X_8^\pm(4),L_\epsilon, \R X_8^\pm(4)}(d,0)\ge 0. $$
\end{cor}
\begin{cor}\label{cor:X8 congruence}
For any $d\in H_2(X_8;\Z)$ one has
$$W_{X_8(0)}(d,0)=GW_{X_8}(d,0)\quad \mbox{mod }4. $$
\end{cor}
\begin{proof}
Thanks to Theorems  \ref{thm:GWX8} and \ref{thm:WX8}, we are left 
to prove that there exists $\x$ is as in Proposition \ref{prop:all
  real} such that
$W_{\X_{8,1}(0)}(d')=GW_{X_8}(d',0)\quad \mbox{mod }4$ for all
$d'\in H_2(\X_{8,1};\Z)$ such that $d'\cdot c_1(\X_{8,1})-1\le d\cdot c_1(X_{8})-1$. One can construct 
such a configuration exactly as in {\cite[Theorem 3, see proof
    of Theorem 10]{IKS13}}.
\end{proof}

\begin{cor}\label{cor:X8 van}
For any $d\in H_2(X_8;\Z)$, one has
\[
W_{X_8^\pm(4),L,\R X_8^\pm(4)}(d,s)=0
\]
as soon as $r\ge 2$.
\end{cor}

\begin{exa}\label{exa:WX8 2c1}
Using Theorem \ref{thm:WX8}, one computes the 
 Welschinger invariants of  $X_8^\pm(\kappa)$ listed in Table 
\ref{tab:comp WX8}. The invariants $W_{X_8(\kappa)}(2c_1(X_8),0)$ with
$\kappa\le 3$ have been first computed by Horev and Solomon in \cite{HorSol12}.
\begin{table}[!h]
\begin{center}
\begin{tabular}{ c|c| c|c|c|c|c|c|c}
 $ \kappa $  & $0$& $1$& $2$& $3$& $4-$& $4-$ &$4+$ & $4+$
\\ &   & & & &  $L=\R P^2_1$ & $L=\R P^2$ & $L=\R P^2_2$& $L=S^2$
\\\hline
$W_{X_8^\pm(\kappa),L, \R X_8^\pm(\kappa)}(2c_1(X_8),0)$ &  46 & 30  & 18 &10    & 6 & 6&6 &6
\end{tabular}

\end{center}
\vspace{2ex}

\caption{Welschinger invariants of $X_8$ for the class $2c_1(X_8)$}
\label{tab:comp WX8}
\end{table}
To get the above numbers, I used the method of {\cite[Theorem
    3]{IKS13}} to find configurations $\x$ as in Proposition
\ref{prop:all real}, and obtained the values of 
$W_{\X_{8,1}(\kappa),L}(4[D]-\sum_{i=1}^8[\widetilde E_i]-2[\widetilde E_9])$
 listed in Table 
\ref{tab:comp WXX81}\footnote{In the published version of this paper,
the number $W_{\X_{8,1}(4),\overline{L}_1}(4[D]-\sum_{i=1}^8[\widetilde
  E_i]-2[\widetilde E_9])$ is erroneously claimed to be equal to $4$.}.

\begin{table}[!h]
\begin{center}
\begin{tabular}{ c|c| c|c|c|c|c}
 $ \kappa $  & $0$& $1$& $2$& $3$& $4$& $4$
\\ &   & & & &  $L=\overline L_{-1}$ & $L=\overline L_{1}$
\\\hline
 $FW$ & 30 & 18 & 10  & 6  & 6 & 6
\end{tabular}
\end{center}
\caption{$W_{\X_{8,1}(\kappa),L}(4[D]-\sum_{i=1}^8[\widetilde E_i]-2[\widetilde E_9],\x)$}
\label{tab:comp WXX81}
\end{table}

\end{exa}

\section{Concluding remarks}\label{sec:conclusion}

\subsection{Floor diagrams relative to a conic with empty real part}\label{sec:sharpness}
Recall that  the invariant 
$W_{(X,c),L,L'}(d,s)$ is said to be sharp if there 
 exists a real configuration $\x$ with $s$ pairs of complex conjugated points
such that $|\R \CC(d,0,\x)|=|W_{(X,c),L,L'}(d,s)|$.
When  $r=0$ or $1$, Welschinger proved in \cite{Wel4} the sharpness
of $W_{(X,c),L}(d,s)$ 
 when
  $L$ is homeomorphic to 
either $T^2$, $ S^2$, or $\R P^2$, with the additional
 assumption that $(X,c)=X_{2\kappa}(\kappa)$ with
  $\kappa\le 3$ in the latter case.
In the case of $\C P^2$, one possible way to prove 
this result 
is by
 degenerating $\C P^2$ into
the union of $\C P^2$ and the normal bundle of a 
 real conic with an empty real part. 

 The methods exposed in this paper adapt without any problem to  the case
when $r=0$ or $1$ and $E$ has an
empty real part. In particular,  adaptations of Theorems
\ref{WFD} and \ref{thm:WX7} in this case allow one to extend 
{\cite[Theorem 1.1]{Wel4}} to  the real surface $X_7^-(4)$.

\begin{prop}\label{prop:X7 sign}
Let $d\in H_2(X_7;\Z)$,  $r\in\{0,1\}$, and $s\ge 0$ such that
$c_1(X_7)\cdot d-1=r+2s$. Then $W_{X_7^-(4),\R P^2,\R X_7^-(4)}(d,s)$ is sharp 
and has the same sign as $(-1)^{\frac{d^2-c_1(X_7)\cdot d+2}{2}}$.
\end{prop}

Recently, Koll\'ar proved in \cite{Kol14}  the optimality of
some real
enumerative invariants of 
 projective spaces of any dimension, by
specializing the constraints to a real quadric with an empty real
part. It could be interesting to 
try to generalize Koll\'ar's examples, and to
tackle 
the optimality problem of  the invariants defined in 
\cite{Wel2,GeoZin13}
 via  floor diagrams relative to a quadric in $\C P^n$.

\subsection{Other Welschinger invariants of $X_8$}
Since this paper is already rather long, 
I restricted  in Section
\ref{sec:X8} to the case
$s=0$ and to  standard real structures on $\X_{8,1}$. However
I do not see any obstruction other than technical to
 extend Section
\ref{sec:X8}
to the enumeration of real curves in
$\X_{8,1}$ for arbitrary $r,s$ and any real structure on
$\X_{8,1}$. In particular Theorem \ref{thm:WX8} should  generalize
to  Welschinger invariant of $X_8$ 
for almost all, if not all, real structures.
The standard methods from \cite{IKS11,IKS13,Br6b} should  also apply here
to study logarithmic asymptotic of Welschinger invariants.

The method of this paper should also apply to compute the invariants
recently defined in \cite{Shu14}.

\subsection{Sign of Welschinger invariants}
The sign of Welschinger invariants 
seem to obey to some mysterious rule related
to the topology of the real part of the ambient manifold. The present
work together with
\cite{Wel4}, \cite{IKS3}, \cite{IKS11},  \cite{Br20}, and \cite{BP14} explicit this
rule in a few cases, namely when $L=T^2$, or $S^2$ and $r=0,1$,  when
$X=X_8$ and $s$ is very small, or when $L$ intersects a real Lagrangian
sphere in a single point and $r=1$. 
In the particular case of Del Pezzo surfaces,
floor diagrams relative to a conic, with either empty or non-empty real
part, provide a unified
 way to prove this rule when either $r$ or $s$ is small.
Unfortunately, the rule controlling the sign of Welschinger invariants
in its full generality still remains mysterious.

As an example, I describe how the signs of Welschinger invariants
of $\C P^2$
seem to  behave:
as $r$ goes from $3d-1$ to $0$ or 1, the numbers $W_{\C P^2}(d,s)$ are
first positive, and starting from some mysterious threshold, 
have an alternating sign.
This observation has been made experimentally using Solomon's
real version of WDVV equations \cite{Sol1} for $\C P^2$.

\subsection{Relation with tropical Welschinger invariants and refined
  Severi degrees}
Invariance of Gromov-Witten and Welschinger invariants combined with
Theorems \ref{thm:NFD2}, \ref{thm:W X6}, \ref{thm:GWX7}, \ref{thm:WX7},
\ref{thm:GWX8}, and \ref{thm:WX8} 
provide non-trivial relations among
marked floor diagrams counted with their various multiplicities. It is not
obvious to me how those relations  follow from a purely
combinatorial study of marked floor diagrams. 

Denote by 
$W^{\alpha^\Re,\beta^\Re,\alpha^\Im,\beta^\Im}_{\X_n(\kappa)}(d,g,s,\x)$ 
the straightforward
generalization to any genus of the numbers  
$W^{\alpha^\Re,\beta^\Re,\alpha^\Im,\beta^\Im}_{\X_n(\kappa)}(d,s,\x)$ defined in
Section \ref{sec:real Xn}. 
In the case when $s=0$,  all definitions from Section \ref{sec:real FD}  also
make sense  for positive genus, and Theorem \ref{WFD} still
holds (the proof is exactly the same, see Remarks \ref{rem:s=0} and
\ref{rem:s bis}).
If $\x^\circ$ is a configuration of real points in
$\R\X_n(\kappa)$ as in the proof of Theorems \ref{NFD} and \ref{WFD}, then
one sees easily from the proof of Theorem \ref{WFD} that
   the numbers $W^{0,\beta^\Re_1u_1,0,\beta^\Im_1
  u_1}_{\X_n(\kappa)}(d,g,0,\x^\circ\sqcup\x_E)$  do not depend on the
 position  in each copy of $\N$ of the points in $\x^\circ$.

More surprisingly, the numbers $W^{0,(d\cdot
  E)u_1,0,0}_{\X_n(\kappa)}(d,g,0,\x^\circ\sqcup\x_E)$  I computed on
a few examples,  
with $\x^\circ$ as in the proof of Theorems \ref{NFD} and \ref{WFD},
also satisfy relations analogous to Theorems \ref{thm:W X6}, \ref{thm:WX7},
 and \ref{thm:WX8} for positive genus. 
Furthermore in the case of $X_3$, 
the numbers I obtained in this way are the corresponding 
tropical Welschinger
invariants (see \cite{IKS3} for a definition).
This observation is certainly in favor of the existence
of a more conceptual definition
and signification of those tropical Welschinger invariants. Up to my
knowledge, 
only some tropical Welschinger invariants of the second
Hirzebruch surface yet found such an
interpretation in \cite{Br20,BP14}, where they are shown to correspond to
genuine Welschinger invariants of the quadric ellipsoid.

Tropical Welschinger invariants are also related to \emph{refined
  Severi degrees} \cite{GotShe13,Blo13,BlGo14,IteMik13}. Still in the
case $s=0$, it would have  been
possible to define and compute analogous polynomials 
 interpolating between  real and complex
multiplicities of marked floor
diagrams relative to a conic. Unfortunately, no relations, even conjectural, are
known yet  between refined Severi
degrees and Welschinger invariants when $s>0$. 
Since
 I was interested here in the computation of
those latter for any values of $s$ and $r$, I  chose not to
develop  the refined Severi degree aspect of my computations.

\bibliographystyle {alpha}
\bibliography {../../../Biblio.bib}

\end{document}